\documentclass[10pt,a4paper]{article}

\usepackage{vmargin}
\setmarginsrb{3cm}{1cm}{3cm}{1cm}{1cm}{1cm}{1cm}{1cm}

\usepackage[latin1]{inputenc}
\usepackage[T1]{fontenc}

\usepackage{amsmath,amsthm}
\usepackage{amsfonts}
\usepackage{amssymb}
\usepackage[inline]{asymptote}

\newtheorem{theorem}{Theorem}[section]
\newtheorem{proposition}[theorem]{Proposition}
\newtheorem{lemma}[theorem]{Lemma}
\newtheorem{corollary}[theorem]{Corollary}
\newtheorem{definition}[theorem]{Definition}

\newtheorem{notation}[theorem]{Notation}
\newtheorem{remark}[theorem]{Remark}
\newtheorem{assumption}[theorem]{Assumption}

\title{Uniform ball property and existence of optimal shapes for a wide class of geometric functionals} 
\author{\textsc{Jeremy Dalphin}\footnote{\textit{Institut Elie Cartan de Lorraine UMR CNRS 7502, Universit\'{e} de Lorraine, BP 70239 54506 Vand\oe{}uvre-l\`{e}s-Nancy Cedex, France.} E-mail: \texttt{jeremy.dalphin@mines-nancy.org}}}
\date{}

\begin{document}
\maketitle


\begin{abstract}
In this paper, we are interested in shape optimization problems involving the geometry (normal, curvatures) of the surfaces. We consider a class of hypersurfaces in $\mathbb{R}^{n}$ satisfying a uniform ball condition and we prove the existence of a $C^{1,1}$-regular minimizer for general geometric functionals and constraints involving the first- and second-order properties of surfaces, such as in $\mathbb{R}^{3}$ problems of the form:
\[ \inf \int_{\partial \Omega} j_{0} \left[ \mathbf{x},\mathbf{n}\left(\mathbf{x}\right) \right] dA \left( \mathbf{x} \right) + \int_{\partial \Omega} j_{1} \left[ \mathbf{x},\mathbf{n}\left(\mathbf{x}\right),H\left( \mathbf{x} \right)\right] dA \left( \mathbf{x} \right) + \int_{\partial \Omega} j_{2}\left[\mathbf{x},\mathbf{n}\left(\mathbf{x}\right),K\left(\mathbf{x}\right)\right] dA \left( \mathbf{x}\right),  \] 
where $\mathbf{n}$, $H$, and $K$ respectively denotes the normal, the scalar mean curvature and the Gaussian curvature. We gives some various applications in the modelling of red blood cells such as the Canham-Helfrich energy and the Willmore functional.
\end{abstract}

{\bf Keywords :} shape optimization, uniform ball condition, Helfrich, Willmore, curvature depending energies, geometric functionals.
\medskip

{\bf AMS classification :} primary 49Q10, secondary 53A05, 49J45

\section{Introduction}
In the universe, many physical phenomena are governed by the geometry of their environment. The governing principle is usually modelled by some kind of energy minimization. Some problems such as soap films involve the first-order properties of surfaces (the area, the normal, the first fundamental form), while others such as the equilibrium shapes of red blood cells also concern the second-order ones (the principal curvatures, the second fundamental form). 
\bigskip

In this article, we are interested in the existence of solutions to such shape optimization problems and in the determination of an accurate class of admissible shapes. Indeed, although geometric measure theory \cite{SimonBook} often provides a general framework for understanding these questions precisely, the minimizer usually comes with a poorer regularity than the one expected, and it is very difficult to understand (and to prove) in which sense it is, since singularities can sometimes occur.
\bigskip

Using the shape optimization point of view, the aim of this paper is to introduce a more reasonable class of surfaces, in which the existence of an enough regular minimizer is ensured for general functionals and constraints involving the first- and second-order geometric properties of surfaces. Inspired by what Chenais did in \cite{Chenais1975} when she considered the uniform cone property, we consider the (hyper-)surfaces that satisfy a uniform ball condition in the following sense.

\begin{definition}
\label{definition_epsilon_boule}
Let $\varepsilon > 0$ and $B \subseteq \mathbb{R}^{n}$ be open, $n \geqslant 2$. We say that an open set $ \Omega \subseteq B $ satisfies the $ \varepsilon $-ball condition and we write $\Omega \in \mathcal{O}_{\varepsilon}(B)$ if for any $ \mathbf{x} \in \partial \Omega $, there exits a unit vector $ \mathbf{d_{x}} $ of $\mathbb{R}^{n}$ such that:
\[  \left\lbrace \begin{array}{l}
B_{\varepsilon}(\mathbf{x}-\varepsilon \mathbf{d_{x}} ) \subseteq \Omega \\
\\
 B_{\varepsilon}(\mathbf{x} + \varepsilon \mathbf{d_{x}} ) \subseteq B \backslash \overline{\Omega} , \\
\end{array} \right. \]  
where $ B_{r}(\mathbf{z}) = \lbrace \mathbf{y} \in \mathbb{R}^{n},~ \Vert \mathbf{y} - \mathbf{z} \Vert  < r \rbrace $ denotes the open ball of $ \mathbb{R}^{n} $ centred at $ \mathbf{z} $ and of radius $ r $, where $\overline{\Omega}$ is the closure of $\Omega$, and where $\partial \Omega = \overline{\Omega} \backslash \Omega$ refers to its boundary.
\end{definition}

The uniform (exterior/interior) ball condition was already considered by Poincar\'{e} in 1890 \cite{Poincare}. It avoids the formation of singularities such as corners, cuts, or self-intersections. In fact, it has been known to characterize the $C^{1,1}$-regularity of hypersurfaces for a long time by oral tradition, and also the positiveness of their reach, a notion introduced by Federer in \cite{Federer}. We did not find any reference where these two characterizations were gathered. Hence, for completeness, they are established in Section \ref{section_characterizations_epsilon_boule} and we refer to  Theorems \ref{thm_reach_equiv_boule} and \ref{thm_boule_equiv_cunun} for precise statements.
\bigskip

Equipped with this class of admissible shapes, we can now state our main general existence result in the three-dimensional Euclidean space $\mathbb{R}^{3}$. We refer to Section \ref{section_existence_functional} for its most general form in $\mathbb{R}^{n}$, but the following one is enough for the three physical applications we are presenting hereafter (further examples are also detailed in Section \ref{section_existence_functional}).

\begin{theorem}
\label{thm_existence_rtrois}
Let $\varepsilon > 0$ and $B \subset \mathbb{R}^{3} $ an open ball of radius large enough. Consider $(C, \widetilde{C}) \in \mathbb{R} \times \mathbb{R}$, five continuous maps $j_{0}$, $f_{0}$, $g_{0}$, $g_{1} $, $g_{2} : \mathbb{R}^{3} \times \mathbb{S}^{2} \rightarrow \mathbb{R}$, and four maps $j_{1} $, $j_{2}$, $f_{1}$, $f_{2}: \mathbb{R}^{3} \times \mathbb{S}^{2} \times \mathbb{R} \rightarrow \mathbb{R}$ which are continuous and convex in the last variable. Then, the following problem has at least one solution (see Notation \ref{notation_geometrie_rtrois}):
\[ \inf \int_{\partial \Omega} j_{0} \left[ \mathbf{x},\mathbf{n}\left(\mathbf{x}\right) \right] dA \left( \mathbf{x} \right) + \int_{\partial \Omega} j_{1} \left[ \mathbf{x},\mathbf{n}\left(\mathbf{x}\right),H\left( \mathbf{x} \right)\right] dA \left( \mathbf{x} \right) + \int_{\partial \Omega} j_{2}\left[\mathbf{x},\mathbf{n}\left(\mathbf{x}\right),K\left(\mathbf{x}\right)\right] dA \left( \mathbf{x}\right),  \] 
where the infimum is taken among any $\Omega \in \mathcal{O}_{\varepsilon}(B)$ satisfying a finite number of constraints of the following form:
\[ \left\lbrace \begin{array}{l} 
 \displaystyle{\int_{\partial \Omega} f_{0}\left[\mathbf{x},\mathbf{n}\left( \mathbf{x} \right) \right] dA \left( \mathbf{x} \right) + \int_{\partial \Omega} f_{1} \left[ \mathbf{x},\mathbf{n}\left(\mathbf{x}\right),H\left(\mathbf{x}\right)\right] dA \left( \mathbf{x}\right) + \int_{\partial \Omega} f_{2}\left[\mathbf{x},\mathbf{n}\left(\mathbf{x}\right),K\left(\mathbf{x}\right) \right] dA \left( \mathbf{x}\right) \leqslant C } \\
\\
\displaystyle{\int_{\partial \Omega} g_{0}\left[\mathbf{x},\mathbf{n}\left( \mathbf{x} \right) \right] dA \left( \mathbf{x} \right) + \int_{\partial \Omega} H \left(\mathbf{x} \right) g_{1} \left[\mathbf{x},\mathbf{n}\left(\mathbf{x}\right) \right] dA \left( \mathbf{x} \right) + \int_{\partial \Omega}  K \left( \mathbf{x} \right) g_{2}\left[ \mathbf{x},\mathbf{n}\left(\mathbf{x}\right) \right] dA \left( \mathbf{x} \right) = \widetilde{C}. } \\
\end{array} \right. \]
\end{theorem}

The proof of Theorem \ref{thm_existence_rtrois} only relies on basic tools of analysis and does not use the ones of geometric measure theory. We also mention that the particular case $j_{0} \geqslant 0$ and $j_{1} = j_{2} = 0$ without constraints was obtained in parallel to our work in \cite{GuoYang2013}.

\begin{notation}
\label{notation_geometrie_rtrois}
We denote by $A(.)$ (respectively $V(.)$) the area (resp. the volume) i.e. the two(resp. three)-dimensional Hausdorff measure, and the integration on a surface is done with respect to $A$. The Gauss map $\mathbf{n} : \mathbf{x} \mapsto \mathbf{n} (\mathbf{x} ) \in \mathbb{S}^{2} $ always refers to the unit outer normal field of the surface, while $H = \kappa_{1}+ \kappa_{2}$ is the scalar mean curvature and $K = \kappa_{1} \kappa_{2}$ is the Gaussian curvature.
\end{notation}

\begin{remark}
In the above theorem, the radius of $B$ is large enough to avoid $\mathcal{O}_{\varepsilon}(B)$ being empty. Moreover, the assumptions on $B$ can be relaxed by requiring $B$ to be a non-empty bounded open set, smooth enough (Lipschitz for example) such that its boundary has zero three-dimensional Lebesgue measure, and large enough to contain at least an open ball of radius $3 \varepsilon$.
\end{remark} 

\subsection{First application: minimizing the Canham-Helfrich energy with area and volume constraints}
\label{section_canham_helfrich}
In biology, when a sufficiently large amount of phospholipids is inserted in a aqueous media, they immediately gather in pairs to form bilayers also called vesicles. Devoid of nucleus among mammals, red blood cells are typical examples of vesicles on which is fixed a network of proteins playing the role of a skeleton inside the membrane \cite{Pandit}. In the 70s, Canham \cite{Canham} then Helfrich \cite{Helfrich} suggested a simple model to characterize vesicles. Imposing the area of the bilayer and the volume of fluid it contains, their shape is a minimizer for the following free-bending energy (see Notation \ref{notation_geometrie_rtrois}):
\begin{equation}
\label{expression_canham_helfrich_energy}
\mathcal{E} = \dfrac{k_{b}}{2} \int_{\mathrm{membrane}} \left( H  - H_{0} \right)^{2} dA + k_{G} \int_{\mathrm{membrane}} K dA  , 
\end{equation} 
where $H_{0} \in \mathbb{R}$ (called the spontaneous curvature) measures the asymmetry between the two layers, and where $k_{b} > 0 $, $k_{G} < 0 $ are two other physical constants. Note that if $k_{G} > 0$, for any $ k_{b}, H_{0} \in \mathbb{R}$, the Canham-Helfrich energy \eqref{expression_canham_helfrich_energy} with prescribed area $A_{0}$ and volume $V_{0}$ is not bounded from below. Indeed, in that case, from the Gauss-Bonnet Theorem, the second term tends to $- \infty$ as the genus $g \rightarrow + \infty $, while the first term remains bounded by $ 4 \vert k_{b} \vert (12 \pi + \frac{1}{4} H_{0}^{2} A_{0} )$ (to see this last point, combine \cite[Remark 1.7 (iii) (1.5)]{KellerMondinoRiviere}, \cite[Theorem 1.1]{Schygulla}, and \cite[Inequality (0.2)]{Simon}).
\bigskip

The two-dimensional case of \eqref{expression_canham_helfrich_energy} is considered by Bellettini, Dal Maso, and Paolini in \cite{BellettiniMasoPaolini}. Some of their results is recovered by Delladio \cite{Delladio} in the framework of special generalized Gauss graphs from the theory of currents. Then, Choksi and Veneroni \cite{ChoksiVeneroni} solve the axisymmetric situation of \eqref{expression_canham_helfrich_energy} in $\mathbb{R}^{3}$ assuming $-2k_{b} < k_{G} < 0$. In the general case, this hypothesis gives a fundamental coercivity property \cite[Lemma 2.1]{ChoksiVeneroni} and the integrand of \eqref{expression_canham_helfrich_energy} is standard in the sense of \cite[Definition 4.1.2]{Hutchinson}. Hence, we get a minimizer for \eqref{expression_canham_helfrich_energy} in the class of rectifiable integer oriented $2$-varifold in $\mathbb{R}^{3}$ with $L^{2}$-bounded generalized second fundamental form \cite[Theorem 5.3.2]{Hutchinson} \cite[Section 2]{Mondino} \cite[Appendix]{BellettiniMugnai}. These compactness and lower semi-continuity properties were already noticed in \cite[Section 9.3]{BellettiniMugnai}.
\bigskip

However, the regularity of minimizers remains an open problem and experiments show that singular behaviours can occur to vesicles such as the budding effect \cite{SeifertBerndlLipowsky,Seifert}. This cannot happen to red blood cells because their skeleton prevents the membrane from bending too much locally \cite[Section 2.1]{Pandit}. To take this aspect into account, the uniform ball condition of Definition \ref{definition_epsilon_boule} is also motivated by the modelization of the equilibrium shapes of red blood cells. We even have a clue for its physical value \cite[Section 2.1.5]{Pandit}. Our result states as follows.

\begin{theorem}
\label{thm_existence_canham_helfrich}
Let $H_{0},k_{G} \in \mathbb{R}$ and $\varepsilon, k_{b}, A_{0}, V_{0} > 0$ such that $A_{0}^{3} > 36 \pi V_{0}^{2}$. Then, the following problem has at least one solution (see Notation \ref{notation_geometrie_rtrois}):
\[ \inf_{\substack{\Omega \in \mathcal{O}_{\varepsilon}(\mathbb{R}^{n}) \\ A(\partial \Omega) = A_{0} \\ V(\Omega) = V_{0} }} \dfrac{k_{b}}{2} \int_{\partial \Omega} (H-H_{0})^{2}dA + k_{G} \int_{\partial \Omega} K dA.  \]
\end{theorem}   

\begin{remark}
\label{remarque_isoperimetrique}
From the isoperimetric inequality, if $A_{0}^{3} < 36 \pi V_{0}^{2}$, one cannot find any $\Omega \in \mathcal{O}_{\varepsilon}(\mathbb{R}^{n})$ satisfying the two constraints; and if equality holds, the only admissible shape is the ball of area $A_{0}$ and volume $V_{0}$. Moreover, in the above theorem, note that we did not assume the $\Omega \in \mathcal{O}_{\varepsilon}(B)$ as it is the case for Theorem \ref{thm_existence_rtrois} because a uniform bound on their diameter is already given by the functional and the area constraint \cite[Lemma 1.1]{Simon}. Finally, the result above also holds if $H_{0}$ is continuous function of the position and the normal.    
\end{remark}

\subsection{Second application: minimizing the Canham-Helfrich energy with prescribed genus, area, and volume}
\label{section_helfrich}
Since the Gauss-Bonnet Theorem is valid for sets of positive reach \cite[Theorem 5.19]{Federer}, we get from Theorem \ref{thm_reach_equiv_boule} that $\int_{\Sigma} K dA = 4 \pi (1-g)$ for any compact connected $C^{1,1}$-surface $\Sigma$ (without boundary embedded in $\mathbb{R}^{3}$) of genus $g \in \mathbb{N}$. Hence, instead of minimizing \eqref{expression_canham_helfrich_energy}, people usually fix the topology and search for a minimizer of the following energy (see Notation \ref{notation_geometrie_rtrois}):
\begin{equation}
\label{expression_helfrich_energy}
\mathcal{H}(\Sigma) = \int_{\Sigma} \left( H-H_{0} \right)^{2} dA,
\end{equation}
with prescribed area and enclosed volume. Like \eqref{expression_canham_helfrich_energy}, such a functional depends on the surface but also on its orientation. However, in the case $H_{0} \neq 0$, Energy \eqref{expression_helfrich_energy} is not even lower semi-continuous with respect to the varifold convergence \cite[Section 9.3]{BellettiniMugnai}: the counterexample is due to Gro{\ss}e-Brauckmann \cite{Brauckmann}. In this case, we cannot directly use the tools of geometric measure theory but we can prove the following result. 

\begin{theorem}
\label{thm_existence_helfrich}
Let $H_{0} \in \mathbb{R}$, $g \in \mathbb{N}$, and $\varepsilon, A_{0}, V_{0} > 0$ such that $A_{0}^{3} > 36 \pi V_{0}^{2}$. Then, the following problem has at least one solution (see Notation \ref{notation_geometrie_rtrois} and Remark \ref{remarque_isoperimetrique}):
\[ \inf_{\substack{\Omega \in \mathcal{O}_{\varepsilon}(\mathbb{R}^{n}) \\ \mathrm{genus}(\partial \Omega) = g \\ A(\partial \Omega) = A_{0} \\ V(\Omega) = V_{0} }} \int_{\partial \Omega} (H-H_{0})^{2}dA,  \]
where $\mathrm{genus}(\partial \Omega) = g$ has to be understood as $\partial \Omega$ is a compact connected $C^{1,1}$-surface of genus $g$.
\end{theorem}

\subsection{Third application: minimizing the Willmore functional with various constraints}
\label{section_willmore}
The particular case $H_{0} = 0$ in \eqref{expression_helfrich_energy} is known as the Willmore functional (see Notation \ref{notation_geometrie_rtrois}):
\begin{equation}
\label{expression_willmore_energy}
\mathcal{W}(\Sigma) = \dfrac{1}{4} \int_{\Sigma} H^{2} dA.
\end{equation} 
It has been widely studied by geometers. Without constraint, Willmore \cite[Theorem 7.2.2]{Willmore} proved that spheres are the only global minimizers of \eqref{expression_willmore_energy}. The existence was established by Simon \cite{Simon} for genus-one surfaces, Bauer and Kuwert \cite{BauerKuwert} for higher genus. Recently, Marques and Neves \cite{MarquesNeves} solved the so-called Willmore conjecture: the conformal transformations of the stereographic projection of the Clifford torus are the only global minimizers of \eqref{expression_willmore_energy} among smooth genus-one surfaces.
\bigskip

A main ingredient is the conformal invariance of \eqref{expression_willmore_energy}, from which we can in particular deduce that minimizing \eqref{expression_willmore_energy} with prescribed isoperimetric ratio is equivalent to impose the area and the enclosed volume. In this direction, Schygulla \cite{Schygulla} established the existence of a minimizer for \eqref{expression_willmore_energy} among analytic surfaces of zero genus and given isoperimetric ratio. For higher genus, Keller, Mondino, and Riviere \cite{KellerMondinoRiviere} recently obtained similar results, using the point of view of immersions developed by Riviere \cite{Riviere} to characterize precisely the critical points of \eqref{expression_willmore_energy}.
\bigskip

An existence result related to \eqref{expression_willmore_energy} is the particular case $H_{0} = 0$ of Theorem \ref{thm_existence_helfrich}. Again, the difficulty with these kind of functionals is not to obtain a minimizer (compactness and lower semi-continuity in the class of varifolds for example) but to show that it is regular in the usual sense (i.e. a smooth surface). We now give a last application of Theorem \ref{thm_existence_rtrois} which comes from the modelling of vesicles. It is known as the bilayer-couple model \cite[Section 2.5.3]{Seifert} and it states as follows.

\begin{theorem}
\label{thm_existence_willmore}
Let $M_{0} \in \mathbb{R}$ and $\varepsilon, A_{0}, V_{0} > 0$ such that $A_{0}^{3} > 36 \pi V_{0}^{2}$. Then, the following problem has at least one solution (see Notation \ref{notation_geometrie_rtrois} and Remark \ref{remarque_isoperimetrique}):
\[ \inf_{\substack{\Omega \in \mathcal{O}_{\varepsilon}(\mathbb{R}^{n}) \\ \mathrm{genus}(\partial \Omega) = g \\ A(\partial \Omega) = A_{0} \\ V(\Omega) = V_{0} \\ \int_{\partial \Omega} H dA = M_{0} }} \dfrac{1}{4}\int_{\partial \Omega} H^{2}dA,  \]
where $\mathrm{genus}(\partial \Omega) = g$ has to be understood as $\partial \Omega$ is a compact connected $C^{1,1}$-surface of genus $g$.
\end{theorem}

To conclude this introduction, the paper is organized as follows. In Section \ref{section_characterizations_epsilon_boule}, we establish precise statements of the two characterizations associated with the uniform ball condition, namely Theorem \ref{thm_reach_equiv_boule} in terms of positive reach and Theorem \ref{thm_boule_equiv_cunun} in terms of $C^{1,1}$-regularity. The proofs were already given in \cite{Dalphin} and are shortly reproduced here for completeness. 
\bigskip

Following the classical method from the calculus of variations, in Section \ref{section_compacite}, we first obtain the compactness of the class $\mathcal{O}_{\varepsilon}(B)$ for various modes of convergence. This essentially follows from the fact that the $\varepsilon$-ball condition implies a uniform cone property, for which we already have some good compactness results. 
\bigskip

Then, in the rest of Section \ref{section_parametrization_i}, we prove the key ingredient of Theorem \ref{thm_existence_rtrois}: we manage to parametrize in a fixed local frame simultaneously all the graphs associated with the boundaries of a converging sequence in $\mathcal{O}_{\varepsilon}(B)$ and to prove the $C^{1}$-strong convergence and the $W^{2,\infty}$-weak-star convergence of these local graphs. 
\bigskip

Finally, in Section \ref{section_continuite}, we show how to use this local result on a suitable partition of unity to get the global continuity of general geometric functionals. We conclude by giving some existence results in Section \ref{section_existence_functional}. We prove Theorem \ref{thm_existence_rtrois}, its generalization to $\mathbb{R}^{n}$, and detail many applications such as Theorems \ref{thm_existence_canham_helfrich}, \ref{thm_existence_helfrich}, and \ref{thm_existence_willmore}, mainly coming from the modelling of vesicles and red blood cells.

\section{Two characterizations of the uniform ball property}
\label{section_characterizations_epsilon_boule}
In this section, we establish two characterizations of the $\varepsilon$-ball condition, namely Theorems \ref{thm_reach_equiv_boule} and \ref{thm_boule_equiv_cunun}. First, we show that this property is equivalent to the notion of positive reach introduced by Federer \cite{Federer}. Then, we prove that it is equivalent to a uniform $C^{1,1}$-regularity of hypersurfaces. These are known facts. The proofs, already given in \cite{Dalphin}, are reproduced here for completeness. 
\bigskip

Indeed, we did not find any reference where these two characterizations were gathered although many parts of Theorems \ref{thm_reach_equiv_boule} and \ref{thm_boule_equiv_cunun} can be found in the literature as remarks \cite[below Theorem 1.4]{Hormander} \cite[(1.10)]{Mitrea} \cite[Remark 4.20]{Federer}, sometimes with proofs \cite[Section 2.1]{Fu} \cite[Theorem 2.2]{GuoYang2012} \cite[\S 4 Theorem 1]{Lucas} \cite[Proposition 1.4]{Lytchak}, or as consequences of results \cite[Theorem 1.2]{GhomiHoward} \cite[Theorem 1.1 (1.2)]{AlvaroBrighamMazyaMitreaZiade}. 

\subsection{Definitions, notation, and statements}
Before stating the theorems, we recall some definitions and notation, used thereafter in the article. Consider any integer $ n \geqslant 2 $ henceforth set. The space $ \mathbb{R}^{n} $ whose points are marked $ \mathbf{x} = (x_{1}, \ldots, x_{n}) $ is naturally provided with its usual Euclidean structure, $\langle \mathbf{x} ~\vert~ \mathbf{y} \rangle = 
\sum_{k=1}^{n} x_{k}y_{k}$ and $ \Vert \mathbf{x} \Vert = \sqrt{\langle \mathbf{x} ~\vert~ \mathbf{x} \rangle} $, but also with a direct orthonormal frame whose choice will be specified later on. Inside this frame, every point $ \mathbf{x}$ of $ \mathbb{R}^{n} $ will be written into the form $ (\mathbf{x'}, x_{n}) $ such that $ \mathbf{x'} = (x_{1}, \ldots, x_{n-1}) \in \mathbb{R}^{n-1}$. In particular, the symbols $ \mathbf{0} $ and $ \mathbf{0'} $ respectively refer to the zero vector of $ \mathbb{R}^{n} $ and $ \mathbb{R}^{n-1} $. 
\bigskip

First, some of the notation introduced in \cite{Federer} by Federer are recalled. For every non-empty subset $A $ of $ \mathbb{R}^{n}$, the following map is well defined and $1$-Lipschitz continuous:
\[ \begin{array}{rrcl}
d(.,A):& \mathbb{R}^{n} &\longrightarrow & [0,+\infty[ \\
 & \mathbf{x} & \longmapsto & \displaystyle{ d( \mathbf{x},A ) = \inf_{\mathbf{a} \in A} \Vert \mathbf{x} -\mathbf{a} \Vert.} \\  
\end{array} \]
Furthermore, we introduce: 
\[ \mathrm{Unp}(A) = \lbrace \mathbf{x} \in \mathbb{R}^{n} ~\vert~ \exists ! \mathbf{a} \in A, \quad \Vert \mathbf{x} - \mathbf{a} \Vert  = d(\mathbf{x},A) \rbrace. \] 
This is the set of points in $ \mathbb{R}^{n}$ having a unique projection on $A$, that is the maximal domain on which the following map is well defined:
\[ \begin{array}{rrcl}
p_{A} : & \mathrm{Unp}(A) & \longrightarrow & A \\
 & \mathbf{x} & \longmapsto & p_{A}(\mathbf{x}) , \\ 
\end{array} \]
where $p_{A}(\mathbf{x})$ is the unique point of $ A$ such that $\Vert p_{A}(\mathbf{x}) - \mathbf{x} \Vert = d(\mathbf{x},A)$. We can also notice that $A \subseteq \mathrm{Unp}(A)$ thus in particular $\mathrm{Unp}(A) \neq \emptyset$. We can now express what is a set of positive reach.

\begin{definition}
\label{definition_reach}
Consider any non-empty subset $A$ of $\mathbb{R}^{n}$. First,  we set for any point $\mathbf{a} \in A$:
\[  \mathrm{Reach}(A,\mathbf{a}) = \sup \left\lbrace r > 0, \quad B_{r}(\mathbf{a}) \subseteq \mathrm{Unp}(A) \right\rbrace ,  \]
with the convention $\sup \emptyset = 0$. Then, we define the reach of $A$ by the following quantity: 
\[ \mathrm{Reach}(A) = \inf_{\mathbf{a} \in A} \mathrm{Reach}(A,\mathbf{a}). \] 
Finally, we say that $A$ has a positive reach if we have $\mathrm{Reach}(A) > 0$.
\end{definition}

\begin{remark}
\label{remarque_reach_bord}
From Definition \ref{definition_reach}, the reach of a subset of $\mathbb{R}^{n}$ is defined if it is not empty. Consequently, when considering the reach associated with the boundary of an open subset $\Omega$ of $\mathbb{R}^{n}$, we will have to ensure $\partial \Omega \neq \emptyset$ and to do so, we will assume $\Omega$ is not empty and different from $\mathbb{R}^{n}$. Indeed, if $\partial \Omega = \emptyset$, then $\overline{\Omega} = \Omega \cup \partial \Omega = \Omega$ thus $ \Omega =\emptyset$ or $ \Omega = \mathbb{R}^{n}$ because it is both open and closed.
\end{remark}

Then, we also recall the definition of a $C^{1,1}$-hypersurface in terms of local graph. Note that from the Jordan-Brouwer Separation Theorem, any compact topological hypersurface of $\mathbb{R}^{n}$ has a well-defined inner domain, and in particular a well-defined enclosed volume. If instead of being compact, it is connected and closed as a subset of $\mathbb{R}^{n}$, then it remains the boundary of an open set \cite[Theorem 4.16]{MontielRos} \cite[Section 8.15]{Dold}, which is not unique and possibly unbounded in this case.

\begin{definition}
\label{definition_regularite_cunun}
Consider any subset $\mathcal{S}$ of $\mathbb{R}^{n}$. We say that $\mathcal{S}$ is a $C^{1,1}$-hypersurface if there exists an open subset $\Omega$ of $\mathbb{R}^{n}$ such that $\partial \Omega = \mathcal{S}$, and such that for any point $ \mathbf{x}_{0} \in \partial \Omega $, there exists a direct orthonormal frame centred at $\mathbf{x}_{0}$ such that in this local frame, there exists a map $\varphi: D_{r}(\mathbf{0'}) \rightarrow ]-a,a[$ continuously differentiable  with $a > 0$, such that $\varphi$ and its gradient $\nabla \varphi$ are $L$-Lipschitz continuous with $L > 0$, satisfying $\varphi(\mathbf{0}') = 0$, $\nabla \varphi (\mathbf{0}') = \mathbf{0'}$, and also:
\[\left\lbrace \begin{array}{rcl}
\partial \Omega \cap \left(  D_{r} \left(\mathbf{0'}\right) \times ]- a, a [ \right) & = & \left\lbrace \left(\mathbf{x'},\varphi(\mathbf{x'})\right), \quad \mathbf{x'} \in D_{r}(\mathbf{0'}) \right\rbrace \\
& &   \\
\Omega \cap  \left( D_{r} \left(\mathbf{0'}\right) \times   ]- a, a [ \right)  & = & \left\lbrace (\mathbf{x'},x_{n}), \quad \mathbf{x'} \in D_{r}(\mathbf{0'}) ~ \mathrm{and} ~
-a < x_{n} <\varphi(\mathbf{x'}) \right\rbrace , \\
\end{array} \right. \]
where $D_{r}(\mathbf{0'}) = \lbrace \mathbf{x'} \in \mathbb{R}^{n-1}, ~ \Vert \mathbf{x'} \Vert < r \rbrace$ denotes the open ball of $\mathbb{R}^{n-1}$ centred at the origin $\mathbf{0'}$ and of radius $r > 0$.
\end{definition}

Finally, we recall the definition of the uniform cone property introduced by Chenais in \cite{Chenais1975}, and from which the $\varepsilon$-ball condition is inspired. We also refer to \cite[Definition 2.4.1]{HenrotPierre}.
 
\begin{definition}
\label{definition_alpha_cone}
Let $\alpha \in ]0,\frac{\pi}{2}[$ and $\Omega$ be an open subset of $\mathbb{R}^{n}$. We say that $\Omega$ satisfies the $\alpha$-cone condition if for any point $\mathbf{x} \in \partial \Omega$, there exists a unit vector $\mathbf{\xi_{x}}$ of $\mathbb{R}^{n}$ such that:
\[\forall \mathbf{y} \in B_{\alpha}(\mathbf{x}) \cap \Omega, \quad C_{\alpha}(\mathbf{y},\mathbf{\xi_{x}}) \subseteq \Omega ,  \]
where $C_{\alpha}(\mathbf{y}, \mathbf{\xi_{x}}) = \lbrace \mathbf{z} \in B_{\alpha}(\mathbf{y}), ~ \Vert \mathbf{z} - \mathbf{y} \Vert \cos \alpha < \langle  \mathbf{z} - \mathbf{y} ~ \vert ~ \mathbf{\xi_{x}} \rangle \rbrace $ refers to the open cone of corner $\mathbf{y}$, direction $\mathbf{\xi_{x}}$, and span $\alpha$. 
\end{definition}

We are now in position to precisely state the two main regularity results associated with the uniform ball condition.
\begin{theorem}[\textbf{A characterization in terms of positive reach}]
\label{thm_reach_equiv_boule}
Consider any non-empty open subset $\Omega$ of $\mathbb{R}^{n}$ different from $\mathbb{R}^{n}$. Then, the following implications are true:
\begin{itemize}
\item[(i)] if there exists $\varepsilon > 0$ such that $\Omega \in \mathcal{O}_{\varepsilon}(\mathbb{R}^{n})$ as in Definition \ref{definition_epsilon_boule}, then $\partial \Omega$ has a positive reach in the sense of Definition \ref{definition_reach} and we have $\mathrm{Reach}(\partial \Omega) \geqslant \varepsilon $;
\item[(ii)] if $\partial \Omega$ has a positive reach, then $\Omega \in \mathcal{O}_{\varepsilon}(\mathbb{R}^{n})$ for any $\varepsilon \in ]0,\mathrm{Reach}(\partial \Omega)[ $, and moreover, if $\partial \Omega$ has a finite positive reach, then $\Omega$ also satisfies the $\mathrm{Reach}(\partial \Omega)$-ball condition.
\end{itemize}
In other words, we have the following characterization:
\[ \mathrm{Reach}(\partial \Omega) = \sup \left\lbrace \varepsilon > 0 , \quad \Omega \in \mathcal{O}_{\varepsilon}(\mathbb{R}^{n}) \right\rbrace , \]
with the convention $\sup \emptyset = 0$. Moreover, this supremum becomes a maximum if it is not zero and finite. Finally, we get $\mathrm{Reach}(\partial \Omega) = + \infty$ if and only if $\partial \Omega$ is an affine hyperplane of $\mathbb{R}^{n}$. 
\end{theorem}

\begin{theorem}[\textbf{A characterization in terms of $C^{1,1}$-regularity}]
\label{thm_boule_equiv_cunun}
Let $\Omega$ be a non-empty open subset of $\mathbb{R}^{n}$ different from $\mathbb{R}^{n}$. If there exists $\varepsilon > 0$ such that $\Omega \in \mathcal{O}_{\varepsilon}(\mathbb{R}^{n})$, then its boundary $\partial \Omega$ is a $C^{1,1}$-hypersurface of $\mathbb{R}^{n}$ in the sense of Definition \ref{definition_regularite_cunun}, where $a = \varepsilon$ and the constants $L$, $r$ depend only on $\varepsilon$. Moreover, we have the following properties:
\begin{itemize}
\item[(i)] $ \Omega$ satisfies the $f^{-1}(\varepsilon)$-cone property as in Definition \ref{definition_alpha_cone} with $f: \alpha \in ]0,\frac{\pi}{2}[ \mapsto \frac{2 \alpha}{\cos \alpha} \in ]0, + \infty[ $;
\item[(ii)] the vector $\mathbf{d_{x}}$ of Definition \ref{definition_epsilon_boule} is the unit outer normal to the hypersurface at the point $\mathbf{x}$;
\item[(iii)] the Gauss map $\mathbf{d}: \mathbf{x} \in \partial \Omega \mapsto \mathbf{d_{x}} \in \mathbb{S}^{n-1}$ is well defined and $\frac{1}{\varepsilon}$-Lipschitz continuous.
\end{itemize}
Conversely, if $\mathcal{S}$ is a non-empty compact $C^{1,1}$-hypersurface of $\mathbb{R}^{n}$ in the sense of Definition \ref{definition_regularite_cunun}, then there exists $\varepsilon > 0$ such that its inner domain $\Omega \in \mathcal{O}_{\varepsilon}(\mathbb{R}^{n})$. In particular, it has a positive reach with $\mathrm{Reach}(\mathcal{S}) = \max \left\lbrace \varepsilon > 0, ~ \Omega \in \mathcal{O}_{\varepsilon}(\mathbb{R}^{n}) \right\rbrace $.
\end{theorem}

\begin{remark}
In the above assertion, note that $a$, $L$, and $r$ only depend on $\varepsilon$ for any point of the hypersurface. This uniform dependence of the $C^{1,1}$-regularity characterizes the class $\mathcal{O}_{\varepsilon}(\mathbb{R}^{n})$. Indeed, the converse part of Theorem \ref{thm_boule_equiv_cunun} also holds if instead of being compact, the non-empty $C^{1,1}$-hypersurface $\mathcal{S}$ satisfies: $ \exists \varepsilon > 0, \forall \mathbf{x}_{0} \in \mathcal{S}, \min(\frac{1}{L},\frac{r}{3},\frac{a}{3}) \geqslant \varepsilon$. In this case, we still have $\Omega \in \mathcal{O}_{\varepsilon}(\mathbb{R}^{n})$ where $\Omega$ is the open set of Definition \ref{definition_regularite_cunun} such that $\partial \Omega = \mathcal{S}$. 
\end{remark}

\begin{remark}
\label{remarque_borne_linfty}
From Point (iii) of Theorem \ref{thm_boule_equiv_cunun}, the Gauss map $\mathbf{d}$ is $\frac{1}{\varepsilon}$-Lipschitz continuous. Hence, it is differentiable almost everywhere and $\Vert D_{\bullet} \mathbf{d} \Vert_{L^{\infty}(\partial \Omega)} \leqslant \frac{1}{\varepsilon}$ \cite[Section 5.2.2]{HenrotPierre}. In particular, the principal curvatures (see Section \ref{section_geometrie} for definitions and \eqref{eqn_kappa} for details) satisfy $\Vert \kappa_{l} \Vert_{L^{\infty}(\partial \Omega)} \leqslant \frac{1}{\varepsilon}$.
\end{remark}

\subsection{The sets of positive reach and the uniform ball condition}
Throughout this section, $\Omega$ refers to any non-empty open subset of $\mathbb{R}^{n}$ different from $\mathbb{R}^{n}$. Hence, its boundary $\partial \Omega$ is not empty and $\mathrm{Reach}(\partial \Omega)$ is well defined (cf. Remark \ref{remarque_reach_bord}). First, we establish some properties that were mentioned in Federer's paper \cite{Federer} and then, we show Theorem \ref{thm_reach_equiv_boule} holds.

\subsubsection{Positive reach implies uniform ball condition}
\begin{lemma}
\label{lemme_reach}
For any point $ \mathbf{x} \in \partial \Omega$, we have: $ \mathrm{Reach} (\partial \Omega, \mathbf{x}) = \min \left( \mathrm{Reach}(\overline{\Omega}, \mathbf{x}), \mathrm{Reach}(\mathbb{R}^{n} \backslash \Omega, \mathbf{x} ) \right) $.
\end{lemma} 

\begin{proof}
We only sketch the proof. Observe $  d(\mathbf{x},\partial \Omega) = \max ( d(\mathbf{x},\overline{\Omega}),d(\mathbf{x},\mathbb{R}^{n} \backslash \Omega) ) $ for any $\mathbf{x}\in \mathbb{R}^{n}$ to get $\mathrm{Unp}(\partial \Omega) = \mathrm{Unp}(\overline{\Omega}) \cap \mathrm{Unp}(\mathbb{R}^{n} \backslash \Omega)$ and the equality of Lemma \ref{lemme_reach} follows from definitions.
\end{proof}

\begin{proposition}[\textbf{Federer \cite[Theorem 4.8]{Federer}}]
\label{prop_federer_absurde}
Consider any non-empty closed subset $A$ of $\mathbb{R}^{n}$, a point $\mathbf{x} \in A$, and a vector $\mathbf{v}$ of $\mathbb{R}^{n}$. If the set $\lbrace t > 0, ~ \mathbf{x} + t \mathbf{v} \in \mathrm{Unp}(A) ~ \mathrm{and} ~ p_{A}(\mathbf{x} + t \mathbf{v}) = \mathbf{x} \rbrace  $ is not empty and bounded from above, then its supremum $\tau$ is well defined and $\mathbf{x} + \tau \mathbf{v}$ cannot belong to the interior of $\mathrm{Unp}(A)$. 
\end{proposition}

\begin{proof}
We refer to \cite{Federer} for a proof using Peano's Existence Theorem on differential equations. 
\end{proof}

\begin{corollary}
\label{coro_reach_existence_projete}
For any point $ \mathbf{x} \in \partial \Omega$ satisfying $  \mathrm{Reach}( \partial \Omega,\mathbf{x}) > 0 $, there exists two different points $   \mathbf{y} \in \mathrm{Unp}(\overline{\Omega}) \backslash \lbrace \mathbf{x} \rbrace $ and $\mathbf{\tilde{y}} \in \mathrm{Unp}(\mathbb{R}^{n}  \backslash \Omega) \backslash \lbrace \mathbf{x} \rbrace$ such that $  p_{\overline{\Omega}} (\mathbf{y}) = p_{\mathbb{R}^{n} \backslash \Omega}(\mathbf{\tilde{y}}) = \mathbf{x}$.
\end{corollary}

\begin{proof}
Consider $\mathbf{x} \in \partial \Omega$ satisfying $ \mathrm{Reach}(\partial\Omega,\mathbf{x}) > 0$. From Lemma \ref{lemme_reach}, there exists $r > 0$ such that $ \overline{B_{r}}(\mathbf{x}) \subseteq \mathrm{Unp}(\overline{\Omega})$. Let $(\mathbf{x}_{i})_{i \in \mathbb{N}}$ be a sequence of elements in $B_{\frac{r}{2}} (\mathbf{x}) \backslash \overline{\Omega}$ converging to $\mathbf{x}$. We set:
\[ \forall i \in \mathbb{N},~ \forall t \in \mathbb{R}, \quad \mathbf{z}_{i} (t) = p_{\overline{\Omega}}(\mathbf{x}_{i}) + t \frac{\mathbf{x}_{i} - p_{\overline{\Omega}}(\mathbf{x}_{i})}{\Vert \mathbf{x}_{i} - p_{\overline{\Omega}}(\mathbf{x}_{i}) \Vert } \qquad \mathrm{and} \qquad t_{i} = \dfrac{r}{2} + d(\mathbf{x}_{i},\overline{\Omega}) , \] 
which is well defined since $\mathbf{x}_{i} \in \mathrm{Unp}(\overline{\Omega})$. First, $\mathbf{z}_{i}(t) \in \overline{B_{\frac{r}{2}}}(\mathbf{x}_{i}) \subseteq B_{r}(\mathbf{x}) \subseteq \mathrm{Unp}(\overline{\Omega})$ for any $t \in [0,t_{i}] $. Then, using Federer's result recalled in Proposition \ref{prop_federer_absurde}, one can prove by contradiction that:
\[ \forall t \in [0,t_{i}], \quad p_{\overline{\Omega}}(\mathbf{z}_{i}(t)) = p_{\overline{\Omega}}(\mathbf{x}_{i}). \] 
Finally, the sequence $\mathbf{y}_{i} = \mathbf{\mathbf{z}}_{i}(t_{i})$ satisfies $\Vert \mathbf{y}_{i} - \mathbf{x}_{i} \Vert = \frac{r}{2}$ and also $p_{\overline{\Omega}}(\mathbf{y}_{i}) = p_{\overline{\Omega}}(\mathbf{x}_{i})$. Moreover, since it is bounded, $(\mathbf{y}_{i})_{i \in \mathbb{N}}$ is converging, up to a subsequence, to a point denoted by $\mathbf{y} \in \overline{B_{r}}(\mathbf{x}) \subseteq \mathrm{Unp}(\overline{\Omega})$. Using the continuity of $p_{\overline{\Omega}}$ \cite[Theorem 4.8 (4)]{Federer}, we get $\mathbf{y} \in \mathrm{Unp}(\overline{\Omega}) \backslash \lbrace \mathbf{x} \rbrace$ and $p_{\overline{\Omega}}(\mathbf{y}) = p_{\overline{\Omega}}(\mathbf{x})= \mathbf{x}$. To conclude, similar arguments work when replacing $\overline{\Omega}$ by the set $\mathbb{R}^{n} \backslash \Omega$ so Corollary \ref{coro_reach_existence_projete} holds.
\end{proof}

\begin{proof}[\textbf{Proof of Point (ii) in Theorem \ref{thm_reach_equiv_boule}}]
Since $\Omega \notin \lbrace \emptyset, \mathbb{R}^{n} \rbrace$, $\partial \Omega \neq \emptyset$ thus its reach is well defined. We assume $\mathrm{Reach}(\partial \Omega) > 0$, choose $\varepsilon \in ]0,\mathrm{Reach}(\partial \Omega)[$, and consider $\mathbf{x} \in \partial \Omega$. From Corollary \ref{coro_reach_existence_projete}, there exists $\mathbf{y} \in \mathrm{Unp}(\overline{\Omega}) \backslash \lbrace \mathbf{x} \rbrace$ such that $p_{\overline{\Omega}}(\mathbf{y}) = \mathbf{x}$ so we can set $\mathbf{d_{x}} = \frac{\mathbf{x} - \mathbf{y}}{\Vert \mathbf{x} - \mathbf{y} \Vert}$. From Lemma \ref{lemme_reach}, we get $\mathbf{x} + [0,\varepsilon] \mathbf{d_{x}} \subseteq \mathrm{Unp}(\overline{\Omega})$. Then, we use Proposition \ref{prop_federer_absurde} again to prove by contradiction that:
\[ \forall t \in [0, \varepsilon], \quad p_{\overline{\Omega}}(\mathbf{x}+ t \mathbf{d_{x}}) = \mathbf{x} . \]
In particular, we have $\Vert \mathbf{z} - (\mathbf{x} + \varepsilon \mathbf{d_{x}}) \Vert > \varepsilon$ for any point $\mathbf{z} \in \overline{\Omega} \backslash \lbrace \mathbf{x} \rbrace$ from which we deduce that:
\[  \overline{\Omega} \subseteq \lbrace \mathbf{x} \rbrace \cup \left( \mathbb{R}^{n} \backslash \overline{B_{\varepsilon}}(\mathbf{x}+ \varepsilon \mathbf{d_{x}}) \right) \quad \Longleftrightarrow \quad \overline{B_{\varepsilon}}(\mathbf{x} + \varepsilon \mathbf{d_{x}}) \backslash \lbrace \mathbf{x} \rbrace \subseteq \mathbb{R}^{n} \backslash \overline{\Omega}. \]
Similarly, there exists a unit vector $\mathbf{\xi_{x}}$ of $\mathbb{R}^{n}$ such that we get $\overline{B_{\varepsilon}}(\mathbf{x}+ \varepsilon \mathbf{\xi_{\mathbf{x}}}) \backslash \lbrace \mathbf{x} \rbrace \subseteq \Omega$. Since we have $ \overline{ B_{\varepsilon}}(\mathbf{x} + \varepsilon \mathbf{\xi_{x}}) \cap \overline{B_{\varepsilon}}(\mathbf{x}+ \varepsilon \mathbf{d_{x}}) = \lbrace \mathbf{x} \rbrace $, we obtain $\mathbf{d_{x}} = - \mathbf{\xi_{x}}$. To conclude, if $  \mathrm{Reach}(\partial \Omega) < + \infty $, then observe that $ B_{\mathrm{Reach}(\partial \Omega)}(\mathbf{x} \pm \mathrm{Reach}(\partial \Omega) \mathbf{d_{x}}) = \bigcup_{0 < \varepsilon < \mathrm{Reach}(\partial \Omega)} \overline{B_{\varepsilon}}(\mathbf{x}\pm \varepsilon \mathbf{d_{x}}) \backslash \lbrace \mathbf{x} \rbrace  $ in order to check that $\Omega$ also satisfies the $\mathrm{Reach}(\partial \Omega)$-ball condition.
\end{proof}

\subsubsection{Uniform ball condition implies positive reach}
\label{section_boule_imply_reach}
\begin{proposition}
\label{prop_normale_lipschitzienne}
Assume there exists $\varepsilon > 0$ such that $\Omega \in \mathcal{O}_{\varepsilon}(\mathbb{R}^{n}) $. Then, we have:
\begin{equation}
\label{eqn_lipschitz_normale}
\forall (\mathbf{x},\mathbf{y}) \in \partial \Omega \times \partial \Omega, \quad \Vert \mathbf{d_{x}} - \mathbf{d_{y}} \Vert  \leqslant  \dfrac{1}{\varepsilon} \Vert \mathbf{x} - \mathbf{y} \Vert .   
\end{equation}
In particular, if $\mathbf{x} = \mathbf{y}$, then $\mathbf{d_{x}} = \mathbf{d_{y}}$ which ensures the unit vector $\mathbf{d_{x}}$ of Definition \ref{definition_epsilon_boule} is unique. In other words, the map $\mathbf{d}: \mathbf{x} \in \partial \Omega \mapsto \mathbf{d_{x}}$ is well defined and $ \frac{1}{\varepsilon} $-Lipschitz continuous. 
\end{proposition}

\begin{proof}
Let $\varepsilon > 0$ and $\Omega \in \mathcal{O}_{\varepsilon}(\mathbb{R}^{n})$. Since $\Omega \notin \lbrace \emptyset, \mathbb{R}^{n} \rbrace$, $\partial \Omega $ is not empty so choose $ (\mathbf{x},\mathbf{y}) \in \partial \Omega \times \partial \Omega $. First, from the $\varepsilon$-ball condition on $\mathbf{x}$ and $\mathbf{y}$, we have $ B_{\varepsilon}(\mathbf{x}\pm \varepsilon \mathbf{d_{x}}) \cap B_{\varepsilon}(\mathbf{y} \mp \varepsilon \mathbf{d_{y}}) = \emptyset$, from which we deduce $  \Vert \mathbf{x} - \mathbf{y} \pm \varepsilon(\mathbf{d_{x}} + \mathbf{d_{y}}) \Vert \geqslant 2\varepsilon $. Then, squaring these two inequalities and summing them, one obtains the result \eqref{eqn_lipschitz_normale} of the statement: $ \Vert \mathbf{x} - \mathbf{y} \Vert^{2} \geqslant 2\varepsilon^{2} - 2 \varepsilon^{2} \langle \mathbf{d_{x}} ~\vert~ \mathbf{d_{y}} \rangle = \varepsilon^{2} \Vert \mathbf{d_{x}} - \mathbf{d_{y}} \Vert^{2} $.
\end{proof}

\begin{proof}[\textbf{Proof of Point (i) in Theorem \ref{thm_reach_equiv_boule}}]
Let $\varepsilon > 0$ and assume that $\Omega$ satisfies the $\varepsilon$-ball condition. Since $\Omega \notin \lbrace \emptyset, \mathbb{R}^{n} \rbrace$, $\partial \Omega $ is not empty so choose any $\mathbf{x} \in \partial \Omega$ and let us prove $B_{\varepsilon}(\mathbf{x}) \subseteq \mathrm{Unp}(\partial \Omega) $. First, we assume $\mathbf{y} \in B_{\varepsilon}(\mathbf{x} ) \cap \Omega$. Since $\partial \Omega$ is closed, there exists $\mathbf{z} \in \partial \Omega$ such that $d(\mathbf{y},\partial \Omega) = \Vert \mathbf{z} - \mathbf{y} \Vert$. Moreover, we obtain from the $\varepsilon$-ball condition and $\mathbf{y} \in \Omega$:
\[ \left\lbrace \begin{array}{l}
B_{\varepsilon}(\mathbf{z} + \varepsilon \mathbf{d_{z}}) \subseteq \mathbb{R}^{n} \backslash \overline{\Omega} \\
\\
B_{d(\mathbf{y},\partial \Omega)}(\mathbf{y}) \subseteq \Omega
\end{array} \right.  \qquad \Longrightarrow \qquad B_{\varepsilon}(\mathbf{z} + \varepsilon \mathbf{d_{z}}) \cap B_{d(\mathbf{y},\partial \Omega)}(\mathbf{y}) = \emptyset.  \]  
Therefore, we deduce that $\mathbf{y} = \mathbf{z} - d(\mathbf{y},\partial \Omega) \mathbf{d_{z}}$. Then, we show that such a $\mathbf{z}$ is unique. Considering another projection $\mathbf{\tilde{z}}$ of $\mathbf{y}$ on $\partial \Omega$, we get from the foregoing: $\mathbf{y} = \mathbf{z} - d(\mathbf{y},\partial \Omega) \mathbf{d_{z}} = \mathbf{\tilde{z}} - d(\mathbf{y},\partial \Omega) \mathbf{d_{\tilde{z}}} $. Using \eqref{eqn_lipschitz_normale}, we have:
\[ \Vert \mathbf{d_{z}} - \mathbf{d_{\tilde{z}}} \Vert \leqslant \dfrac{1}{\varepsilon} \Vert \mathbf{z} - \mathbf{\tilde{z}} \Vert = \dfrac{d(\mathbf{y},\partial \Omega)}{ \varepsilon} \Vert \mathbf{d_{z}} - \mathbf{d_{\tilde{z}}} \Vert. \] 
Since $d(\mathbf{y},\partial \Omega) \leqslant \Vert \mathbf{x} - \mathbf{y} \Vert < \varepsilon$, the above inequality can only hold if $\Vert \mathbf{d_{z}} - \mathbf{d_{\tilde{z}}} \Vert = 0$ i.e. $ \mathbf{z} = \mathbf{\tilde{z}} $. Hence, we obtain $B_{\varepsilon}(\mathbf{x}) \cap \Omega \subseteq \mathrm{Unp}(\partial \Omega)$ and similarly, one can prove that $ B_{\varepsilon}(\mathbf{x}) \cap (\mathbb{R}^{n} \backslash \overline{ \Omega}) \subseteq \mathrm{Unp}(\partial \Omega)$. Since $\partial \Omega \subseteq \mathrm{Unp}(\partial \Omega)$, we finally get $B_{\varepsilon}(\mathbf{x}) \subseteq \mathrm{Unp}(\partial \Omega)$. To conclude, we have $\mathrm{Reach}(\partial \Omega, \mathbf{x}) \geqslant \varepsilon$ for every $\mathbf{x} \in \partial \Omega$ i.e. $\mathrm{Reach}(\partial \Omega) \geqslant \varepsilon $ as required. 
\end{proof}

\begin{proposition}
\label{prop_inegalites_locale_globale}
Assume there exists $\varepsilon > 0$ such that $\Omega  \in \mathcal{O}_{\varepsilon}(\mathbb{R}^{n}) $. Then, we have:
\begin{equation}
\label{eqn_inegalite_globale}
 \forall (\mathbf{a},\mathbf{x}) \in \partial \Omega \times \partial \Omega, \quad \vert \left\langle \mathbf{x} - \mathbf{a} ~\vert~ \mathbf{d_{a}} \right\rangle \vert \leqslant \frac{1}{2\varepsilon} \Vert \mathbf{x} - \mathbf{a} \Vert^{2} .
\end{equation}
Moreover, introducing the vector $(\mathbf{x} - \mathbf{a})' = (\mathbf{x} - \mathbf{a}) - \langle \mathbf{x} - \mathbf{a} ~\vert~ \mathbf{d_{a}} \rangle \mathbf{d_{a}}  $, if we assume $ \Vert (\mathbf{x} - \mathbf{a})' \Vert < \varepsilon $ and $ \vert \langle \mathbf{x} - \mathbf{a} ~\vert~ \mathbf{d_{a}} \rangle \vert < \varepsilon $, then the following local inequality holds: 
\begin{equation}
\label{eqn_inegalite_locale}
\frac{1}{2 \varepsilon} \Vert \mathbf{x} - \mathbf{a} \Vert^{2} \leqslant \varepsilon - \sqrt{\varepsilon^{2} - \Vert (\mathbf{x} - \mathbf{a})' \Vert^{2}} . 
\end{equation} 
\end{proposition}

\begin{proof}
Let $\varepsilon > 0$ and $\Omega \in \mathcal{O}_{\varepsilon}(\mathbb{R}^{n})$. Since $\Omega \notin \lbrace \emptyset, \mathbb{R}^{n} \rbrace$, $\partial \Omega$ is not empty so choose $ (\mathbf{a},\mathbf{x}) \in \partial \Omega \times \partial \Omega $. Observe that the point $\mathbf{x}$ cannot belong neither to $B_{\varepsilon}(\mathbf{a} - \varepsilon \mathbf{d_{a}}) \subseteq \Omega$ nor to $B_{\varepsilon}(\mathbf{a} + \varepsilon \mathbf{d_{a}}) \subseteq \mathbb{R}^{n} \backslash \overline{\Omega} $. Hence, we have $ \Vert \mathbf{x} - \mathbf{a} \mp \varepsilon \mathbf{d_{a}} \Vert \geqslant \varepsilon $. Squaring these two inequalities, we obtain \eqref{eqn_inegalite_globale}: 
\[ \Vert \mathbf{x} - \mathbf{a} \Vert^{2} \geqslant 2 \varepsilon \vert  \left\langle \mathbf{x} - \mathbf{a} ~\vert~ \mathbf{d_{a}} \right\rangle  \vert \Longleftrightarrow \vert \left\langle \mathbf{x} - \mathbf{a} ~\vert~ \mathbf{d_{a}} \right\rangle \vert^{2} - 2 \varepsilon  \vert \left\langle \mathbf{x} - \mathbf{a} ~\vert~ \mathbf{d_{a}} \right\rangle \vert + \Vert (\mathbf{x} - \mathbf{a})' \Vert^{2} \geqslant 0 . \] 
It is a second-order polynomial inequality and we assume that its reduced discriminant is positive: $\Delta' = \varepsilon^{2} - \Vert (\mathbf{x} - \mathbf{a})' \Vert^{2} > 0  $. Hence, the unknown cannot be located between the two roots: either $ \vert \left\langle \mathbf{x} - \mathbf{a} ~\vert~ \mathbf{d_{a}} \right\rangle \vert \leqslant \varepsilon - \sqrt{\Delta'}$ or $ \vert \left\langle \mathbf{x} - \mathbf{a} ~\vert~ \mathbf{d_{a}} \right\rangle \vert \geqslant \varepsilon + \sqrt{\Delta'}$. We assume $ \vert \left\langle \mathbf{x} - \mathbf{a} ~\vert~ \mathbf{d_{a}} \right\rangle \vert < \varepsilon $ and the last case cannot hold. Squaring the remaining relation, we get the local inequality \eqref{eqn_inegalite_locale} of the statement: $ \Vert \mathbf{x} - \mathbf{a} \Vert^{2} =  \vert \left\langle \mathbf{x} - \mathbf{a} ~\vert~ \mathbf{d_{a}} \right\rangle \vert^{2} + \Vert (\mathbf{x} - \mathbf{a})' \Vert^{2} \leqslant 2\varepsilon^{2} - 2 \varepsilon \sqrt{\varepsilon^{2} - \Vert (\mathbf{x} - \mathbf{a})' \Vert^{2}} $.
\end{proof}

\subsection{The uniform ball condition and the compact $C^{1,1}$-hypersurfaces}
In this section, Theorem \ref{thm_boule_equiv_cunun} is proved. First, we show $\partial \Omega$ can be considered locally as the graph of a function whose $C^{1,1}$-regularity is then established. Finally, we demonstrate that the converse statement holds in the compact case. Hence, it is the optimal regularity we can expect from the uniform ball property. The proofs in Sections \ref{section_boule_imply_reach}, \ref{section_parametrization_cunun}, and \ref{section_regularite_cunun} inspire those of Section \ref{section_parametrization_i}. 

\subsubsection{A local parametrization of the boundary $\partial \Omega$}
\label{section_parametrization_cunun}
We now set $\varepsilon > 0$ and assume that the open set $\Omega$ satisfies the $\varepsilon$-ball condition. Since $\Omega \notin \lbrace \emptyset, \mathbb{R}^{n} \rbrace$, $\partial \Omega $ is not empty so we consider any point $ \mathbf{x}_{0} \in \partial \Omega $ and its unique vector $ \mathbf{d}_{\mathbf{x}_{0}} $ from Proposition \ref{prop_normale_lipschitzienne}. We choose a basis $ \mathcal{B}_{\mathbf{x}_{0}} $ of the hyperplane $ \mathbf{d}_{\mathbf{x}_{0}}^{\bot} $ so that $ (\mathbf{x}_{0},\mathcal{B}_{\mathbf{x}_{0}},\mathbf{d}_{\mathbf{x}_{0}} ) $ is a direct orthonormal frame. Inside this frame, any point $ \mathbf{x} \in \mathbb{R}^{n} $ is of the form $ (\mathbf{x'}, x_{n}) $ such that $ \mathbf{x'} = (x_{1}, \ldots, x_{n-1}) \in \mathbb{R}^{n-1} $. The zero vector $ \mathbf{0} $ of $\mathbb{R}^{n}$ is now identified with $\mathbf{x}_{0}$ so we have $ B_{\varepsilon}(\mathbf{0'},-\varepsilon ) \subseteq \Omega $ and $ B_{\varepsilon}(\mathbf{0'},\varepsilon) \subseteq \mathbb{R}^{n} \backslash \overline{\Omega} $.

\begin{proposition}
\label{prop_parametrisations_locales}
The following maps $\varphi^{\pm}$ are well defined on $ D_{\varepsilon}(\mathbf{0'}) = \lbrace \mathbf{x'} \in \mathbb{R}^{n-1}, ~ \Vert \mathbf{x'} \Vert < \varepsilon \rbrace $:
\[ \left\lbrace \begin{array}{rcl}
\varphi^{+} : \quad D_{\varepsilon}(\mathbf{0'}) & \longrightarrow & ]- \varepsilon,\varepsilon[ \\
 \mathbf{x'} & \longmapsto & \sup \lbrace x_{n} \in [-\varepsilon,\varepsilon], \quad (\mathbf{x'},x_{n}) \in \Omega \rbrace \\
 & & \\
\varphi^{-} : \quad  D_{\varepsilon}(\mathbf{0'}) & \longrightarrow & ]- \varepsilon,\varepsilon[ \\
 \mathbf{x'} & \longmapsto & \inf \lbrace x_{n} \in [- \varepsilon,\varepsilon], \quad (\mathbf{x'},x_{n}) \in \mathbb{R}^{n} \backslash \overline{\Omega} \rbrace. \\
\end{array} \right.  \]
Moreover, for any $ \mathbf{x'} \in D_{\varepsilon}(\mathbf{0'}) $, introducing the points $ \mathbf{x^{\pm}} = (\mathbf{x'},\varphi^{\pm}(\mathbf{x'})) $, we have $ \mathbf{x^{\pm}} \in \partial \Omega $ and: 
\begin{equation}
\label{eqn_parametrisation_locale}
\vert \varphi^{\pm}(\mathbf{x'}) \vert \leqslant \dfrac{1}{2 \varepsilon} \Vert \mathbf{x^{\pm}} - \mathbf{x}_{0} \Vert^{2} \leqslant \varepsilon - \sqrt{\varepsilon^{2}-\Vert \mathbf{x'} \Vert^{2}}.
\end{equation}
\end{proposition}

\begin{proof}
Let $ \mathbf{x'} \in D_{\varepsilon}(\mathbf{0'}) $ and $g: t \in [-\varepsilon,\varepsilon] \mapsto (\mathbf{x'},t)$. Since $-\varepsilon \in g^{-1}(\Omega) \subseteq [-\varepsilon,\varepsilon] $, we can set $\varphi^{+}(\mathbf{x'}) = \sup g^{-1}(\Omega)$. The map $g$ is continuous so $g^{-1}(\Omega)$ is open and $\varphi^{+}(\mathbf{x'}) \neq \varepsilon$ thus we get $\varphi(\mathbf{x'}) \notin g^{-1}(\Omega) $ i.e. $\mathbf{x}^{+} \in \overline{\Omega} \backslash \Omega $. Similarly, the map $ \varphi^{-}$ is well defined and $\mathbf{x}^{-} \in \partial \Omega$. Finally, we use \eqref{eqn_inegalite_globale} and \eqref{eqn_inegalite_locale} on the points $ \mathbf{x}_{0}$ and $\mathbf{x} = \mathbf{x}^{\pm}$ in order to obtain \eqref{eqn_parametrisation_locale}.
\end{proof}

\begin{lemma}
\label{lemme_parametrisation_locale}
Let $ r  = \frac{\sqrt{3}}{2} \varepsilon$ and $\mathbf{x'} \in D_{r}(\mathbf{0'})$. We assume that there exists $x_{n} \in ]- \varepsilon,\varepsilon[$ such that $\mathbf{x} = (\mathbf{x'},x_{n}) \in \partial \Omega$ and $ \tilde{x}_{n} \in \mathbb{R} $ such that $ \vert  \tilde{x}_{n} \vert \leqslant \varepsilon - \sqrt{\varepsilon^{2} - \Vert \mathbf{x'} \Vert^{2} }$. Then, we introduce $ \mathbf{\tilde{x}} = (\mathbf{x'}, \tilde{x}_{n})$ and the two following implications hold: $ (\tilde{x}_{n} < x_{n} \Longrightarrow  \mathbf{\tilde{x}} \in \Omega) \quad \mathrm{and} \quad (\tilde{x}_{n} > x_{n}  \Longrightarrow  \mathbf{\tilde{x}} \in \mathbb{R}^{n} \backslash \overline{\Omega}) $.
\end{lemma}

\begin{proof}
Let $ \mathbf{x'} \in D_{r}(\mathbf{0'}) $. Since $  \mathbf{\tilde{x}} - \mathbf{x}  = (\tilde{x}_{n}- x_{n})\mathbf{d}_{\mathbf{x}_{0}} $, if we assume  $\tilde{x}_{n} > x_{n} $, then we have:
\[ \begin{array}{rcl}
\Vert \mathbf{\tilde{x}} - \mathbf{x} - \varepsilon \mathbf{d_{x}} \Vert^{2}  - \varepsilon^{2} & = &  \vert \tilde{x}_{n} - x_{n} \vert \left( \vert \tilde{x}_{n} - x_{n} \vert + \varepsilon \Vert \mathbf{d_{x}} - \mathbf{d}_{\mathbf{x}_{0}} \Vert^{2} - 2 \varepsilon \right) \\
 & &\\
&\leqslant &  \vert \tilde{x}_{n} - x_{n} \vert \left( \vert \tilde{x}_{n} \vert + \vert x_{n} \vert + \frac{1}{\varepsilon} \Vert \mathbf{x} - \mathbf{x}_{0} \Vert^{2} - 2 \varepsilon \right) \\
 & & \\
 & \leqslant&  \vert \tilde{x}_{n} - x_{n} \vert \left( 2 \varepsilon - 4 \sqrt{\varepsilon^{2} - \Vert \mathbf{x'} \Vert^{2}} \right) < \vert \tilde{x}_{n} - x_{n} \vert \left( 2 \varepsilon - 4 \sqrt{\varepsilon^{2} - r^{2}} \right) = 0. \\
\end{array} \]
Indeed, we used \eqref{eqn_lipschitz_normale} with $\mathbf{x} \in \partial \Omega$ and $\mathbf{y} = \mathbf{x}_{0}$, \eqref{eqn_inegalite_globale} and \eqref{eqn_inegalite_locale} applied to $\mathbf{x} \in \partial \Omega$ and $\mathbf{a} = \mathbf{x}_{0}$, and also the hypothesis made on  $\tilde{x}_{n}$. Hence, we proved that if $ \tilde{x}_{n} > x_{n}$, then $\mathbf{\tilde{x}} \in B_{\varepsilon}(\mathbf{x} + \varepsilon \mathbf{d_{x}}) \subseteq \mathbb{R}^{n} \backslash \overline{\Omega} $. Similarly, one can prove that if $ \tilde{x}_{n} < x_{n}$, then we have $\mathbf{\tilde{x}} \in B_{\varepsilon}(\mathbf{x} - \varepsilon \mathbf{d_{x}}) \subseteq \Omega   $. 
\end{proof}

\begin{proposition} 
\label{propo_parametrisation_locale}
Set $r = \frac{\sqrt{3}}{2} \varepsilon $. Then, the two maps $\varphi^{\pm}$ of Proposition \ref{prop_parametrisations_locales} coincide on $ D_{r}(\mathbf{0'}) $. We denote by $ \varphi $ their common restriction. Moreover, we have $\varphi(\mathbf{0'}) = 0$ and also:
\[\left\lbrace \begin{array}{rcl}
\partial \Omega \cap  \left( D_{r}(\mathbf{0'}) \times ]- \varepsilon, \varepsilon [ \right) & = & \lbrace (\mathbf{x'},\varphi(\mathbf{x'})), \quad \mathbf{x'} \in D_{r}(\mathbf{0'}) \rbrace \\
  \\
\Omega \cap \left(  D_{r}(\mathbf{0'}) \times   ]- \varepsilon, \varepsilon [ \right)  & = &   \lbrace (\mathbf{x'},x_{n}), \quad \mathbf{x'} \in D_{r}(\mathbf{0'}) ~ \mathrm{and} ~
-\varepsilon < x_{n} <\varphi(\mathbf{x'}) \rbrace. \\
\end{array} \right. \]
\end{proposition}

\begin{proof}
Assume by contradiction that there exists $ \mathbf{x'} \in D_{r}(\mathbf{0'}) $ such that $ \varphi^{-}(\mathbf{x'}) \neq \varphi^{+}(\mathbf{x'}) $. We set $ \mathbf{x}=(\mathbf{x'},\varphi^{+}(\mathbf{x'})) $ and $ \mathbf{\tilde{x}}=(\mathbf{x'},\varphi^{-}(\mathbf{x'})) $. By using \eqref{eqn_parametrisation_locale}, the hypothesis of Lemma \ref{lemme_parametrisation_locale} are satisfied for $\mathbf{x}$ and $\mathbf{\tilde{x}}$. Hence, either $ (\varphi^{-}(\mathbf{x'}) < \varphi^{+}(\mathbf{x'}) \Rightarrow  \mathbf{\tilde{x}} \in \Omega ) $ or $ (\varphi^{-}(\mathbf{x'}) > \varphi^{+}(\mathbf{x'}) \Rightarrow  \mathbf{\tilde{x}} \in \mathbb{R}^{n} \backslash \overline{\Omega}) $ whereas $ \mathbf{\tilde{x}} \in \partial \Omega $. We deduce $ \varphi^{-}(\mathbf{x'}) = \varphi^{+}(\mathbf{x'}) $ for any $\mathbf{x'} \in D_{r}(\mathbf{0'})$. Now consider $ \mathbf{x'} \in D_{r}(\mathbf{0'}) $ and $ x_{n} \in ]-\varepsilon,\varepsilon[$. We set $ \mathbf{x} = (\mathbf{x'},\varphi(\mathbf{x'}))  $ and $\mathbf{\tilde{x}} = (\mathbf{x'},x_{n})$. If $x_{n} = \varphi(\mathbf{x'})$, then Proposition \ref{prop_parametrisations_locales} ensures that $\mathbf{x} \in \partial \Omega$. Moreover, if $ - \varepsilon < x_{n} < -\varepsilon + \sqrt{\varepsilon^{2}- \Vert \mathbf{x'} \Vert^{2}}$, then $\mathbf{\tilde{x}} \in B_{\varepsilon}(\mathbf{0'},-\varepsilon) \subseteq \Omega $, and if $ - \varepsilon + \sqrt{\varepsilon^{2} - \Vert \mathbf{x'} \Vert^{2}} \leqslant x_{n} < \varphi(\mathbf{x'}) $, then apply Lemma \ref{lemme_parametrisation_locale} to get $\mathbf{\tilde{x}} \in \Omega$. Consequently, we proved $ ( - \varepsilon < x_{n} <  \varphi (\mathbf{x'}) \Longrightarrow  (\mathbf{x'},x_{n}) \in \Omega) $ for any $\mathbf{x'} \in D_{r}(\mathbf{0'})$. Similar arguments hold when $ \varepsilon > x_{n} > \varphi(\mathbf{x'})$ and imply $(\mathbf{x'},x_{n}) \in \mathbb{R}^{n} \backslash \overline{\Omega}$. To conclude, note that $\mathbf{x}_{0} = \mathbf{0} = (\mathbf{0'},\varphi(\mathbf{0'}))$.
\end{proof}

\subsubsection{The $C^{1,1}$-regularity of the local graph}
\label{section_regularite_cunun} 
\begin{lemma}
\label{lemme_alpha_cone}
The map $f: \alpha \in ]0,\frac{\pi}{2}[ \mapsto \frac{2 \alpha}{\cos \alpha} \in ]0, + \infty[ $ is well defined, continuous, surjective and increasing. In particular, it is an homeomorphism and its inverse $f^{-1}$ satisfies:
\begin{equation}
\label{eqn_lemme_alpha_cone}
\forall \varepsilon > 0 , \quad f^{-1}(\varepsilon) < \dfrac{\varepsilon }{2} .
\end{equation}
\end{lemma}

\begin{proof}
The proof is basic calculus.
\end{proof}

\begin{proposition}[\textbf{Point (i) of Theorem \ref{thm_boule_equiv_cunun}}]
\label{prop_alpha_cone} 
Consider any $\alpha \in ]0,f^{-1}(\varepsilon)]$ where $f$ is defined in Lemma \ref{lemme_alpha_cone}. Then, we have $ C_{\alpha}(\mathbf{x},-\mathbf{d}_{\mathbf{x}_{0}}) \subseteq \Omega $ for any $ \mathbf{x} \in B_{\alpha}(\mathbf{x}_{0}) \cap \overline{\Omega} $. In particular, the set $\Omega$ satisfies the $f^{-1}(\varepsilon)$-cone property in the sense of Definition \ref{definition_alpha_cone}.
\end{proposition}

\begin{proof}
We set $r = \frac{\sqrt{3}}{2} \varepsilon $ and $\mathcal{C}_{r,\varepsilon} = D_{r}(\mathbf{0'})\times ]-\varepsilon,\varepsilon[$. We choose any $ \alpha  \in ]0,f^{-1}(\varepsilon)] $ then consider $\mathbf{x} = (\mathbf{x'},x_{n}) \in B_{\alpha}(\mathbf{x}_{0}) \cap \overline{\Omega} $ and $ \mathbf{y} = (\mathbf{y'},y_{n}) \in C_{\alpha}(\mathbf{x},- \mathbf{d}_{\mathbf{x}_{0}}) $. The proof of the assertion $\mathbf{y} \in \Omega$ is divided into three steps:
\begin{itemize}
\item check that $\mathbf{x} \in \mathcal{C}_{r,\varepsilon} $ so as to introduce the point $\mathbf{\tilde{x}} = (\mathbf{x'},\varphi(\mathbf{x'})) $ of $ \partial \Omega $ satisfying $ x_{n} \leqslant \varphi (\mathbf{x'} )$;
\item consider $\mathbf{\tilde{y}} = (\mathbf{y'},y_{n} + \varphi(\mathbf{x'}) - x_{n})$ and prove that  $\mathbf{\tilde{y}} \in C_{\alpha}(\mathbf{\tilde{x}},- \mathbf{d}_{\mathbf{x}_{0}} ) \subseteq B_{\varepsilon}(\mathbf{\tilde{x}} - \varepsilon \mathbf{d}_{\mathbf{\tilde{x}}}) \subseteq \Omega$;
\item show that $(\mathbf{\tilde{y}},\mathbf{y}) \in \mathcal{C}_{r,\varepsilon} \times \mathcal{C}_{r,\varepsilon} $ in order to deduce $ y_{n} + \varphi(\mathbf{x'}) - x_{n} < \varphi (\mathbf{y'})$ and conclude $ \mathbf{y} \in \Omega$.
\end{itemize}
First, from \eqref{eqn_lemme_alpha_cone}, we have: $ \max ( \Vert \mathbf{x'} \Vert, \vert x_{n} \vert ) \leqslant \Vert \mathbf{x} - \mathbf{x}_{0} \Vert < \alpha \leqslant f^{-1}(\varepsilon) < \frac{\varepsilon}{2} $. Hence, we get $\mathbf{x} \in \overline{\Omega} \cap \mathcal{C}_{r,\varepsilon}$ and applying Proposition \ref{propo_parametrisation_locale}, it comes $ x_{n} \leqslant \varphi (\mathbf{x'})$. We set $\mathbf{\tilde{x}} = (\mathbf{x'},\varphi(\mathbf{x'})) \in \partial \Omega \cap \mathcal{C}_{r,\varepsilon}$. Note that $\mathbf{\tilde{x}} \in B_{\alpha \sqrt{2} }(\mathbf{x}_{0})$ because Relation \eqref{eqn_parametrisation_locale} applied to $\mathbf{\tilde{x}} = (\mathbf{x'},\varphi(\mathbf{x'}))$ gives: 
\[ \Vert \mathbf{\tilde{x}} - \mathbf{x}_{0} \Vert^{2} \leqslant 2 \varepsilon^{2} - 2 \varepsilon \sqrt{\varepsilon^{2} - \Vert \mathbf{x'} \Vert^{2}} = \frac{4 \varepsilon^{2} \Vert \mathbf{x'} \Vert^{2}}{2 \varepsilon^{2} + 2 \varepsilon  \sqrt{\varepsilon^{2} - \Vert \mathbf{x'} \Vert^{2}} } \leqslant 2 \Vert \mathbf{x'} \Vert^{2} \leqslant 2 \Vert \mathbf{x} - \mathbf{x}_{0} \Vert^{2} < 2 \alpha^{2}. \]
Then, we prove $ C_{\alpha}(\mathbf{\tilde{x}}, -\mathbf{d}_{\mathbf{x}_{0}}) \subseteq B_{\varepsilon}(\mathbf{\tilde{x}} - \varepsilon \mathbf{d}_{\mathbf{\tilde{x}}} )  $ so consider any point $\mathbf{z} \in C_{\alpha}(\mathbf{\tilde{x}}, - \mathbf{d}_{\mathbf{x}_{0}} ) $. Using the Cauchy-Schwartz inequality, \eqref{eqn_lipschitz_normale} applied to $ \mathbf{\tilde{x}} \in \partial \Omega$ and $\mathbf{y} = \mathbf{x}_{0}$, the fact that $\mathbf{z}  \in C_{\alpha}(\mathbf{\tilde{x}}, - \mathbf{d}_{\mathbf{x}_{0}} ) $, and the foregoing observation $ \mathbf{\tilde{x}} \in  B_{\alpha \sqrt{2} }(\mathbf{x}_{0})  $, we have successively:
\[  \begin{array}{rcl}
 \Vert \mathbf{z} - \mathbf{\tilde{x}} + \varepsilon \mathbf{d_{\tilde{x}}} \Vert^{2} - \varepsilon^{2} & \leqslant &  \Vert \mathbf{z} - \mathbf{\tilde{x}} \Vert^{2} + 2 \varepsilon \Vert \mathbf{z} - \mathbf{\tilde{x}} \Vert \Vert \mathbf{d_{\tilde{x}}} - \mathbf{d}_{\mathbf{x}_{0}} \Vert + 2 \varepsilon \left\langle \mathbf{z} - \mathbf{\tilde{x}} ~\vert~ \mathbf{d}_{\mathbf{x}_{0}} \right\rangle  \\
  & & \\
  & < &  \Vert \mathbf{z} - \mathbf{\tilde{x}} \Vert^{2} + 2 \Vert \mathbf{z} - \mathbf{\tilde{x}} \Vert \Vert \mathbf{\tilde{x}} - \mathbf{x}_{0} \Vert - 2 \varepsilon  \Vert \mathbf{z} - \mathbf{\tilde{x}} \Vert \cos  \alpha \\
  \\
  & < &  \Vert \mathbf{z} - \mathbf{\tilde{x}} \Vert \left[ \left( 1 + 2\sqrt{2} \right) \alpha - 2 \varepsilon \cos \alpha \right] ~< ~ 2  \Vert \mathbf{z} - \mathbf{\tilde{x}} \Vert \cos \alpha  \left( f(\alpha) - \varepsilon \right) \leqslant 0 . \\
\end{array} \]
Hence, we get $\mathbf{z} \in B_{\varepsilon}(\mathbf{\tilde{x}} - \varepsilon \mathbf{d}_{\mathbf{\tilde{x}}})$ i.e. $C_{\alpha}(\mathbf{\tilde{x}}, -\mathbf{d}_{\mathbf{x}_{0}}) \subseteq B_{\varepsilon}(\mathbf{\tilde{x}} - \varepsilon \mathbf{d}_{\mathbf{\tilde{x}}} ) \subseteq \Omega$ using the $\varepsilon$-ball condition. Moreover, since $\mathbf{\tilde{y}} - \mathbf{\tilde{x}} = \mathbf{y} - \mathbf{x}$ and $\mathbf{y} \in C_{\alpha}(\mathbf{x},-\mathbf{d}_{\mathbf{x}_{0}})$, we obtain $\mathbf{\tilde{y}} \in C_{\alpha}(\mathbf{\tilde{x}},-\mathbf{d}_{\mathbf{x}_{0}}) $ and thus $\mathbf{\tilde{y}} \in \Omega$. Finally, we show that $(\mathbf{y},\mathbf{\tilde{y}}) \in \mathcal{C}_{r,\varepsilon}\times \mathcal{C}_{r, \varepsilon}$. We have successively: 
\[  \left\lbrace \begin{array}{l}
\Vert \mathbf{y'} \Vert  \leqslant \Vert \mathbf{y'} - \mathbf{x'} \Vert + \Vert \mathbf{x'} \Vert < \sqrt{\alpha^{2} - \alpha^{2} \cos^{2} \alpha} + \alpha = \dfrac{\alpha}{\cos \alpha} \left(  \dfrac{1}{2} \sin 2 \alpha + \cos \alpha \right) \leqslant \frac{3  f(\alpha) }{4}  \leqslant \frac{3 \varepsilon }{4} < r \\
\\
\vert y_{n} \vert \leqslant \vert y_{n} - x_{n} \vert + \vert x_{n} \vert \leqslant \Vert \mathbf{y} - \mathbf{x} \Vert + \Vert \mathbf{x} - \mathbf{x}_{0} \Vert  < 2 \alpha < f(\alpha) \leqslant \varepsilon \\
\\
\vert y_{n} + \varphi(\mathbf{x'}) - x_{n} \vert \leqslant \Vert  \mathbf{y} - \mathbf{x} \Vert + \varepsilon - \sqrt{\varepsilon^{2} - \Vert \mathbf{x'} \Vert^{2}} < \alpha + \dfrac{\Vert \mathbf{x'} \Vert^{2}}{\varepsilon + \sqrt{\varepsilon^{2} - \Vert \mathbf{x'} \Vert^{2}}}  \leqslant \alpha + \frac{\alpha^{2}}{\varepsilon}  < \frac{3}{2} \alpha \leqslant \varepsilon. \\
\end{array} \right. \]
We used \eqref{eqn_parametrisation_locale}, \eqref{eqn_lemme_alpha_cone}, the fact that $\mathbf{y} \in C_{\alpha}(\mathbf{x},-\mathbf{d}_{\mathbf{x}_{0}})$, and $\mathbf{x} \in B_{\alpha}(\mathbf{x}_{0})$. To conclude, apply Proposition \ref{propo_parametrisation_locale} to $\mathbf{\tilde{y}} \in \Omega \cap \mathcal{C}_{r,\varepsilon}$ in order to obtain $ y_{n} + \varphi(\mathbf{x'}) - x_{n} < \varphi (\mathbf{y'})$. Since we firstly proved $ x_{n} \leqslant \varphi (\mathbf{x'})$, we deduce $ y_{n} < \varphi(\mathbf{y'}) $. Applying Proposition \ref{propo_parametrisation_locale} to $\mathbf{y} \in \mathcal{C}_{r,\varepsilon}$, we get $\mathbf{y} \in \Omega $ as required.
\end{proof}

\begin{corollary}
\label{lemme_lipschitz}
The map $\varphi$ restricted to $ D_{\frac{\sqrt{2}}{4}f^{-1}(\varepsilon)}(\mathbf{0'}) $ is $ \frac{1}{\tan[f^{-1}(\varepsilon)]} $-Lipschitz continuous.
\end{corollary}

\begin{proof}
We set $\alpha = f^{-1}(\varepsilon)$, $ r = \frac{\sqrt{3}}{2} \varepsilon$, and $\tilde{r} = \frac{\sqrt{2}}{4} f^{-1}(\varepsilon)$. We choose any $( \mathbf{x_{+}'}, \mathbf{x_{-}'}) \in D_{\tilde{r}}(\mathbf{0'}) \times D_{\tilde{r}}(\mathbf{0'}) $. From \eqref{eqn_lemme_alpha_cone}, we get  $\tilde{r} < r $ so we can consider $\mathbf{x}_{\pm} = (\mathbf{x'_{\pm}},\varphi(\mathbf{x'_{\pm}}))$ and Proposition \ref{prop_parametrisations_locales} gives:
\[ \Vert \mathbf{x_{\pm}} - \mathbf{x}_{0} \Vert^{2} \leqslant 2 \varepsilon^{2} - 2 \varepsilon \sqrt{\varepsilon^{2}- \Vert \mathbf{x_{\pm}'} \Vert^{2} } = \frac{4 \varepsilon^{2} \Vert \mathbf{x_{\pm}'} \Vert^{2}}{2 \varepsilon^{2} +  2 \varepsilon \sqrt{\varepsilon^{2} - \Vert \mathbf{x_{\pm}'} \Vert^{2}} } \leqslant 2 \Vert \mathbf{x_{\pm}'} \Vert^{2} < 2 \tilde{r}^{2} < \alpha^{2}.  \]
Hence, we obtain $\mathbf{x_{\pm}} \in B_{\alpha}(\mathbf{x}_{0}) \cap \partial \Omega$. We also have: $ \Vert \mathbf{x_{+}} - \mathbf{x_{-}} \Vert \leqslant \Vert \mathbf{x_{+}} - \mathbf{x}_{0} \Vert + \Vert \mathbf{x}_{0} - \mathbf{x_{-}} \Vert  < 2 \tilde{r} \sqrt{2} = \alpha $. Finally, applying Proposition \ref{prop_alpha_cone}, the points $\mathbf{x_{\pm}} $ cannot belong to the cones $ C_{\alpha}(\mathbf{x_{\mp}},- \mathbf{d}_{\mathbf{x}_{0}}) \subseteq \Omega $ thus we get: $ \vert \langle \mathbf{x_{+}} - \mathbf{x_{-}} ~\vert~ \mathbf{d}_{\mathbf{x}_{0}} \rangle \vert \leqslant \cos \alpha \Vert \mathbf{x_{+}} - \mathbf{x_{-}} \Vert = \cos \alpha \sqrt{\Vert \mathbf{x_{+}'} - \mathbf{x_{-}'} \Vert^{2} + \vert \langle \mathbf{x_{+}} - \mathbf{x_{-}} ~\vert~ \mathbf{d}_{\mathbf{x}_{0}} \rangle \vert^{2}} $. Consequently, one can re-arrange these terms in order to obtain the result of the statement: $  \vert \varphi(\mathbf{x_{+}'}) - \varphi(\mathbf{x_{-}'}) \vert = \vert \langle \mathbf{x_{+}} - \mathbf{x_{-}} ~\vert~ \mathbf{d}_{\mathbf{x}_{0}} \rangle \vert \leqslant  \frac{1}{\tan \alpha} \Vert \mathbf{x_{+}'} - \mathbf{x_{-}'} \Vert   $.
\end{proof}

\begin{proposition}
\label{prop_regularite_cunun}
Set $\tilde{r} = \frac{\sqrt{2}}{4} f^{-1}(\varepsilon) $. The map $\varphi$ of Proposition \ref{propo_parametrisation_locale} restricted to $ D_{\tilde{r}} (\mathbf{0'})$ is differentiable and its gradient $\nabla \varphi: D_{ \tilde{r}}(\mathbf{0'}) \rightarrow \mathbb{R}^{n-1}$ is $L$-Lipschitz continuous where $L > 0$ depends only on $\varepsilon$.  Moreover, we have $\nabla \varphi (\mathbf{0'}) = \mathbf{0'}$ and also:
\[ \forall \mathbf{a'} \in D_{\tilde{r}}(\mathbf{0'}), \quad \nabla \varphi (\mathbf{a'}) = \dfrac{-1}{ \langle \mathbf{d_{a}}~\vert~ \mathbf{d}_{\mathbf{x}_{0}} \rangle} \mathbf{d_{a}'},  \qquad \mathrm{where} \quad \mathbf{a}= (\mathbf{a'},\varphi(\mathbf{a'})). \]  
\end{proposition}

\begin{proof}
Let $\mathbf{a'}\in D_{\tilde{r}}(\mathbf{0'})$ and $ \mathbf{x'} \in  \overline{D_{\tilde{r} - \Vert \mathbf{a'} \Vert}}(\mathbf{a'}) $. Consequently, we have $(\mathbf{a'},\mathbf{x'}) \in D_{\tilde{r}}(\mathbf{0'}) \times  D_{\tilde{r}}(\mathbf{0'}) $ and from \eqref{eqn_lemme_alpha_cone}, we get $  \tilde{r} < \frac{\sqrt{3}}{2} \varepsilon $. Hence, using Proposition \ref{propo_parametrisation_locale}, we can introduce $ \mathbf{x} :=(\mathbf{x'},\varphi(\mathbf{x'})) $ and $ \mathbf{a} :=(\mathbf{a'},\varphi(\mathbf{a'})) $. Applying \eqref{eqn_inegalite_globale} to $(\mathbf{a},\mathbf{x}) \in \partial \Omega \times \partial \Omega$ and using the Lipschitz continuity of $\varphi$ on $D_{\tilde{r}}(\mathbf{0'})$ proved in Corollary \ref{lemme_lipschitz}, we deduce that:
\[  \vert \left( \varphi(\mathbf{x'}) - \varphi( \mathbf{a'}) \right) \mathbf{d_{a}}_{n} + \langle \mathbf{d_{a}'} ~\vert ~ \mathbf{x'} - \mathbf{a'} \rangle \vert  \leqslant  \frac{1}{2 \varepsilon} \Vert \mathbf{x} - \mathbf{a} \Vert^{2}   \leqslant \underbrace{  \dfrac{1}{2\varepsilon} \left( 1 + \dfrac{1}{\tan^{2}[f^{-1}(\varepsilon)]} \right) }_{:= C(\varepsilon) > 0} \Vert \mathbf{x'} - \mathbf{a'} \Vert^{2}, \]
where we set $\mathbf{d}_{\mathbf{a}} = (\mathbf{d_{a}'},\mathbf{d_{a}}_{n})$ with $\mathbf{d_{a}}_{n} = \langle \mathbf{d_{a}}~\vert~ \mathbf{d}_{\mathbf{x}_{0}} \rangle$. It represents a first-order Taylor expansion of the map $\varphi$ if we can divide the above inequality by a uniform positive constant smaller than $\mathbf{d_{a}}_{n} $. Let us justify this assertion. Apply \eqref{eqn_lipschitz_normale} to $\mathbf{x} = \mathbf{a}$ and $\mathbf{y} = \mathbf{x}_{0}$, then use \eqref{eqn_parametrisation_locale} to get:
\[ \mathbf{d_{a}}_{n}  =  1 - \frac{1}{2} \Vert \mathbf{d_{a}} - \mathbf{d}_{\mathbf{x}_{0}} \Vert^{2} \geqslant  1 - \frac{1}{2 \varepsilon^{2}} \Vert \mathbf{a} - \mathbf{x}_{0} \Vert^{2} \geqslant 1 - \frac{\varepsilon - \sqrt{\varepsilon^{2} - \Vert \mathbf{a'} \Vert^{2}}}{\varepsilon}  = 1 - \frac{\Vert \mathbf{a'} \Vert^{2}}{\varepsilon ( \varepsilon + \sqrt{\varepsilon^{2} - \Vert \mathbf{a'} \Vert^{2}} )} . \] 
Hence, using \eqref{eqn_lemme_alpha_cone}, we obtain $\mathbf{d_{a}}_{n} > 1 - \frac{\tilde{r}^{2}}{\varepsilon^{2}} > \frac{31}{32} > 0$. Therefore, $ \varphi $ is a differentiable map at any point $ \mathbf{a'} \in D_{\tilde{r}}(\mathbf{0'}) $ and its gradient is the one given in the statement:
\[ \forall \mathbf{x'} \in \overline{D_{\tilde{r} - \Vert \mathbf{a'} \Vert }}(\mathbf{a'}), \quad \begin{array}{|c|}   \varphi(\mathbf{x'}) - \varphi(\mathbf{a'}) + \left\langle \dfrac{\mathbf{d_{a}'}}{\mathbf{d_{a}}_{n}} ~~\vert~ \mathbf{x'} - \mathbf{a'} \right\rangle \\ \end{array} \leqslant  \dfrac{32}{31} C(\varepsilon)  \Vert \mathbf{x'} - \mathbf{a'} \Vert^{2} . \]
Moreover, for any $(\mathbf{a'},\mathbf{x'}) \in D_{\tilde{r}}(\mathbf{0'}) \times D_{\tilde{r}}(\mathbf{0'})  $, we have successively: 
\[ \begin{array}{rcl}
\Vert \nabla \varphi (\mathbf{x'}) - \nabla \varphi (\mathbf{a'}) \Vert & \leqslant & \begin{array}{|c|} \dfrac{1}{\mathbf{d_{a}}_{n} } -  \dfrac{1}{\mathbf{d_{x}}_{n} } \\ \end{array} ~   \Vert \mathbf{d_{x}'} \Vert +  \begin{array}{|c|} \dfrac{1}{ \mathbf{d_{a}}_{n}  } \\ \end{array} ~ \Vert \mathbf{d_{a}'} - \mathbf{d_{x}'} \Vert \leqslant   \left(  \dfrac{ 32^{2}}{31^{2}}  + \dfrac{32}{31} \right) \Vert \mathbf{d_{a}} - \mathbf{d_{x}} \Vert \\
 & & \\
 & \leqslant & \dfrac{32}{31\varepsilon} \left( 1 + \dfrac{32}{31} \right) \Vert \mathbf{x} - \mathbf{a} \Vert   \leqslant  \dfrac{32}{31\varepsilon} \left( 1 + \dfrac{32}{31} \right) \sqrt{1 + \dfrac{1}{\tan^{2}[f^{-1}(\varepsilon)]}} \Vert \mathbf{x'} - \mathbf{a'} \Vert. \\
\end{array} \]
We applied \eqref{eqn_lipschitz_normale} to $\mathbf{x} $ and $\mathbf{y} = \mathbf{a}$, then used the Lipschitz continuity of $\varphi$ proved in Corollary \ref{lemme_lipschitz}. Hence, $ \nabla \varphi:  \mathbf{a'} \in  D_{\tilde{r}}(\mathbf{0'}) \mapsto \nabla \varphi (\mathbf{a'})  $ is $L$-Lipschitz continuous with $L > 0$ depending only on $\varepsilon$.
\end{proof}

\begin{corollary}[\textbf{Points (ii) and (iii) of Theorem \ref{thm_boule_equiv_cunun}}] 
\label{coro_normale_lipschitzienne}
The unit vector $\mathbf{d}_{\mathbf{x}_{0}}$ of Definition \ref{definition_epsilon_boule} is the outer normal to $\partial \Omega$ at the point $\mathbf{x}_{0}$. In particular, the $\frac{1}{\varepsilon}$-Lipschitz continuous map $\mathbf{d}: \mathbf{x} \mapsto \mathbf{d_{x}} $ of Proposition \ref{prop_normale_lipschitzienne} is the Gauss map of the $C^{1,1}$-hypersurface  $\partial \Omega$.
\end{corollary}

\begin{proof}
Consider the map $\varphi : D_{\tilde{r}}(\mathbf{0'}) \rightarrow ]-\varepsilon,\varepsilon[$ whose $C^{1,1}$-regularity comes from Proposition \ref{prop_regularite_cunun}. We define the $C^{1,1}$-map $ X: D_{\tilde{r}}(\mathbf{0'}) \rightarrow \partial \Omega $ by $ X(\mathbf{x'}) = (\mathbf{x'},\varphi(\mathbf{x'})) $ then we consider $ \mathbf{x'} \in D_{\tilde{r}}(\mathbf{0'}) $. We denote by $ (e_{k})_{1 \leqslant k \leqslant n-1} $ the first vectors of our local basis. The tangent plane of $\partial \Omega$ at $  X(\mathbf{x'}) $ is spanned by the vectors $ \partial_{k}X(\mathbf{x'})= e_{k} + (\mathbf{0}', \partial_{k} \varphi (\mathbf{x'})) $. Since any normal vector $\mathbf{u} = (u_{1},\ldots,u_{n}) $ to this hyperplane is orthogonal to this $ (n-1) $ vectors, we have: $ \langle \mathbf{u} ~\vert~ \partial_{k} X(\mathbf{x'}) \rangle = 0    \Leftrightarrow   u_{k} = \frac{u_{n}}{\mathbf{d_{x}}_{n}} \mathbf{d_{x}}_{k} $. Hence, we obtain $\mathbf{u} = \frac{u_{n}}{\mathbf{d_{x}}_{n}} \mathbf{d_{x}}$ so $\mathbf{u}$ is collinear to $\mathbf{d_{x}}$. Now, if we impose that $\mathbf{u}$ points outwards $ \Omega$ and if we assume $ \Vert \mathbf{u} \Vert = 1$, then we get $\mathbf{u} = \mathbf{d_{x}} $. 
\end{proof}

\subsubsection{The compact case: when $C^{1,1}$-regularity implies the uniform ball condition}

\begin{proof}[\textbf{Proof of Theorem \ref{thm_boule_equiv_cunun}}]
Combining Proposition \ref{prop_alpha_cone} and Corollary \ref{coro_normale_lipschitzienne}, it remains to prove the converse part of Theorem \ref{thm_boule_equiv_cunun}. Consider any non-empty compact $C^{1,1}$-hypersurface $\mathcal{S}$ of $\mathbb{R}^{n}$ and its associated inner domain $ \Omega$. Choose any $\mathbf{x}_{0} \in \partial \Omega$ and its local frame as in Definition \ref{definition_regularite_cunun}. First, we have for any $(\mathbf{x'},\mathbf{y'}) \in D_{r}(\mathbf{0'}) \times D_{r}(\mathbf{0'})$ with $g: t \in [0,1] \mapsto \varphi(\mathbf{x'}+ t (\mathbf{y'} - \mathbf{x'}))$:
\[ \begin{array}{rcl} 
\vert \varphi(\mathbf{y'}) - \varphi(\mathbf{x'}) - \langle \nabla \varphi(\mathbf{x'}) ~\vert~ \mathbf{y'} - \mathbf{x'} \rangle  \vert & = &  \begin{array}{|c|} \displaystyle{  \int_{0}^{1} \left[ g'(t) - g'(0) \right] dt} \\ \end{array} \quad \leqslant \quad \displaystyle{ \int_{0}^{1} \vert g'(t) - g'(0) \vert dt } \\
& & \\
& \leqslant & \displaystyle{   \int_{0}^{1} \Vert  \nabla \varphi \left(\mathbf
x'+t(\mathbf{y'}-\mathbf{x'}) \right) - \nabla \varphi(\mathbf{x'}) \Vert  \Vert \mathbf{y'} - \mathbf{x'} \Vert dt } \\
& & \\
& \leqslant &  \dfrac{L}{2} \Vert \mathbf{y'} - \mathbf{x'} \Vert^{2}.   \\
\end{array} \]
Then, we set $\varepsilon_{0} = \min(\frac{1}{L}, \frac{r}{3}, \frac{a}{3}) $ and consider any $\mathbf{x} \in B_{\varepsilon_{0}}(\mathbf{x}_{0}) \cap \partial \Omega $. Since $\varepsilon_{0} \leqslant \min(r,a)$, there exists $\mathbf{x'} \in D_{r}(\mathbf{0'})$ such that $\mathbf{x} = (\mathbf{x'},\varphi(\mathbf{x'}))$. We introduce the notation $\mathbf{d_{x}}_{n} = (1 + \Vert \nabla \varphi(\mathbf{x'}) \Vert^{2})^{-\frac{1}{2}}$ and $\mathbf{d_{x}'} = - \mathbf{d_{x}}_{n} \nabla \varphi(\mathbf{x'})$ so that $\mathbf{d_{x}} := (\mathbf{d_{x}'},\mathbf{d_{x}}_{n})$ is a unit vector. Now, let us show that $ \Omega$ satisfy the $\varepsilon_{0}$-ball condition at the point $\mathbf{x}$ so choose any $\mathbf{y} \in B_{\varepsilon_{0}}(\mathbf{x}+ \varepsilon_{0} \mathbf{d_{x}}) \subseteq B_{2 \varepsilon_{0}}(\mathbf{x}) \subseteq B_{3 \varepsilon_{0}}(\mathbf{x}_{0})$. Since $ 3\varepsilon_{0} \leqslant \min(r,a)$, there exists $\mathbf{y'} \in D_{r}(\mathbf{0'})$ and $y_{n} \in ]-a,a[$ such that $\mathbf{y} = (\mathbf{y'}, y_{n})$. Moreover, we have $\mathbf{y} \in \mathbb{R}^{n} \backslash \overline{\Omega}$ iff $y_{n} > \varphi(\mathbf{y'})$. Observing that $\Vert \mathbf{y} - \mathbf{x} - \varepsilon_{0} \mathbf{d_{x}} \Vert < \varepsilon_{0} \Leftrightarrow \frac{1}{2 \varepsilon_{0}} \Vert \mathbf{y} - \mathbf{x} \Vert^{2} < \langle \mathbf{y} - \mathbf{x} ~\vert~ \mathbf{d_{x}} \rangle $, we obtain successively:
\[ \begin{array}{rcl}
y_{n} - \varphi(\mathbf{y'}) & = & \dfrac{1}{\mathbf{d_{x}}_{n}} \left[ \mathbf{d_{x}}_{n} \left( y_{n} - \varphi(\mathbf{x'}) \right) + \langle \mathbf{d_{x}'}  ~\vert~ \mathbf{y'} - \mathbf{x'} \rangle - \langle \mathbf{d_{x}'}  ~\vert~ \mathbf{y'} - \mathbf{x'} \rangle + \mathbf{d_{x}}_{n} \left( \varphi ( \mathbf{x'}) - \varphi(\mathbf{y'}) \right) \right] \\
& & \\
& = & \dfrac{1}{\mathbf{d_{x}}_{n}} \langle \mathbf{y} - \mathbf{x} ~\vert~ \mathbf{d_{x}} \rangle \quad - \quad  \varphi(\mathbf{y'})   \quad + \quad  \varphi(\mathbf{x'}) \quad + \quad \langle \nabla \varphi(\mathbf{x'}) ~\vert~ \mathbf{y'} - \mathbf{x'} \rangle \\
& & \\
&  > &  \dfrac{\Vert \mathbf{y} - \mathbf{x} \Vert^{2}}{2 \varepsilon_{0} \mathbf{d_{x}}_{n}} \quad  -  \quad \dfrac{L}{2} \Vert \mathbf{y'} - \mathbf{x'} \Vert^{2} \quad >  \quad \dfrac{1}{2 \mathbf{d_{x}}_{n}} \Vert \mathbf{y'} - \mathbf{x'} \Vert^{2} \left( \dfrac{1}{\varepsilon_{0}} - L \right) \quad  \geqslant \quad 0. \\
\end{array} \]
Consequently, we get $\mathbf{y} \notin \overline{\Omega}$ and we proved $B_{\varepsilon_{0}}(\mathbf{x} + \varepsilon_{0} \mathbf{d_{x}}) \subseteq \mathbb{R}^{n} \backslash \overline{\Omega}  $. Similarly, we can obtain $B_{\varepsilon_{0}}(\mathbf{x} - \varepsilon_{0} \mathbf{d_{x}}) \subseteq \Omega $. Hence, for any $\mathbf{x}_{0} \in \partial \Omega$, there exists $\varepsilon_{0} > 0$ such that $ \Omega \cap B_{\varepsilon_{0}}(\mathbf{x}_{0})$ satisfies the $\varepsilon_{0}$-ball condition. Finally, as $\partial \Omega$ is compact, it is included in a finite reunion of such balls $B_{\varepsilon_{0}}(\mathbf{x}_{0})$. Define $\varepsilon > 0$ as the minimum of this finite number of $\varepsilon_{0}$ and $ \Omega $ will satisfy the $\varepsilon$-ball property.
\end{proof}

\section{Parametrization of a converging sequence from $\mathcal{O}_{\varepsilon}(B)$} 
\label{section_parametrization_i}
In this section, we first recall a known compactness result about the uniform cone property \cite{Chenais1975}. Since we know from Point (i) of Theorem \ref{thm_boule_equiv_cunun} that every set satisfying the $\varepsilon$-ball condition also satisfies the $f^{-1}(\varepsilon)$-cone property, we only have to check that $\mathcal{O}_{\varepsilon}(B)$ is closed under the Hausdorff convergence to get its compactness. Hence, we obtain the following result.

\begin{proposition}
\label{prop_compacite_boule}
Let $\varepsilon > 0$ and $B \subset \mathbb{R}^{n}$ a bounded open set, large enough to contain an open ball of radius $3 \varepsilon$, and smooth enough so that $\partial B$ has zero $n$-dimensional Lebesgue measure. If $(\Omega_{i})_{i \in \mathbb{N}}$ is a sequence of elements from $\mathcal{O}_{\varepsilon}(B)$, then there exists $\Omega \in \mathcal{O}_{\varepsilon}(B)$ such that a subsequence $(\Omega_{\psi(i)})_{i \in \mathbb{N}}$ converges to $\Omega$ in the following sense (see Definition \ref{def_convergence} for the various modes of convergence):
\begin{itemize}
\item[(i)] $(\Omega_{\psi(i)})_{i \in \mathbb{N}}$ converges to $\Omega$ in the Hausdorff sense;
\item[(ii)] $( \partial \Omega_{\psi( i)}) _{i \in \mathbb{N}}$ converges to $\partial \Omega$ for the Hausdorff distance;
\item[(iii)] $(\overline{\Omega_{\psi(i)}})_{i \in \mathbb{N}}$ converges to $\overline{\Omega}$ for the Hausdorff distance;
\item[(iv)] $(B \backslash \overline{\Omega_{\psi(i)}})_{i \in \mathbb{N}}$ converges to $B \backslash \overline{\Omega}$ in the Hausdorff sense;
\item[(v)] $(\Omega_{\psi(i)})_{i \in \mathbb{N}}$ converges to $\Omega$ in the sense of compact sets;
\item[(vi)] $(\Omega_{\psi(i)})_{i \in \mathbb{N}}$ converges to $\Omega$ in the sense of characteristic functions.
\end{itemize} 
\end{proposition}

Then, in the rest of this section, we consider a sequence $(\Omega_{i})_{i \in \mathbb{N}} $ of elements from $ \mathcal{O}_{\varepsilon}(B)$ converging to $\Omega \in \mathcal{O}_{\varepsilon}(B)$ in the sense of Proposition \ref{prop_compacite_boule}, and we prove that locally the boundaries $\partial \Omega_{i}$ can be parametrized simultaneously by $C^{1,1}$-graphs in a fixed local frame associated with $\partial \Omega$. Finally, we get the $C^{1}$-strong and $W^{2,\infty}$-weak-star convergence of these local graphs as follows.

\begin{theorem}
\label{thm_parametrisation_locale_i}
Let $(\Omega_{i})_{i \in \mathbb{N}}\subset \mathcal{O}_{\varepsilon}(B)$ converge to $\Omega \in \mathcal{O}_{\varepsilon}(B)$ in the sense of Proposition \ref{prop_compacite_boule} (i)-(vi). Then, for any point $\mathbf{x}_{0} \in \partial \Omega$, there exists a direct orthonormal frame centred at $\mathbf{x}_{0}$, and also $I \in \mathbb{N}$ depending only on $\mathbf{x}_{0}$, $\varepsilon$, $\Omega$, and $(\Omega_{i})_{i \in \mathbb{N}}$, such that inside this frame, for any integer $i \geqslant I$, there exists a continuously differentiable map $\varphi_{i}: \overline{D_{\tilde{r}}}(\mathbf{0'}) \rightarrow ]-\varepsilon,\varepsilon[ $, whose gradient $\nabla \varphi_{i}$ and $\varphi_{i}$ are $L$-Lipschitz continuous with $L > 0 $ and $  \tilde{r} > 0$ depending only on $\varepsilon$, and such that:
\[ \left\lbrace \begin{array}{rcl}
\partial \Omega_{i} \cap \left( \overline{D_{ \tilde{r}}}(\mathbf{0'}) \cap [-\varepsilon,\varepsilon] \right) &=& \left\lbrace (\mathbf{x'},\varphi_{i}(\mathbf{x'})), \quad \mathbf{x'} \in \overline{D_{ \tilde{r}}}(\mathbf{0'}) \right\rbrace \\
& & \\
\Omega_{i} \cap \left( \overline{D_{ \tilde{r}}}(\mathbf{0'}) \cap [-\varepsilon,\varepsilon] \right) &=& \left\lbrace (\mathbf{x'},x_{n}), \quad \mathbf{x'} \in \overline{D_{ \tilde{r}}}(\mathbf{0'}) ~\mathrm{and}~ - \varepsilon \leqslant x_{n} <\varphi_{i}(\mathbf{x'}) \right\rbrace . \\
\end{array} \right. \]  
Moreover, considering the map $\varphi$ of Definition \ref{definition_regularite_cunun} associated with the point $\mathbf{x}_{0} $ of $\partial \Omega$, we have:
\begin{equation}
\label{eqn_convergence}
 \varphi_{i} \rightarrow \varphi \quad \mathrm{in} ~ C^{1}\left(\overline{D_{ \tilde{r}}}(\mathbf{0'})\right)  \qquad \mathrm{and} \qquad \varphi_{i} \rightharpoonup \varphi \quad \mathrm{weak-star}~\mathrm{in}~ W^{2,\infty}\left(D_{ \tilde{r}}(\mathbf{0'})\right). 
\end{equation} 
\end{theorem}

Hence, the rest of this section is devoted to the proof of Theorem \ref{thm_parametrisation_locale_i}, which is done in the same spirit as Sections \ref{section_boule_imply_reach}, \ref{section_parametrization_cunun}, and \ref{section_regularite_cunun}. It is organized as follows.
\begin{itemize}
\item Some global and local geometric inequalities are established. 
\item The boundary $\partial \Omega_{i}$ is locally parametrized by a graph. 
\item We obtain the $C^{1,1}$-regularity of the local graph associated with $\partial \Omega_{i}$. 
\item We prove that \eqref{eqn_convergence} holds for the local graphs.
\end{itemize}

\begin{remark}
\label{remarque_hypothese_faible_i}
Only Point (v) of Proposition \ref{prop_compacite_boule} is needed to obtain the first part of Theorem \ref{thm_parametrisation_locale_i}. To get the second part, we also need to assume Point (ii) of Proposition \ref{prop_compacite_boule}. Indeed, this hypothesis ensures that the converging sequence of local graphs converges to the one associated with $\partial \Omega$. 
\end{remark}

\subsection{Compactness of the class $\mathcal{O}_{\varepsilon}(B)$}
\label{section_compacite}
First, we quickly define the modes of convergence given in Proposition \ref{prop_compacite_boule}. Then, we state the compactness theorem associated with the uniform cone property. Finally, Proposition \ref{prop_compacite_boule} is proved.

\begin{definition}
\label{def_convergence}
The Hausdorff distance $d_{H}$ between two compact sets $X, Y \subset  \mathbb{R}^{n}$ is defined by $d_{H}(X,Y) = \max ( \max_{\mathbf{x} \in X} d(\mathbf{x},Y), \max_{\mathbf{y} \in Y} d(\mathbf{y},X) )$. We say that a sequence of compacts sets $(K_{i})_{i \in \mathbb{N}}$ converges to a compact set $K$ for the Hausdorff distance if $d_{H}(K_{i},K) \rightarrow 0$. Let $B$ be any non-empty bounded open subset of $\mathbb{R}^{n}$. A sequence of open sets $(\Omega_{i})_{i \in \mathbb{N}} \subset B$ converges to $\Omega \subset B$:
\begin{itemize}
\item in the Hausdorff sense if $(\overline{B} \backslash \Omega_{i})_{i \in \mathbb{N}}$ converges to $\overline{B} \backslash \Omega$ for the Hausdorff distance;
\item in the sense of compact sets if for any compact sets $K$ and $ L $ such that $K \subset \Omega$ and $L \subset B \backslash \overline{\Omega}$, there exists $I \in \mathbb{N}$ such that for any integer $i \geqslant I$, we have $K \subset \Omega_{i}$ and $L \subset B \backslash \overline{\Omega_{i}}$;
\item in the sense of characteristic functions if we have $ \int_{\overline{B}} \vert \mathbf{1}_{\Omega_{i}}(\mathbf{x}) - \mathbf{1}_{\Omega}(\mathbf{x}) \vert d \mathbf{x} \rightarrow 0$, where $\mathbf{1}_{X}$ is the characteristic function of $X$, valued one for the points of $X$, otherwise zero.
\end{itemize}
\end{definition}

In \cite[Theorem 2.8]{GuoYang2012}, Point (i) of Proposition \ref{prop_compacite_boule} is proved. However, we can prove Proposition \ref{prop_compacite_boule} by applying Theorem \ref{thm_boule_equiv_cunun} (i) and the following result.

\begin{theorem}[\textbf{Chenais \cite[Theorem 2.4.10]{HenrotPierre}}]
\label{thm_compacite_cone}
Let $\alpha \in ]0,\frac{\pi}{2}[$ and $B$ be as in Proposition \ref{prop_compacite_boule}. We set $\mathfrak{O}_{\alpha}(B)$ the class of non-empty open sets $\Omega \subseteq B$ that satisfy the $ \alpha$-cone property as in Definition \ref{definition_alpha_cone}. If $(\Omega_{i})_{i \in \mathbb{N}}$ is a sequence of elements from $\mathfrak{O}_{\alpha}(B)$, then there exists $\Omega \in \mathfrak{O}_{\alpha}(B)$ such that a subsequence $(\Omega_{\psi(i)})_{i \in \mathbb{N}}$ converges to $\Omega$ in the sense of Proposition \ref{prop_compacite_boule} (i)-(vi).
\end{theorem}  

\begin{proof}
We only sketch the proof and refer to \cite[Theorem 2.4.10]{HenrotPierre} for further details. First, consider any $(\Omega_{i})_{i \in \mathbb{N}} \subset B$ and show that, up to a subsequence, it is converging to $\Omega \subset B$ in the Hausdorff sense. Then, use the uniform cone condition to get $\Omega \in \mathfrak{O}_{\alpha}(B)$ and $\lim_{i \rightarrow + \infty} d_{H}(\partial \Omega_{i},\partial \Omega) = \lim_{i \rightarrow + \infty} d_{H}(\overline{ \Omega_{i}},\overline{ \Omega}) = 0$. Next, deduce that $(B \backslash \overline{\Omega_{i}})_{i  \in \mathbb{N}}$ converges to $B \backslash \overline{\Omega}$ in the Hausdorff sense, and $(\Omega_{i})_{i \in \mathbb{N}}$ to $\Omega$ in the sense of compact sets. Finally, since $\Omega \in \mathfrak{O}_{\alpha}(B)$, $\partial \Omega$ is a finite reunion of Lipschitz graphs so it has zero $n$-dimensional Lebesgue measure \cite[Section 2.4.2 Theorem 2]{EvansGariepy} and so does $\partial B$ by assumption. Combining this observation with the convergence in the sense of compacts, we obtain the convergence in the sense of characteristic functions.   
\end{proof}

\begin{proof}[\textbf{Proof of Proposition \ref{prop_compacite_boule}}]
Since $\mathcal{O}_{\varepsilon}(B) \subset \mathfrak{O}_{f^{-1}(\varepsilon)}(B)$ (Point (i) of Theorem \ref{thm_boule_equiv_cunun}), Theorem \ref{thm_compacite_cone} holds and we only have to check $ \Omega \in \mathcal{O}_{\varepsilon}(B)$. Consider any $\mathbf{x} \in \partial \Omega$. From \cite[Proposition 2.2.14]{HenrotPierre}, there exists a sequence of points $\mathbf{x}_{i} \in \partial \Omega_{i}$ converging to $\mathbf{x}$. Then, we can apply the $\varepsilon$-ball condition on each point $\mathbf{x}_{i}$ so there exists a sequence of unit vector $\mathbf{d}_{\mathbf{x}_{i}}$ of $\mathbb{R}^{n}$ such that:
\[ \forall i \in \mathbb{N}, \quad \left\lbrace \begin{array}{l}
B_{\varepsilon}(\mathbf{x}_{i} - \varepsilon \mathbf{d}_{\mathbf{x}_{i}}) \subseteq \Omega_{i} \\
\\
B_{\varepsilon}(\mathbf{x}_{i} + \varepsilon \mathbf{d}_{\mathbf{x}_{i}}) \subseteq B \backslash \overline{\Omega_{i}}. \\
\end{array} \right. \] 
Since $ \Vert \mathbf{d}_{\mathbf{x}_{i}} \Vert =1$, there exists a unit vector $\mathbf{d}_{\mathbf{x}}$ of $\mathbb{R}^{n}$ such that, up to a subsequence, $(\mathbf{d}_{\mathbf{x}_{i}})_{i \in \mathbb{N}}$ converges to $\mathbf{d_{x}}$. Finally, the inclusion is stable under the Hausdorff convergence \cite[(2.16)]{HenrotPierre} and we get the $\varepsilon$-ball condition of Definition \ref{definition_epsilon_boule} by letting $i \rightarrow + \infty $ in the above inclusions.
\end{proof}

\subsection{Some global and local geometric inequalities}
In the rest of Section \ref{section_parametrization_i}, we consider a sequence $  (\Omega_{i})_{i \in \mathbb{N}} \subset \mathcal{O}_{\varepsilon}(B)$ converging to $\Omega \in \mathcal{O}_{\varepsilon}(B)$ in the sense of Proposition \ref{prop_compacite_boule} (i)-(vi) and we make the following hypothesis.

\begin{assumption}
\label{hypotheses_parametrization_i}
Let $ \mathbf{x}_{0} \in \partial \Omega $ henceforth set. From the $\varepsilon$-ball condition, a unit vector $\mathbf{d}_{\mathbf{x}_{0}}$ is associated with the point $\mathbf{x}_{0}$ (which is unique from Proposition \ref{prop_normale_lipschitzienne}). Moreover, we have:
\[ \left\lbrace \begin{array}{l}
B_{\varepsilon}(\mathbf{x}_{0} - \varepsilon \mathbf{d}_{\mathbf{x}_{0}}) \subseteq \Omega \\
 \\
B_{\varepsilon}(\mathbf{x}_{0} + \varepsilon \mathbf{d}_{\mathbf{x}_{0}}) \subseteq B \backslash \overline{\Omega} . \\
\end{array} \right. \]
Then, we also consider $\eta \in ]0,\varepsilon[$. Since we assume Point (v) of Proposition \ref{prop_compacite_boule}, there exists $I \in \mathbb{N}$ depending on $(\Omega_{i})_{i \in \mathbb{N}}$, $\Omega$, $\mathbf{x}_{0}$, $\varepsilon$ and $\eta$, such that for any integer $i \geqslant I$, we have:
\begin{equation}
\label{expression_condition_epsilon_boule_xo_omega_i}
\left\lbrace \begin{array}{l}
\overline{B_{\varepsilon - \eta}}(\mathbf{x}_{0} - \varepsilon \mathbf{d}_{\mathbf{x}_{0}}) \subseteq \Omega_{i} \\
\\
\overline{B_{\varepsilon - \eta}}(\mathbf{x}_{0} + \varepsilon \mathbf{d}_{\mathbf{x}_{0}}) \subseteq B \backslash \overline{ \Omega_{i} }. \\
\end{array} \right. 
\end{equation}
Finally, we consider any integer $i \geqslant I$.
\end{assumption}

\begin{proposition}
\label{prop_inegalite_normale_lipschitzienne_i}
Assume \eqref{expression_condition_epsilon_boule_xo_omega_i}. For any point $\mathbf{x}_{i} \in \partial \Omega_{i}$, we have the following inequality:
\begin{equation}
\label{eqn_a}
\Vert \mathbf{d}_{\mathbf{x}_{i}} - \mathbf{d}_{\mathbf{x}_{0}} \Vert^{2} \leqslant \dfrac{1}{\varepsilon^{2}} \Vert \mathbf{x}_{i} - \mathbf{x}_{0} \Vert^{2} + \dfrac{(2 \varepsilon)^{2} - (2\varepsilon - \eta)^{2}}{\varepsilon^{2}}. 
\end{equation}
\end{proposition}

\begin{proof}
Combine \eqref{expression_condition_epsilon_boule_xo_omega_i} with the $\varepsilon$-ball condition at $\mathbf{x}_{i} \in \partial \Omega_{i}$ to get $\overline{B_{\varepsilon - \eta}}(\mathbf{x}_{0} \pm \varepsilon \mathbf{d}_{\mathbf{x}_{0}}) \cap B_{\varepsilon}(\mathbf{x}_{i} \mp \varepsilon \mathbf{d}_{\mathbf{x}_{i}}) = \emptyset $. We deduce $\Vert \mathbf{x}_{i} - \mathbf{x}_{0} \mp \varepsilon (\mathbf{d}_{\mathbf{x}_{i}} + \mathbf{d}_{\mathbf{x}_{0}}) \Vert \geqslant 2 \varepsilon - \eta  $. Squaring these two inequalities and summing them, we obtain the required one: $ \Vert \mathbf{x}_{i} - \mathbf{x}_{0} \Vert^{2} + 4 \varepsilon^{2} - (2 \varepsilon - \eta)^{2} \geqslant 2 \varepsilon^{2} - 2 \varepsilon^{2} \langle \mathbf{d}_{\mathbf{x}_{i}} ~\vert~ \mathbf{d}_{\mathbf{x}_{0}} \rangle = \varepsilon^{2} \Vert \mathbf{d}_{\mathbf{x}_{i}} - \mathbf{d}_{\mathbf{x}_{0}} \Vert^{2} $.
\end{proof}

\begin{proposition}
\label{prop_inegalites_locale_globale_i}
Under assumption \ref{hypotheses_parametrization_i}, for any $\mathbf{x}_{i} \in \partial \Omega_{i}$, we have the following global inequality:
\begin{equation}
\label{eqn_d}
\vert \langle \mathbf{x}_{i} - \mathbf{x}_{0} ~\vert~ \mathbf{d}_{\mathbf{x}_{0}} \rangle \vert < \dfrac{1}{2\varepsilon} \Vert \mathbf{x}_{i} - \mathbf{x}_{0} \Vert^{2} + \dfrac{\varepsilon^{2} - (\varepsilon - \eta)^{2}}{2\varepsilon}.   
\end{equation}
Moreover, if we introduce the vector $(\mathbf{x}_{i} - \mathbf{x}_{0})' = (\mathbf{x}_{i} - \mathbf{x}_{0}) - \langle \mathbf{x}_{i} - \mathbf{x}_{0} ~\vert~ \mathbf{d}_{\mathbf{x}_{0}} \rangle \mathbf{d}_{\mathbf{x}_{0}} $ and if we assume that $\Vert (\mathbf{x}_{i} - \mathbf{x}_{0})' \Vert \leqslant \varepsilon - \eta  $ and $\vert \langle \mathbf{x}_{i} - \mathbf{x}_{0} ~\vert~ \mathbf{d}_{\mathbf{x}_{0}} \rangle \vert \leqslant \varepsilon$, then we have the following local inequality:
\begin{equation}
\label{eqn_b}
 \dfrac{1}{2\varepsilon} \Vert \mathbf{x}_{i} - \mathbf{x}_{0} \Vert^{2} + \dfrac{\varepsilon^{2} - (\varepsilon - \eta)^{2}}{2\varepsilon} < \varepsilon - \sqrt{(\varepsilon - \eta)^{2} - \Vert (\mathbf{x}_{i} - \mathbf{x}_{0})' \Vert^{2}}.
\end{equation} 
\end{proposition}

\begin{proof}
From \eqref{expression_condition_epsilon_boule_xo_omega_i}, any point $\mathbf{x}_{i} \in \partial \Omega_{i}$ cannot belong to the sets $\overline{B_{\varepsilon - \eta}}(\mathbf{x}_{0} \pm \varepsilon \mathbf{d}_{\mathbf{x}_{0}}) $. Hence, we have: $\Vert \mathbf{x}_{i} - \mathbf{x}_{0} \mp \varepsilon \mathbf{d}_{\mathbf{x}_{0}} \Vert > \varepsilon - \eta $. Squaring these two inequalities, we get the first required relation \eqref{eqn_d}: $\Vert \mathbf{x}_{i} - \mathbf{x}_{0} \Vert^{2} + \varepsilon^{2} - (\varepsilon - \eta)^{2} > 2 \varepsilon \vert \langle \mathbf{x}_{i} - \mathbf{x}_{0} ~\vert~ \mathbf{d}_{\mathbf{x}_{0}} \rangle \vert $. Then, by introducing the vector $( \mathbf{x}_{i} - \mathbf{x}_{0} )'$ of the statement, the previous inequality now takes the following form:
\[ \vert \langle \mathbf{x}_{i} - \mathbf{x}_{0} ~\vert~ \mathbf{d}_{\mathbf{x}_{0}} \rangle \vert^{2} - 2 \varepsilon \vert \langle \mathbf{x}_{i} - \mathbf{x}_{0} ~\vert~ \mathbf{d}_{\mathbf{x}_{0}} \rangle \vert + \Vert ( \mathbf{x}_{i} - \mathbf{x}_{0})' \Vert^{2} + \varepsilon^{2} - (\varepsilon-\eta)^{2} > 0.  \] 
We assume that its left member is a second-order polynomial whose discriminant is non-negative: $\Delta' := (\varepsilon - \eta)^{2} - \Vert ( \mathbf{x}_{i} - \mathbf{x}_{0})' \Vert^{2} \geqslant 0 $. Hence, the unknown satisfies either $ \vert \langle \mathbf{x}_{i} - \mathbf{x}_{0} ~\vert~ \mathbf{d}_{\mathbf{x}_{0}} \rangle \vert < \varepsilon - \sqrt{\Delta '}$ or $ \vert \langle \mathbf{x}_{i} - \mathbf{x}_{0} ~\vert~ \mathbf{d}_{\mathbf{x}_{0}} \rangle \vert > \varepsilon + \sqrt{\Delta '}$. We assume $ \vert \langle \mathbf{x}_{i} - \mathbf{x}_{0} ~\vert~ \mathbf{d}_{\mathbf{x}_{0}} \rangle \vert \leqslant \varepsilon$ and the last case cannot hold. Squaring the remaining inequality, we get: $  \vert \langle \mathbf{x}_{i} - \mathbf{x}_{0} ~\vert~ \mathbf{d}_{\mathbf{x}_{0}} \rangle \vert^{2} +  \Vert ( \mathbf{x}_{i} - \mathbf{x}_{0})' \Vert^{2}  < \varepsilon^{2} + (\varepsilon - \eta)^{2} - 2 \varepsilon \sqrt{\Delta '} $, which is the second required relation \eqref{eqn_b} since its left member is equal to $\Vert \mathbf{x}_{i} - \mathbf{x}_{0} \Vert^{2}$.
\end{proof}

\begin{corollary}
\label{coro_inegalites_i}
With the same assumptions and notation as in Propositions \ref{prop_inegalite_normale_lipschitzienne_i} and \ref{prop_inegalites_locale_globale_i}, we have:
\begin{gather}
\Vert \mathbf{x}_{i} - \mathbf{x}_{0} \Vert < 2 \eta + 2 \Vert (\mathbf{x}_{i} - \mathbf{x}_{0})' \Vert , \label{eqn_h} \\
\varepsilon \Vert \mathbf{d}_{\mathbf{x}_{i}} - \mathbf{d}_{\mathbf{x}_{0}} \Vert < 2 \sqrt{2 \varepsilon \eta} + \sqrt{2} \Vert (\mathbf{x}_{i} - \mathbf{x}_{0})' \Vert . \label{eqn_g}
\end{gather} 
\end{corollary}

\begin{proof}
Consider any $\mathbf{x}_{i} \in \partial \Omega_{i}$. We set  $(\mathbf{x}_{i} - \mathbf{x}_{0})' = (\mathbf{x}_{i} - \mathbf{x}_{0}) - \langle \mathbf{x}_{i} - \mathbf{x}_{0} ~\vert~ \mathbf{d}_{\mathbf{x}_{0}} \rangle \mathbf{d}_{\mathbf{x}_{0}} $. We assume $\Vert (\mathbf{x}_{i} - \mathbf{x}_{0})' \Vert \leqslant \varepsilon - \eta  $ and $\vert \langle \mathbf{x}_{i} - \mathbf{x}_{0} ~\vert~ \mathbf{d}_{\mathbf{x}_{0}} \rangle \vert \leqslant \varepsilon$. The local estimation \eqref{eqn_b} of Proposition \ref{prop_inegalites_locale_globale_i} gives: 
\[ \begin{array}{rcl}
\Vert \mathbf{x}_{i} - \mathbf{x}_{0} \Vert^{2} & < & \varepsilon^{2} + (\varepsilon - \eta)^{2} - 2 \varepsilon \sqrt{(\varepsilon - \eta)^{2} - \Vert (\mathbf{x}_{i} - \mathbf{x}_{0})' \Vert^{2}} \\
& & \\
& = & \dfrac{\left[ \varepsilon^{2} + \left( \varepsilon - \eta \right)^{2} \right]^{2} - 4 \varepsilon^{2}(\varepsilon - \eta)^{2} + 4 \varepsilon^{2}\Vert (\mathbf{x}_{i} - \mathbf{x}_{0})' \Vert^{2} }{\varepsilon^{2} + (\varepsilon - \eta)^{2} + 2 \varepsilon \sqrt{(\varepsilon - \eta)^{2} - \Vert (\mathbf{x}_{i} - \mathbf{x}_{0})' \Vert^{2}} } \\
 & & \\
  & < & \left[ \dfrac{\varepsilon^{2} - (\varepsilon - \eta)^{2}}{\varepsilon} \right]^{2} + 4 \Vert (\mathbf{x}_{i} - \mathbf{x}_{0})' \Vert^{2} \quad <  \quad 4 \eta^{2} + 4  \Vert (\mathbf{x}_{i} - \mathbf{x}_{0})' \Vert^{2}. \\
\end{array} \]
Hence, we get: $ \Vert \mathbf{x}_{i} - \mathbf{x}_{0} \Vert <  2 \eta + 2  \Vert (\mathbf{x}_{i} - \mathbf{x}_{0})' \Vert$. Then, using \eqref{eqn_a}, we also have:
\[ \varepsilon \Vert \mathbf{d}_{\mathbf{x}_{i}} - \mathbf{d}_{\mathbf{x}_{0}} \Vert \leqslant \sqrt{ 4\varepsilon^{2} - ( 2\varepsilon - \eta)^{2} + \Vert \mathbf{x}_{i} - \mathbf{x}_{0} \Vert^{2}}. \]
Combining the above inequality with \eqref{eqn_b}, we obtain:
\[ \begin{array}{rcl}
\varepsilon \Vert \mathbf{d}_{\mathbf{x}_{i}} - \mathbf{d}_{\mathbf{x}_{0}} \Vert  & < & \sqrt{ 4 \varepsilon \eta - \eta^{2} + \varepsilon^{2} + (\varepsilon - \eta)^{2} - 2 \varepsilon \sqrt{(\varepsilon - \eta)^{2} - \Vert (\mathbf{x}_{i} - \mathbf{x}_{0})' \Vert^{2}}  } \\
 & & \\
  & = & \sqrt{2 \varepsilon \dfrac{ 4 \varepsilon \eta + \Vert (\mathbf{x}_{i} - \mathbf{x}_{0})' \Vert^{2}}{ \varepsilon + \eta + \sqrt{(\varepsilon - \eta)^{2} -\Vert (\mathbf{x}_{i} - \mathbf{x}_{0})' \Vert^{2}}}} \quad < \quad  2 \sqrt{2 \varepsilon \eta} + \sqrt{2}\Vert (\mathbf{x}_{i} - \mathbf{x}_{0})' \Vert. \\
\end{array} \]
Consequently, the two required inequalities \eqref{eqn_h} and \eqref{eqn_g} are established so Corollary \ref{coro_inegalites_i} holds. 
\end{proof}

\subsection{A local parametrization of the boundary $\partial \Omega_{i}$}
Henceforth, we consider a basis $\mathcal{B}_{\mathbf{x}_{0}}$ of the hyperplane $\mathbf{d}_{\mathbf{x}_{0}}^{\perp}$ such that $(  \mathbf{x}_{0}, \mathcal{B}_{\mathbf{x}_{0}}, \mathbf{d}_{\mathbf{x}_{0}} )$ is a direct orthonormal frame. The position of any point is now determined in this local frame associated with $\mathbf{x}_{0}$. More precisely, for any point $\mathbf{x} \in \mathbb{R}^{n}$, we set $\mathbf{x'} = (x_{1}, \ldots, x_{n-1}) $ such that $\mathbf{x} = (\mathbf{x'},x_{n})$. In particular, the symbols $\mathbf{0}$ and $\mathbf{0'}$ respectively refer to the zero vector of $\mathbb{R}^{n}$ and $\mathbb{R}^{n-1}$. Moreover, since $\mathbf{x}_{0}$ is identified with $\mathbf{0}$ in this new frame, Relations \eqref{expression_condition_epsilon_boule_xo_omega_i} of Assumption \ref{hypotheses_parametrization_i} take new forms:
\begin{equation}
\label{eqn_c}
 \left\lbrace \begin{array}{l}
\overline{B_{\varepsilon - \eta}}(\mathbf{0'},-\varepsilon) \subseteq \Omega_{i} \\
\\
\overline{B_{\varepsilon - \eta}}(\mathbf{0'},\varepsilon) \subseteq B \backslash \overline{ \Omega_{i}}. \\
\end{array} \right. 
\end{equation}
We introduce two functions defined on $\overline{D_{\varepsilon - \eta}}(\mathbf{0'}) = \lbrace \mathbf{x'} \in \mathbb{R}^{n-1}, ~\Vert \mathbf{x'} \Vert \leqslant \varepsilon - \eta \rbrace  $. The first one determine around $\mathbf{x}_{0}$ the position of the boundary $\partial \Omega_{i}$ thanks to some exterior points, the other one with interior points. Then, we show these two maps coincide even if it means reducing $\eta$. 

\begin{proposition}
\label{prop_parametrisations_locales_i}
Under Assumption \ref{hypotheses_parametrization_i}, the two following maps $\varphi^{\pm}_{i}$ are well defined:
\[ \left\lbrace \begin{array}{rcl}
\varphi_{i}^{+} :~ \overline{D_{\varepsilon - \eta}}(\mathbf{0'}) & \longrightarrow & ]- \varepsilon,\varepsilon[ \\
\mathbf{x'} & \longmapsto & \varphi_{i}^{+}(\mathbf{x'}) = \sup \lbrace x_{n} \in [- \varepsilon,\varepsilon], \quad (\mathbf{x'},x_{n}) \in \Omega_{i} \rbrace \\
 & & \\ 
\varphi_{i}^{-} :~ \overline{D_{\varepsilon - \eta}}(\mathbf{0'}) & \longrightarrow & ]- \varepsilon,\varepsilon[ \\
\mathbf{x'} & \longmapsto & \varphi_{i}^{-}(\mathbf{x'}) = \inf \lbrace x_{n} \in [- \varepsilon,\varepsilon], \quad (\mathbf{x'},x_{n}) \in B \backslash \overline{\Omega_{i}} \rbrace, \\
\end{array} \right.  \]
Moreover, for any $\mathbf{x'} \in \overline{D_{\varepsilon - \eta}}(\mathbf{0'})$, introducing the points $\mathbf{x}_{i}^{\pm} = (\mathbf{x'},\varphi_{i}^{\pm}(\mathbf{x'}))$, we have $\mathbf{x}_{i}^{\pm} \in \partial \Omega_{i}$ and also the following inequalities:
\begin{equation}
\label{eqn_e}
 \vert  \varphi_{i}^{\pm}(\mathbf{x'}) \vert < \dfrac{1}{2 \varepsilon} \Vert \mathbf{x}_{i}^{\pm} - \mathbf{x}_{0} \Vert^{2} + \dfrac{\varepsilon^{2}- (\varepsilon - \eta)^{2}}{2 \varepsilon} < \varepsilon - \sqrt{(\varepsilon - \eta)^{2}- \Vert \mathbf{x'}\Vert^{2}}. 
 \end{equation}
\end{proposition}

\begin{proof}
Let $ \mathbf{x'} \in \overline{D_{\varepsilon- \eta}}(\mathbf{0'}) $ and $g: t \in [-\varepsilon,\varepsilon] \mapsto (\mathbf{x'},t)$. Since $-\varepsilon \in g^{-1}(\Omega_{i}) \subseteq [-\varepsilon,\varepsilon] $, we can set $\varphi_{i}^{+}(\mathbf{x'}) = \sup g^{-1}(\Omega_{i})$. The map $g$ is continuous so $g^{-1}(\Omega_{i})$ is open and $\varphi_{i}^{+}(\mathbf{x'}) \neq \varepsilon$ thus we get $\varphi_{i}^{+}(\mathbf{x'}) \notin g^{-1}(\Omega_{i}) $ i.e. $\mathbf{x}_{i}^{+} \in \overline{\Omega_{i}} \backslash \Omega_{i} $. Similarly, the map $ \varphi_{i}^{-}$ is well defined and $\mathbf{x}_{i}^{-} \in \partial \Omega_{i}$. Finally, we use \eqref{eqn_d} and \eqref{eqn_b} on the points $ \mathbf{x}_{0}$ and $\mathbf{x}_{i} = \mathbf{x}_{i}^{\pm}$ in order to obtain \eqref{eqn_e}.
\end{proof}

\begin{lemma}
\label{lemme_parametrisation_locale_i}
We make Assumption \ref{hypotheses_parametrization_i} and assume $\eta < \frac{\varepsilon}{3} $. We set $r = \frac{1}{2} \sqrt{4 (\varepsilon - \eta)^{2} - (\varepsilon + \eta)^{2}}$ and $\mathbf{x'} \in \overline{D_{r}}(\mathbf{0'})$. Assume there exists $x_{n} \in [- \varepsilon, \varepsilon]$ such that $\mathbf{x}_{i} := (\mathbf{x'},x_{n})$ belongs to $\partial \Omega_{i}$. We also consider $\tilde{x}_{n} \in  \mathbb{R}$ satisfying the inequality $\vert \tilde{x}_{n} \vert < \varepsilon - \sqrt{(\varepsilon - \eta)^{2} - \Vert \mathbf{x'} \Vert^{2}}$. Introducing $\mathbf{\tilde{x}}_{i} = (\mathbf{x'},\tilde{x}_{n})$, then we have: $ \left( \tilde{x}_{n} < x_{n} \Longrightarrow \mathbf{\tilde{x}}_{i} \in \Omega_{i} \right) ~\mathrm{and}~  \left( \tilde{x}_{n} > x_{n}  \Longrightarrow  \mathbf{\tilde{x}}_{i} \in B \backslash \overline{\Omega_{i}} \right) $.
\end{lemma}

\begin{proof}
We assume $\eta < \frac{\varepsilon}{3}$ so we can set $r= \frac{1}{2} \sqrt{4(\varepsilon - \eta)^{2} - (\varepsilon + \eta)^{2}}$. Consider any $\mathbf{x'} \in \overline{D_{r}}(\mathbf{0'})$ and also $(x_{n},\tilde{x}_{n}) \in [- \varepsilon, \varepsilon]^{2}$ such that $\mathbf{x}_{i} := (\mathbf{x'},x_{n}) \in\partial \Omega_{i}$ and $\mathbf{\tilde{x}}_{i} := (\mathbf{x'},\tilde{x}_{n}) \notin \overline{B_{\varepsilon - \eta}}(\mathbf{0'},\pm \varepsilon)$. We need to show that if $\tilde{x}_{n} \gtrless x_{n}$, then $\mathbf{\tilde{x}}_{i} \in B_{\varepsilon}(\mathbf{x}_{i} \pm \varepsilon \mathbf{d}_{\mathbf{x}_{i}})$. The $\varepsilon$-ball condition on $\Omega_{i}$ will give the result. Since $\mathbf{x}_{i} - \mathbf{\tilde{x}}_{i} = (x_{n} - \tilde{x}_{n}) \mathbf{d}_{\mathbf{x}_{0}}$, if we assume $\tilde{x}_{n} > x_{n}$, then we have:
\[ \begin{array}{rcl}
\Vert \mathbf{\tilde{x}}_{i} - \mathbf{x}_{i}  - \varepsilon \mathbf{d}_{\mathbf{x}_{i}} \Vert^{2} - \varepsilon^{2}  &=&  (\tilde{x}_{n} - x_{n})^{2} - 2 \varepsilon (\tilde{x}_{n} - x_{n}) \langle \mathbf{d}_{\mathbf{x}_{0}} ~\vert~ \mathbf{d}_{\mathbf{x}_{i}} \rangle \\
& & \\
& = & \vert \tilde{x}_{n} - x_{n} \vert \left( \vert \tilde{x}_{n} - x_{n} \vert + \varepsilon \Vert \mathbf{d}_{\mathbf{x}_{i}} - \mathbf{d}_{\mathbf{x}_{0}} \Vert^{2} - 2 \varepsilon \right) \\
 & & \\
 & \leqslant &  \vert \tilde{x}_{n} - x_{n} \vert \left( \vert \tilde{x}_{n} \vert + \vert x_{n} \vert + \frac{\Vert \mathbf{x}_{i} - \mathbf{x}_{0} \Vert^{2} + (2 \varepsilon)^{2} - (2 \varepsilon - \eta)^{2}}{\varepsilon} - 2 \varepsilon \right), \\
\end{array} \]
where the last inequality comes from Proposition \ref{prop_inegalite_normale_lipschitzienne_i} \eqref{eqn_a} applied to $\mathbf{x}_{i} \in \partial \Omega_{i}$. Finally, we use the inequality involving $\tilde{x}_{n}$ and the ones \eqref{eqn_d}-\eqref{eqn_b} of Proposition \ref{prop_inegalites_locale_globale_i} applied to $\mathbf{x}_{i} \in \partial \Omega_{i}$ to obtain:
\[ \Vert \mathbf{\tilde{x}}_{i} - \mathbf{x}_{i}  - \varepsilon \mathbf{d}_{\mathbf{x}_{i}} \Vert^{2} - \varepsilon^{2} <  4 \vert x_{n} - \tilde{x}_{n} \vert \underbrace{ \left( \dfrac{\varepsilon + \eta}{2} - \sqrt{(\varepsilon - \eta)^{2} - \Vert \mathbf{x'} \Vert^{2}} \right) }_{ \leqslant \left( \frac{\varepsilon + \eta}{2} - \sqrt{(\varepsilon - \eta)^{2} -r^{2}} \right)~ = ~0 }. \] 
Hence, if $\tilde{x}_{n} > x_{n}$, then we get $\mathbf{\tilde{x}}_{i} \in B_{\varepsilon}(\mathbf{x}_{i} + \varepsilon \mathbf{d}_{\mathbf{x}_{i}}) \subseteq B \backslash \overline{ \Omega_{i}}$. Similarly, one can prove that if $\tilde{x}_{n} < x_{n}$, then we have $\mathbf{\tilde{x}}_{i} \in B_{\varepsilon}(\mathbf{x}_{i} - \varepsilon \mathbf{d}_{\mathbf{x}_{i}}) \subseteq \Omega_{i}$. 
\end{proof}

\begin{proposition}
\label{prop_parametrisation_locale_i}
Let $\eta $, $r$ be as in Lemma \ref{lemme_parametrisation_locale_i}. Then, the two functions $\varphi_{i}^{\pm}$ of Proposition \ref{prop_parametrisations_locales_i} coincide on $\overline{D_{r}}(\mathbf{0'})$. The map $\varphi_{i}$ refers to their common restrictions and it satisfies:
\[ \left\lbrace \begin{array}{rcl}
\partial \Omega_{i} \cap \left( \overline{D_{r}}(\mathbf{0'}) \cap [-\varepsilon,\varepsilon] \right) &=& \left\lbrace (\mathbf{x'},\varphi_{i}(\mathbf{x'})), \quad \mathbf{x'} \in \overline{D_{r}}(\mathbf{0'}) \right\rbrace \\
& & \\
\Omega_{i} \cap \left( \overline{D_{r}}(\mathbf{0'}) \cap [-\varepsilon,\varepsilon] \right) &=& \left\lbrace (\mathbf{x'},x_{n}), \quad \mathbf{x'} \in \overline{D_{r}}(\mathbf{0'}) ~\mathrm{and}~ - \varepsilon \leqslant x_{n} <\varphi_{i}(\mathbf{x'}) \right\rbrace. \\
\end{array} \right. \]
\end{proposition}

\begin{proof}
First, we assume by contradiction that there exists $ \mathbf{x'} \in \overline{D_{r}}(\mathbf{0'}) $ such that $ \varphi_{i}^{-}(\mathbf{x'}) \neq \varphi_{i}^{+}(\mathbf{x'}) $. The hypothesis of Lemma \ref{lemme_parametrisation_locale_i} are satisfied for the points $ \mathbf{x}_{i}:=(\mathbf{x'},\varphi_{i}^{+}(\mathbf{x'})) $ and $ \mathbf{\tilde{x}}_{i}:=(\mathbf{x'},\varphi^{-}_{i}(\mathbf{x'})) $ by using \eqref{eqn_e}. Hence, either $ (\varphi_{i}^{-}(\mathbf{x'}) < \varphi_{i}^{+}(\mathbf{x'}) \Rightarrow  \mathbf{\tilde{x}}_{i} \in \Omega_{i} ) $ or $ (\varphi_{i}^{-}(\mathbf{x'}) > \varphi_{i}^{+}(\mathbf{x'}) \Rightarrow  \mathbf{\tilde{x}}_{i} \in B \backslash \overline{\Omega}_{i}) $ whereas $ \mathbf{\tilde{x}}_{i} \in \partial \Omega_{i} $. We deduce that $ \varphi_{i}^{-}(\mathbf{x'}) = \varphi_{i}^{+}(\mathbf{x'}) $ for any $\mathbf{x'} \in \overline{D_{r}}(\mathbf{0'})$. Then, we consider $ \mathbf{x'} \in \overline{D_{r}}(\mathbf{0'}) $ and $ x_{n} \in [-\varepsilon,\varepsilon]$. We set $ \mathbf{x}_{i} = (\mathbf{x'},\varphi_{i}(\mathbf{x'}))  $ and $\mathbf{\tilde{x}}_{i} = (\mathbf{x'},x_{n})$. Proposition \ref{prop_parametrisations_locales_i} ensures that if $x_{n} = \varphi_{i}(\mathbf{x'})$, then $\mathbf{x}_{i} \in \partial \Omega_{i}$. Moreover, if $ - \varepsilon \leqslant x_{n} \leqslant -\varepsilon + \sqrt{(\varepsilon- \eta)^{2}- \Vert \mathbf{x'} \Vert^{2}}$, then $\mathbf{\tilde{x}}_{i} \in \overline{B_{\varepsilon - \eta}}(\mathbf{0'},-\varepsilon) \subseteq \Omega_{i} $ and if $ - \varepsilon + \sqrt{(\varepsilon-  \eta)^{2} - \Vert \mathbf{x'} \Vert^{2}} < x_{n} < \varphi(\mathbf{x'}) $, then apply Lemma \ref{lemme_parametrisation_locale_i} in order to get $\mathbf{\tilde{x}}_{i} \in \Omega_{i}$. Consequently, we proved: $ \forall \mathbf{x'} \in \overline{D_{r}}(\mathbf{0'}),~ - \varepsilon \leqslant x_{n} <  \varphi_{i}(\mathbf{x'}) \Longrightarrow  (\mathbf{x'},x_{n}) \in \Omega_{i} $. To conclude, similar arguments hold when $ \varepsilon \geqslant x_{n} > \varphi_{i}(\mathbf{x'})$ and imply $(\mathbf{x'},x_{n}) \in B \backslash \overline{\Omega_{i}}$.
\end{proof}

\subsection{The $C^{1,1}$-regularity of the local graph $\varphi_{i}$}
We previously showed that the boundary $\partial \Omega_{i}$ is locally described by the graph of a well-defined map $\varphi_{i} : \overline{D_{r}}(\mathbf{0'}) \rightarrow ]- \varepsilon, \varepsilon[$. Now we prove its $C^{1,1}$-regularity even if it means reducing $\eta$ and $r$.

\begin{lemma}
\label{lemme_alpha_cone_i}
The following map is well defined, smooth, surjective and increasing:
\[ \begin{array}{rcl}
f_{\eta}: \quad ]0,\frac{\pi}{2}[ & \longrightarrow & ]2\sqrt{2 \varepsilon \eta}, + \infty[ \\
\alpha & \longmapsto & \dfrac{3 \alpha + 2 \sqrt{2 \varepsilon \eta}}{\cos \alpha}. \\
\end{array} \]
In particular, it is an homeomorphism and its inverse $f^{-1}_{\eta}$ satisfies the following inequality:
\begin{equation}
\label{eqn_f}
\forall \varepsilon > 0,~ \forall \eta \in \left] 0,\dfrac{\varepsilon}{8}\right[, \quad   f_{\eta}^{-1}(\varepsilon) < \dfrac{\varepsilon}{3}.
\end{equation}
\end{lemma}

\begin{proof}
The proof is basic calculus.
\end{proof}

\begin{proposition}
\label{prop_alpha_cone_i}
In Assumption \ref{hypotheses_parametrization_i}, let $\eta < \frac{\varepsilon}{8}$ and consider $\alpha \in ]0,f^{-1}_{\eta}(\varepsilon)]$, where $f^{-1}_{\eta}$ has been introduced in Lemma \ref{lemme_alpha_cone_i}. Then, we have:
\[ \forall \mathbf{x}_{i} \in B_{\alpha}(\mathbf{x}_{0}) \cap \overline{\Omega_{i}}, \quad C_{\alpha}(\mathbf{x}_{i}, - \mathbf{d}_{\mathbf{x}_{0}}) \subseteq \Omega_{i}, \]
where $C_{\alpha}(\mathbf{x}_{i}, - \mathbf{d}_{\mathbf{x}_{0}}) $ is defined in Definition \ref{definition_alpha_cone}.
\end{proposition}

\begin{proof}
Since we have $\eta < \frac{\varepsilon}{3}$, we can set $r = \frac{1}{2} \sqrt{4 (\varepsilon - \eta)^{2} - (\varepsilon + \eta)^{2}}$ and $\mathcal{C}_{r,\varepsilon} = \overline{D_{r}}(\mathbf{0'})\times [-\varepsilon,\varepsilon]$. Moreover, we assume $\eta < \frac{\varepsilon}{8}$ i.e. $2 \sqrt{2 \varepsilon \eta} < \varepsilon$ so $f_{\eta}^{-1}(\varepsilon)$ is well defined. Choose $ \alpha  \in ]0,f_{\eta}^{-1}(\varepsilon)] $ then consider $\mathbf{x}_{i} = (\mathbf{x'},x_{n}) \in B_{\alpha}(\mathbf{x}_{0}) \cap \overline{\Omega_{i}} $ and $ \mathbf{y}_{i} = (\mathbf{y'},y_{n}) \in C_{\alpha}(\mathbf{x}_{i},- \mathbf{d}_{\mathbf{x}_{0}}) $. The proof of the assertion $\mathbf{y}_{i} \in \Omega_{i}$ is divided into the three following steps.
\begin{enumerate}
\item Check $\mathbf{x}_{i} \in \mathcal{C}_{r,\varepsilon} $ so as to introduce the point $\mathbf{\tilde{x}}_{i} = (\mathbf{x'},\varphi_{i}(\mathbf{x'})) $ of $ \partial \Omega_{i} $ satisfying $ x_{n} \leqslant \varphi_{i} (\mathbf{x'} )$.
\item Consider $\mathbf{\tilde{y}}_{i} = (\mathbf{y'},y_{n} + \varphi_{i}(\mathbf{x'}) - x_{n})$ and prove $\mathbf{\tilde{y}}_{i} \in C_{\alpha}(\mathbf{\tilde{x}}_{i},- \mathbf{d}_{\mathbf{x}_{0}} ) \subseteq B_{\varepsilon}(\mathbf{\tilde{x}}_{i} - \varepsilon \mathbf{d}_{\mathbf{\tilde{x}}_{i}}) \subseteq \Omega_{i}$.
\item Show $(\mathbf{\tilde{y}}_{i},\mathbf{y}_{i}) \in \mathcal{C}_{r,\varepsilon} \times \mathcal{C}_{r,\varepsilon} $ in order to deduce $ y_{n} + \varphi_{i}(\mathbf{x'}) - x_{n} < \varphi_{i} (\mathbf{y'})$ and conclude $ \mathbf{y}_{i} \in \Omega_{i}$.
\end{enumerate}
First, from \eqref{eqn_f}, we have: $ \max ( \Vert \mathbf{x'} \Vert, \vert x_{n} \vert ) \leqslant \Vert \mathbf{x}_{i} - \mathbf{x}_{0} \Vert < \alpha \leqslant f^{-1}_{\eta}(\varepsilon) < \frac{\varepsilon}{3} $. Since $\eta < \frac{\varepsilon}{8}$, we get $r > \frac{1}{2} [4 (\frac{7\varepsilon}{8})^{2} - (\frac{9\varepsilon}{8})^{2}]^{\frac{1}{2}} > \frac{\varepsilon}{2}  $ thus $\mathbf{x}_{i} \in \overline{\Omega_{i}} \cap \mathcal{C}_{r,\varepsilon}$. Hence, from Proposition \ref{prop_parametrisation_locale_i}, it comes $ x_{n} \leqslant \varphi_{i} (\mathbf{x'})$. We set $\mathbf{\tilde{x}}_{i} = (\mathbf{x'},\varphi_{i}(\mathbf{x'})) \in \partial \Omega_{i} \cap \mathcal{C}_{r,\varepsilon}$. Then, we prove $ C_{\alpha}(\mathbf{\tilde{x}}_{i}, -\mathbf{d}_{\mathbf{x}_{0}}) \subseteq B_{\varepsilon}(\mathbf{\tilde{x}}_{i} - \varepsilon \mathbf{d}_{\mathbf{\tilde{x}}_{i}} )  $ so consider any $\mathbf{y} \in C_{\alpha}(\mathbf{\tilde{x}}_{i}, - \mathbf{d}_{\mathbf{x}_{0}} ) $. Combining the Cauchy-Schwartz inequality and $\mathbf{y}  \in C_{\alpha}(\mathbf{\tilde{x}}_{i}, - \mathbf{d}_{\mathbf{x}_{0}} ) $, we get:
\[ \begin{array}{rcl}  
\Vert \mathbf{y} - \mathbf{\tilde{x}}_{i} + \varepsilon \mathbf{d}_{\mathbf{\tilde{x}}_{i}} \Vert^{2} - \varepsilon^{2} & \leqslant &  \Vert \mathbf{y} - \mathbf{\tilde{x}}_{i} \Vert^{2} + 2 \varepsilon \Vert \mathbf{y} - \mathbf{\tilde{x}}_{i} \Vert \Vert \mathbf{d}_{\mathbf{\tilde{x}}_{i}} - \mathbf{d}_{\mathbf{x}_{0}} \Vert - 2 \varepsilon  \Vert \mathbf{y} - \mathbf{\tilde{x}}_{i} \Vert \cos  \alpha \\
& \\
& < &  2 \Vert \mathbf{y} - \mathbf{\tilde{x}}_{i} \Vert \left(  \dfrac{\alpha}{2} + 2 \sqrt{2 \varepsilon \eta} + \sqrt{2} \Vert \mathbf{x'} \Vert - \varepsilon \cos \alpha \right)  <  2 \alpha \cos \alpha \underbrace{ \left( f_{\eta}(\alpha) - \varepsilon \right)}_{ \leqslant 0}, \\
\end{array} \]
where we used \eqref{eqn_g} on $\mathbf{\tilde{x}}_{i} \in \partial \Omega_{i} \cap \mathcal{C}_{r,\varepsilon}$ and $\Vert \mathbf{x'} \Vert \leqslant \Vert \mathbf{x}_{i} - \mathbf{x}_{0} \Vert < \alpha$. Hence, $\mathbf{y} \in B_{\varepsilon}(\mathbf{\tilde{x}}_{i} - \varepsilon \mathbf{d}_{\mathbf{\tilde{x}}_{i}})$ so $C_{\alpha}(\mathbf{\tilde{x}}_{i}, -\mathbf{d}_{\mathbf{x}_{0}}) \subseteq B_{\varepsilon}(\mathbf{\tilde{x}}_{i} - \varepsilon \mathbf{d}_{\mathbf{\tilde{x}}_{i}} ) \subseteq \Omega_{i}$, using the $\varepsilon$-ball condition. Moreover, since $\mathbf{\tilde{y}}_{i} - \mathbf{\tilde{x}}_{i} = \mathbf{y}_{i} - \mathbf{x}_{i}$ and $\mathbf{y}_{i} \in C_{\alpha}(\mathbf{x}_{i},-\mathbf{d}_{\mathbf{x}_{0}})$, we get $\mathbf{\tilde{y}}_{i} \in C_{\alpha}(\mathbf{\tilde{x}}_{i},-\mathbf{d}_{\mathbf{x}_{0}}) $, which ends the proof of $\mathbf{\tilde{y}}_{i} \in \Omega_{i}$. Finally, we check that $(\mathbf{y}_{i},\mathbf{\tilde{y}}_{i}) \in \mathcal{C}_{r,\varepsilon}\times \mathcal{C}_{r, \varepsilon}$. We have successively: 
\[  \left\lbrace \begin{array}{l}
\Vert \mathbf{y'} \Vert  \leqslant \Vert \mathbf{y'} - \mathbf{x'} \Vert + \Vert \mathbf{x'} \Vert < \sqrt{\alpha^{2} - \alpha^{2} \cos^{2} \alpha} + \alpha = \dfrac{\alpha}{\cos \alpha} \left(  \dfrac{1}{2} \sin 2 \alpha + \cos \alpha \right) < \dfrac{  f_{\eta}(\alpha) }{2} \leqslant \dfrac{\varepsilon}{2} < r \\
\\
 \vert y_{n} \vert \leqslant \vert y_{n} - x_{n} \vert + \vert x_{n} \vert \leqslant \Vert \mathbf{y}_{i} - \mathbf{x}_{i} \Vert + \Vert \mathbf{x}_{i} - \mathbf{x}_{0} \Vert  < 2 \alpha < f(\alpha) \leqslant \varepsilon \\
\\
\vert \tilde{y}_{n} \vert = \vert y_{n} + \varphi_{i}(\mathbf{x'}) - x_{n} \vert \leqslant \Vert  \mathbf{y}_{i} - \mathbf{x}_{i} \Vert + \varepsilon - \sqrt{(\varepsilon- \eta)^{2} - \Vert \mathbf{x'} \Vert^{2}} < \alpha + \dfrac{\eta(2 \varepsilon - \eta) + \Vert \mathbf{x'} \Vert^{2}}{\varepsilon + \sqrt{(\varepsilon - \eta)^{2} - \Vert \mathbf{x'} \Vert^{2}}}. \\
\end{array} \right. \]
Here, we used Relation \eqref{eqn_e}, the fact that $\mathbf{y}_{i} \in C_{\alpha}(\mathbf{x}_{i},-\mathbf{d}_{\mathbf{x}_{0}})$ and $\mathbf{x}_{i} \in B_{\alpha}(\mathbf{x}_{0})$. Hence, we obtain: $\vert \tilde{y}_{n} \vert < 2 \alpha + 2 \eta  < 2 f_{\eta}^{-1}(\varepsilon) + 2 \frac{\varepsilon}{8} \leqslant \frac{2 \varepsilon}{3} + \frac{\varepsilon}{4} < \varepsilon$. To conclude, apply Proposition \ref{prop_parametrisation_locale_i} to $\mathbf{\tilde{y}}_{i} \in \Omega_{i} \cap \mathcal{C}_{r,\varepsilon}$ in order to get $ y_{n} + \varphi_{i}(\mathbf{x'}) - x_{n} < \varphi (\mathbf{y'})$. Since we firstly proved $ x_{n} \leqslant \varphi_{i} (\mathbf{x'})$, we have $ y_{n} < \varphi_{i}(\mathbf{y'}) $. Applying Proposition \ref{prop_parametrisation_locale_i} to $\mathbf{y}_{i} \in \mathcal{C}_{r,\varepsilon}$, we get $\mathbf{y}_{i} \in \Omega_{i} $ as required.
\end{proof}

\begin{lemma}
\label{lemme_lipschitz_i}
The following map is well defined, smooth, surjective and increasing:
\[ \begin{array}{rcl}
g: \quad ]0,\frac{\pi}{8}[ & \longrightarrow & ]0, + \infty[ \\
\eta & \longmapsto & \dfrac{32 \eta}{\cos^{2}(4 \eta)}. \\
\end{array} \]
In particular, it is an homeomorphism and its inverse $g^{-1}$ satisfies the following relations:
\begin{equation}
\label{eqn_i}
 \forall \varepsilon > 0, \quad g^{-1}(\varepsilon) < \frac{\varepsilon}{32} \qquad \mathrm{and} \qquad g^{-1}(\varepsilon) < \frac{1}{4} f^{-1}_{g^{-1}(\varepsilon)}(\varepsilon),  
 \end{equation}
where $f^{-1}_{\eta}$ is defined in Lemma \ref{lemme_alpha_cone_i}.
\end{lemma}

\begin{proof}
We only prove the inequality $ g^{-1}(\varepsilon) < \frac{1}{4} f^{-1}_{g^{-1}(\varepsilon)}(\varepsilon)$. The remaining part is basic calculus. Consider any $\varepsilon > 0$. There exists a unique $\eta \in ]0,\frac{\pi}{8}[$ such that $g(\eta) = \varepsilon$ or equivalently $\eta = g^{-1}(\varepsilon)$. Hence, we have $4 \eta \in ]0,\frac{\pi}{2}[$ so we can compute, using the first inequality $\eta < \frac{\varepsilon}{32}$:
\[ f_{\eta}( 4 \eta) =  \dfrac{2 \sqrt{2 \eta \varepsilon}}{\cos (4 \eta)} \left( 3 \sqrt{\dfrac{2 \eta}{\varepsilon}} + 1 \right) < \dfrac{2 \sqrt{2 \eta \varepsilon}}{\cos (4 \eta)} \left( 3 \sqrt{\dfrac{2}{32}} + 1 \right) < \dfrac{4 \sqrt{2 \varepsilon \eta}}{\cos (4 \eta)} = \sqrt{g(\eta) \varepsilon} = \varepsilon.  \]
Since $f_{\eta}$ is an increasing homeomorphism, so does $f^{-1}_{\eta}$ and the inequality follows: $4 \eta < f_{\eta}^{-1}(\varepsilon)$.
\end{proof}

\begin{corollary}
\label{coro_lipschitz_i}
In Assumption \ref{hypotheses_parametrization_i}, we set $\eta = g^{-1}(\varepsilon)$, then consider $\alpha = f^{-1}_{\eta}(\varepsilon)$ and $\tilde{r} = \frac{1}{4} \alpha - \eta $. The restriction to $\overline{D_{\tilde{r}}}(\mathbf{0'})$ of the map $\varphi_{i}$ defined in Proposition \ref{prop_parametrisation_locale_i} is $\frac{1}{\tan \alpha}$-Lipschitz continuous.
\end{corollary}

\begin{proof}
Let $\eta = g^{-1}(\varepsilon)$ and using \eqref{eqn_i}, we have $\eta < \frac{\varepsilon}{32}$ so we can set $r = \frac{1}{2}\sqrt{ 4(\varepsilon - \eta)^{2} - (\varepsilon + \eta)^{2}}$ and $\alpha = f^{-1}_{\eta}(\varepsilon)$, but we also have $\tilde{r} := \frac{1}{4} \alpha - \eta > 0 $. We consider any $(\mathbf{x_{+}'},\mathbf{x_{-}'}) \in \overline{D_{\tilde{r}}}(\mathbf{0'}) \times \overline{D_{\tilde{r}}}(\mathbf{0'})  $. Using \eqref{eqn_f}-\eqref{eqn_i}, we get $\tilde{r} < \frac{1}{4} f_{\eta}^{-1}(\varepsilon) < \frac{\varepsilon}{12}  < \frac{1}{2} [ 4 (\frac{31 \varepsilon}{32})^{2} - (\frac{33 \varepsilon}{32})^{2} ]^{\frac{1}{2}} < r$. From Proposition \ref{prop_parametrisation_locale_i}, we can define $\mathbf{x}_{i}^{\pm}: = (\mathbf{x_{\pm}'},\varphi_{i}(\mathbf{x_{\pm}'})) \in \partial \Omega_{i}$. Then, we show that $\mathbf{x}_{i}^{\pm} \in \partial \Omega_{i} \cap B_{\alpha}(\mathbf{x}_{0}) \cap B_{\alpha}(\mathbf{x}_{i}^{\mp})$. Relation \eqref{eqn_h} ensures that $ \Vert \mathbf{x}_{i}^{\pm} - \mathbf{x}_{0} \Vert < 2 \Vert \mathbf{x_{\pm}'} \Vert + 2 \eta \leqslant 2 \tilde{r} + 2 \eta < \alpha$ and the triangle inequality gives $\Vert \mathbf{x}_{i}^{+} - \mathbf{x}_{i}^{-} \Vert \leqslant \Vert \mathbf{x}_{i}^{+} - \mathbf{x}_{0} \Vert  + \Vert \mathbf{x}_{0} - \mathbf{x}_{i}^{-} \Vert < 4 \tilde{r} + 4 \eta = \alpha  $. Finally, we apply Proposition \ref{prop_alpha_cone_i} to $\mathbf{x}_{i}^{\pm} \in \partial \Omega_{i} \cap B_{\alpha}(\mathbf{x}_{0})$, which cannot belong to the cone $C_{\alpha}(\mathbf{x}_{i}^{\mp}, - \mathbf{d}_{\mathbf{x}_{0}}) \subseteq \Omega_{i}$. Hence, we obtain:
\[ \vert \langle \mathbf{x}_{i}^{+} - \mathbf{x}_{i}^{-} ~\vert~ \mathbf{d}_{\mathbf{x}_{0}} \rangle \vert \leqslant \cos \alpha \Vert \mathbf{x}_{i}^{+} - \mathbf{x}_{i}^{-} \Vert =  \cos \alpha \sqrt{ \Vert \mathbf{x_{+}'} - \mathbf{x_{-}'} \Vert^{2} +  \vert \langle \mathbf{x}_{i}^{+} - \mathbf{x}_{i}^{-} ~\vert~ \mathbf{d}_{\mathbf{x}_{0}} \rangle \vert^{2} }.  \] 
Re-arranging the above inequality, we deduce that the map $\varphi_{i}$ is $L$-Lipschitz continuous with $L > 0$ depending only on $\varepsilon$ as required: $\vert \varphi_{i}(\mathbf{x_{+}'}) - \varphi_{i}(\mathbf{x_{-}'}) \vert =  \vert \langle \mathbf{x}_{i}^{+} - \mathbf{x}_{i}^{-} ~\vert~ \mathbf{d}_{\mathbf{x}_{0}} \rangle \vert \leqslant \frac{1}{\tan \alpha} \Vert \mathbf{x_{+}'} - \mathbf{x_{-}'} \Vert $.
\end{proof}

\begin{proposition}
\label{prop_regularite_cunun_i}
We set $\tilde{r} = \frac{1}{4} f^{-1}_{g^{-1}(\varepsilon)}(\varepsilon) - g^{-1}(\varepsilon) $, where $f$ and $g$ are defined in Lemmas \ref{lemme_alpha_cone_i} and \ref{lemme_lipschitz_i}. Then, the restriction to $D_{\tilde{r}}(\mathbf{0'})$ of the map $\varphi_{i}$ defined in Proposition \ref{prop_parametrisation_locale_i} is differentiable:
\[ \forall \mathbf{a'} \in D_{\tilde{r}}(\mathbf{0'}), \quad \nabla \varphi_{i}(\mathbf{a'}) =  \dfrac{-1}{\langle \mathbf{d}_{\mathbf{a}_{i}}~\vert~ \mathbf{d}_{\mathbf{x}_{0}} \rangle } \mathbf{d}_{\mathbf{a}_{i}}' \qquad \mathrm{where} \quad \mathbf{a}_{i} := (\mathbf{a'},\varphi_{i}(\mathbf{a'})). \]
Moreover, $\nabla \varphi_{i}:  D_{\tilde{r}}(\mathbf{0'}) \rightarrow \mathbb{R}^{n-1}$ is $L$-Lipschitz continuous with $L > 0$ depending only on $\varepsilon$.  
\end{proposition}

\begin{proof}
Let $\eta = g^{-1}(\varepsilon)$ and using \eqref{eqn_i}, we have $\eta < \frac{\varepsilon}{32}$ so we can set $r = \frac{1}{2}\sqrt{ 4(\varepsilon - \eta)^{2} - (\varepsilon + \eta)^{2}}$ and $\alpha = f^{-1}_{\eta}(\varepsilon)$, but we also have $\tilde{r} := \frac{1}{4} \alpha - \eta > 0 $. Let $\mathbf{a'} \in D_{\tilde{r}} (\mathbf{0'})$ and $\mathbf{x'} \in \overline{D_{\tilde{r} - \Vert \mathbf{a'} \Vert}}(\mathbf{a'})$. Hence, $(\mathbf{a'},\mathbf{x'}) \in D_{\tilde{r}}(\mathbf{0'}) \times D_{\tilde{r}}(\mathbf{0'})  $. Using \eqref{eqn_f}-\eqref{eqn_i}, we get $\tilde{r} < \frac{1}{4} f_{\eta}^{-1}(\varepsilon) < \frac{\varepsilon}{12}  < \frac{1}{2} [ 4 (\frac{31 \varepsilon}{32})^{2} - (\frac{33 \varepsilon}{32})^{2} ]^{\frac{1}{2}} < r$. From Proposition \ref{prop_parametrisation_locale_i}, we can define $\mathbf{x}_{i}^{\pm}: = (\mathbf{x_{\pm}'},\varphi_{i}(\mathbf{x_{\pm}'})) \in \partial \Omega_{i}$. Then, we apply \eqref{eqn_inegalite_globale} to $ \Omega_{i}$ thus:
\[ \vert \langle \mathbf{x}_{i} - \mathbf{a}_{i} ~\vert~ \mathbf{d}_{\mathbf{a}_{i}} \rangle \vert \leqslant \dfrac{1}{2 \varepsilon} \Vert \mathbf{x}_{i} - \mathbf{a}_{i} \Vert^{2} = \dfrac{1}{2 \varepsilon} \left( \Vert \mathbf{x'} - \mathbf{a'} \Vert^{2} + \vert \varphi_{i}( \mathbf{x'} )  - \varphi_{i} (\mathbf{a'}) \vert^{2} \right) \leqslant \underbrace{\dfrac{1}{2 \varepsilon} \left( 1 + \dfrac{1}{\tan^{2} \alpha} \right)}_{ : = C(\varepsilon) > 0} \Vert \mathbf{x'} - \mathbf{a'} \Vert^{2}, \]    
where we also used the Lipschitz continuity of $\varphi_{i}$ on $\overline{D_{\tilde{r}}}(\mathbf{0'})$ established in Corollary \ref{coro_lipschitz_i}. We note that $\mathbf{d}_{\mathbf{a}_{i}} = (\mathbf{d}_{\mathbf{a}_{i}}', (\mathbf{d}_{\mathbf{a}_{i}})_{n})$ where $(\mathbf{d}_{\mathbf{a}_{i}})_{n} = \langle \mathbf{d}_{\mathbf{a}_{i}} ~\vert~ \mathbf{d}_{\mathbf{x}_{0}} \rangle$. Hence, the above inequality takes the form:
\[ \vert \left( \varphi_{i}(\mathbf{x'}) - \varphi_{i} (\mathbf{a'} ) \right) (\mathbf{d}_{\mathbf{a}_{i}})_{n} + \langle \mathbf{d}_{\mathbf{a}_{i}}' ~\vert~ \mathbf{x'} - \mathbf{a'} \rangle \vert \leqslant C(\varepsilon) \Vert \mathbf{x'} - \mathbf{a'} \Vert^{2}.  \]
This last inequality is a first-order Taylor expansion of $\varphi_{i}$ if it can be divided by a uniform positive constant smaller than $(\mathbf{d}_{\mathbf{a}_{i}})_{n}$. Let us justify this last assertion. From \eqref{eqn_a} and \eqref{eqn_b}, we deduce:
\[ (\mathbf{d}_{\mathbf{a}_{i}})_{n} = 1 - \dfrac{1}{2} \Vert \mathbf{d}_{\mathbf{a}_{i}} - \mathbf{d}_{\mathbf{x}_{0}} \Vert^{2} \geqslant 1 - \dfrac{1}{2 \varepsilon^{2}} \Vert \mathbf{a}_{i} - \mathbf{x}_{0} \Vert^{2} - \dfrac{4 \varepsilon \eta - \eta^{2}}{ 2 \varepsilon^{2}} > \dfrac{1}{\varepsilon} \sqrt{(\varepsilon - \eta)^{2} - \Vert \mathbf{a'} \Vert^{2} } - \dfrac{\eta}{\varepsilon}.   \]
Then, Inequality \eqref{eqn_i} gives $ \frac{\eta}{\varepsilon}  < \frac{1}{32}$ and from \eqref{eqn_f}, it comes $\Vert \mathbf{a'} \Vert < \tilde{r} < \frac{\alpha}{4} < \frac{\varepsilon}{12}$. Consequently, we get $(\mathbf{d}_{\mathbf{a}_{i}})_{n} > [(\frac{31 }{32})^{2} - (\frac{1}{12})^{2} ]^{\frac{1}{2}} - \frac{1}{32} > \frac{29}{32}  $ and from the foregoing, we obtain:
\[ \forall \mathbf{x'} \in \overline{D_{\tilde{r} - \Vert \mathbf{a'} \Vert}}(\mathbf{a'}), \quad \begin{array}{|c|} 
\displaystyle{ \varphi_{i}(\mathbf{x'}) - \varphi_{i} (\mathbf{a'} )  + \left\langle \frac{\mathbf{d}_{\mathbf{a}_{i}}'}{(\mathbf{d}_{\mathbf{a}_{i}})_{n}}  ~\vert~ \mathbf{x'} - \mathbf{a'} \right\rangle } \\
\end{array} \leqslant \dfrac{32 C(\varepsilon)}{29} \Vert \mathbf{x'} - \mathbf{a'} \Vert^{2}.  \]
Therefore, $\varphi_{i}$ is differentiable at any point $\mathbf{a'} \in D_{\tilde{r}}(\mathbf{0'})$ with $\nabla \varphi_{i}(\mathbf{a'}) = - \mathbf{d}_{\mathbf{a}_{i}}' / (\mathbf{d}_{\mathbf{a}_{i}})_{n}$. Finally, we show that $\nabla \varphi_{i}: D_{\tilde{r}}(\mathbf{0'}) \rightarrow \mathbb{R}^{n-1}$ is Lipschitz continuous. Let $(\mathbf{x'},\mathbf{a'}) \in D_{\tilde{r}}(\mathbf{0'}) \times D_{\tilde{r}}(\mathbf{0'})  $. We have:
\[ \begin{array}{rcl}
\Vert \nabla \varphi_{i}(\mathbf{x'}) - \nabla \varphi_{i}(\mathbf{a'}) \Vert & \leqslant & \vert \frac{1}{(\mathbf{d}_{\mathbf{x}_{i}})_{n}} - \frac{1}{(\mathbf{d}_{\mathbf{a}_{i}})_{n}} \vert \Vert \mathbf{d}_{\mathbf{x}_{i}}^{'} \Vert + \frac{1}{(d_{\mathbf{a}_{i}})_{n}} \Vert \mathbf{d}_{\mathbf{a}_{i}}^{'} - \mathbf{d}_{\mathbf{x}_{i}}^{'} \Vert  \\
& & \\
 & \leqslant & \dfrac{32}{29} \left( \dfrac{32}{29} \vert (\mathbf{d}_{\mathbf{a}_{i}})_{n} - (\mathbf{d}_{\mathbf{x}_{i}})_{n} \vert + \Vert \mathbf{d}_{\mathbf{a}_{i}} - \mathbf{d}_{\mathbf{x}_{i}} \Vert \right) \\
&  & \\
& \leqslant & \dfrac{32}{29 \varepsilon} \left( 1 + \dfrac{32}{29} \right) \Vert \mathbf{x}_{i} - \mathbf{a}_{i} \Vert ~~ \leqslant ~~ \dfrac{32}{29 \varepsilon} \left( 1 + \dfrac{32}{29} \right) \sqrt{1 + \dfrac{1}{\tan^{2} \alpha}} \Vert \mathbf{x'} - \mathbf{a'} \Vert. \\
\end{array} \]
We used the fact that $(\mathbf{d}_{\mathbf{a}_{i}})_{n} < \frac{29}{32}$, the Lipschitz continuity of $\varphi_{i}$ proved in Corollary \ref{coro_lipschitz_i} and the one of the map $\mathbf{x}_{i} \in \partial \Omega_{i} \mapsto \mathbf{d}_{\mathbf{x}_{i}}$ coming from Proposition \ref{prop_normale_lipschitzienne} applied to $\Omega_{i} \in \mathcal{O}_{\varepsilon}(B)$. To conclude, $\nabla \varphi_{i}$ is an $L$-Lipschitz continuous map, where $L > 0$ depends only on $\varepsilon$.
\end{proof}

\begin{proof}[\textbf{Proof of Theorem \ref{thm_parametrisation_locale_i}}]
Set  $K  = \overline{D_{\frac{\tilde{r}}{2}}}(\mathbf{0'})$ where $\tilde{r}: = \frac{1}{4} f^{-1}_{g^{-1}(\varepsilon)}(\varepsilon) - g^{-1}(\varepsilon) $ is positive from \eqref{eqn_i}. From Propositions \ref{prop_parametrisation_locale_i}, \ref{prop_regularite_cunun_i} and Corollary \ref{coro_lipschitz_i}, we proved that each $\Omega_{i}$ is parametrized by a local graph $\varphi_{i}: K \rightarrow ]- \varepsilon,\varepsilon[ $ as in Theorem \ref{thm_parametrisation_locale_i}. Hence, it remains to prove the convergence of these graphs. Since the sequence $(\varphi_{i})_{i \in \mathbb{N}}$ is uniformly bounded and equi-Lipschitz continuous, from the Arzel\`{a}-Ascoli Theorem and up to a subsequence, it is converging to a continuous function $\tilde{\varphi}: K \rightarrow ]- \varepsilon,\varepsilon[$. Considering the local map $\varphi: K \rightarrow ]- \varepsilon,\varepsilon[ $ associated with $\partial \Omega$, we now show that $\varphi \equiv \tilde{ \varphi}$. Considering any $\mathbf{x'} \in K$, we set $\mathbf{x} = (\mathbf{x'},\tilde{\varphi}(\mathbf{x'}))$ and $\mathbf{x}_{i} = (\mathbf{x'},\varphi_{i}(\mathbf{x'}))$. There exists $\mathbf{y} \in \partial \Omega$ such that $d(\mathbf{x}_{i},\partial \Omega) = \Vert \mathbf{x}_{i} - \mathbf{y} \Vert$. Then, we have:
\[ \begin{array}{rcl} 
d(\mathbf{x},\partial \Omega) & \leqslant & \Vert \mathbf{x} - \mathbf{y} \Vert \leqslant \Vert \mathbf{x} - \mathbf{x}_{i} \Vert + \Vert \mathbf{x}_{i} - \mathbf{y} \Vert = \vert \varphi_{i}(\mathbf{x'}) - \tilde{\varphi}(\mathbf{x'}) \vert + d(\mathbf{x_{i}},\partial \Omega)  \\
 & & \\
 & \leqslant & \Vert \varphi_{i} - \tilde{\varphi} \Vert_{C^{0}(K)} + d_{H}(\partial \Omega_{i},\partial \Omega). \\
 \end{array} \]
By letting $i \rightarrow + \infty$, we obtain $\mathbf{x} \in \partial \Omega \cap (K \times [- \varepsilon,\varepsilon]) $. Hence, Proposition \ref{prop_parametrisation_locale_i} gives $\mathbf{x} = (\mathbf{x'},\varphi(\mathbf{x'}))$ so $\varphi(\mathbf{x'}) = \tilde{\varphi}(\mathbf{x'})$ for any $\mathbf{x'} \in K$. This also show that $\varphi$ is the unique limit of any converging subsequence of $(\varphi_{i})_{i \in \mathbb{N}}$. Hence, the whole sequence $(\varphi_{i})_{i \in \mathbb{N}}$ is converging to $\varphi$ uniformly on $K$. Similarly, $(\nabla \varphi_{i})_{i \in \mathbb{N}}$ is uniformly bounded and equi-Lipschitz continuous, so it converges uniformly on $K$ to a map, which must be $\nabla  \varphi$ (use the convergence in the sense of distributions). To conclude, using \cite[Section 5.2.2]{HenrotPierre}, each coefficient of the Hessian matrix of $\varphi_{i}$ is uniformly bounded in $L^{\infty}(K)$. Hence \cite[Lemma 2.2.27]{HenrotPierre}, each of them weakly-star converges in $L^{\infty}(K)$ to the one of $\varphi$.
\end{proof}

\section{Continuity of some geometric functionals in the class $\mathcal{O}_{\varepsilon}(B)$}
\label{section_continuite}
In this section, we prove that the convergence properties and the uniform $C^{1,1}$-regularity of the class $\mathcal{O}_{\varepsilon}(B)$ ensure the continuity of some geometric functionals. More precisely, with a suitable partition of unity, we show how to use the local convergence results of Theorem \ref{thm_parametrisation_locale_i} to obtain the global continuity of linear integrals in the elementary symmetric polynomials of the principal curvatures. Throughout this section, we make the following hypothesis.

\begin{assumption}
\label{hypothese_continuite}
We assume that $(\Omega_{i})_{i \in \mathbb{N}}$ is a sequence of elements from $\mathcal{O}_{\varepsilon}(B)$ converging to $\Omega \in \mathcal{O}_{\varepsilon}(B)$ in the sense of Proposition \ref{prop_compacite_boule} (i)-(vi), where $\varepsilon $ and $B$ are as in Proposition \ref{prop_compacite_boule}.
\end{assumption}

\begin{remark}
Note that in this section, the proofs are based on the results of Theorem \ref{thm_parametrisation_locale_i}, so we only need to assume Points (ii) and (v) of Proposition \ref{prop_compacite_boule} in the Assumption \ref{hypothese_continuite} (see Remark \ref{remarque_hypothese_faible_i}).
\end{remark}

\begin{definition}
\label{definition_convergence_diagonale}
Let $ f$, $ (f_{i})_{i \in \mathbb{N}} : E \to F$ be some continuous maps between two metric spaces. We say that $(f_{i})_{i \in \mathbb{N}}$ diagonally converges to $f$ if for any sequence $(t_{i})_{i \in \mathbb{N}}$ converging to $t$ in $E$, the sequence $(f_{i}(t_{i}))_{i \in \mathbb{N}}$ converges to $f(t)$ in $F$. 
\end{definition}

\begin{remark}
Note that the uniform convergence implies the diagonal convergence implying itself the pointwise convergence. Conversely, any sequence of equi-continuous maps converging pointwise is diagonally convergent. Moreover, from the Arzel\`{a}-Ascoli Theorem, it is uniformly convergent if in addition, it is uniformly bounded.
\end{remark}

The section is organized as follows. First, we recall the basic notions related to the geometry of hypersurfaces. Then, we study the continuity of functionals which depend on the position and the normal. Next, we consider linear functionals in the scalar mean curvature. Finally, we treat the case of the Gaussian curvature in $\mathbb{R}^{3}$ and we prove in  $\mathbb{R}^{n}$ the following continuity result. 

\begin{theorem}
\label{thm_continuite_rn}
Let $\varepsilon, B, \Omega$, $(\Omega_{i})_{i \in \mathbb{N}}$ be as in Assumption \ref{hypothese_continuite}. We consider some continuous maps $j^{l}, j^{l}_{i} : \mathbb{R}^{n} \times \mathbb{S}^{n-1} \rightarrow \mathbb{R} $ such that each sequence $(j^{l}_{i})_{i \in \mathbb{N}}$ is uniformly bounded on $\overline{B} \times \mathbb{S}^{n-1}$ and diagonally converges to $j^{l}$ for any $l \in \lbrace 0, \ldots, n-1 \rbrace$. Then, the following functional is continuous:
\[ J  \left( \partial \Omega_{i} \right) : = \sum_{l = 0}^{n-1} \int_{\partial \Omega_{i}} \left[ \sum_{1 \leqslant n_{1} < \ldots < n_{l} \leqslant n-1} \kappa_{n_{1}}^{\partial \Omega_{i}} \left( \mathbf{x} \right) \ldots \kappa^{\partial \Omega_{i}}_{n_{l}} \left( \mathbf{x} \right) \right] j^{l}_{i} \left[ \mathbf{x}, \mathbf{n}^{\partial \Omega_{i}} \left( \mathbf{x} \right) \right] dA \left( \mathbf{x} \right) \underset{i \rightarrow +  \infty}{\longrightarrow} J(\partial \Omega), \]
where $\kappa_{1}, \ldots \kappa_{n-1}$ are the principal curvatures, $\mathbf{n}$ the unit outer normal field to the hypersurface, and where the integration is done with respect to the $(n-1)$-dimensional Hausdorff measure $A(.)$.
\end{theorem}  

\begin{remark}
In the specific case of compact $C^{1,1}$-hypersurfaces, note that the above theorem is stronger than Federer's one on sets of positive reach \cite[Theorem 5.9]{Federer}. Indeed, in Theorem \ref{thm_continuite_rn}, taking $j^{l}_{i}(\mathbf{x},\mathbf{n}(\mathbf{x})) = j^{l}(\mathbf{x})$ yields to the convergence of the curvature measures associated with $\partial \Omega_{i}$ to the ones of $\partial \Omega$ in the sense of Radon measures.
\end{remark}

\subsection{On the geometry of hypersurfaces with $C^{1,1}$-regularity}
\label{section_geometrie}
Let us consider a non-empty compact $C^{1,1}$-hypersurface $\mathcal{S} \subset \mathbb{R}^{n}$. Merely speaking, for any point $\mathbf{x}_{0} \in \mathcal{S}$, there exists $r_{\mathbf{x}_{0}} > 0$, $ a_{\mathbf{x}_{0}} > 0$, and a unit vector $\mathbf{d}_{\mathbf{x}_{0}}$ such that in the cylinder defined by:
\begin{equation}
\label{eqn_cylinder}
\mathcal{C}_{r_{\mathbf{x}_{0}},a_{\mathbf{x}_{0}}}(\mathbf{x}_{0}) = \left\lbrace \mathbf{x} \in \mathbb{R}^{n}, \quad \vert \langle \mathbf{x} - \mathbf{x}_{0} ~\vert~ \mathbf{d}_{\mathbf{x}_{0}} \rangle \vert < a_{\mathbf{x}_{0}} ~\mathrm{and}~ \Vert (\mathbf{x} - \mathbf{x}_{0}) - \langle \mathbf{x} - \mathbf{x}_{0} ~\vert~ \mathbf{d}_{\mathbf{x}_{0}} \rangle \mathbf{d}_{\mathbf{x}_{0}} \Vert < r_{\mathbf{x}_{0}} \right\rbrace,  
\end{equation}
the hypersurface $\mathcal{S}$ is the graph of a $C^{1,1}$-map. Introducing the orthogonal projection on the affine hyperplane $\mathbf{x}_{0} + \mathbf{d}_{\mathbf{x}_{0}}^{\perp}$:
\begin{equation} 
\label{eqn_projection}
\begin{array}{rrcl} 
\Pi_{\mathbf{x}_{0}} : & \mathbb{R}^{n} & \longrightarrow & \mathbf{x}_{0} + \mathbf{d}_{\mathbf{x}_{0}}^{\perp} \\
& \mathbf{x} & \longmapsto & \mathbf{x} - \langle \mathbf{x} - \mathbf{x}_{0} ~\vert~ \mathbf{d}_{\mathbf{x}_{0}} \rangle \mathbf{d}_{\mathbf{x}_{0}}, \\
\end{array}  
\end{equation}
and considering the set $D_{r_{\mathbf{x}_{0}}}(\mathbf{x}_{0}) = \Pi_{\mathbf{x}_{0}} ( \mathcal{C}_{r_{\mathbf{x}_{0}},a_{\mathbf{x}_{0}}}(\mathbf{x}_{0})) $, this means that there exists a continuously differentiable map $\varphi_{\mathbf{x}_{0}} : \mathbf{x'} \in D_{r_{\mathbf{x}_{0}}}(\mathbf{x}_{0}) \mapsto  \varphi_{\mathbf{x}_{0}}(\mathbf{x'})  \in  ]-a_{\mathbf{x}_{0}},a_{\mathbf{x}_{0}}[  $ such that its gradient $ \nabla \varphi_{\mathbf{x}_{0}}$ and $\varphi_{\mathbf{x}_{0}}$ are $L_{\mathbf{x}_{0}}$-Lipschitz continuous maps, and such that:
\[ \mathcal{S} \cap C_{r_{\mathbf{x}_{0}},a_{\mathbf{x}_{0}}}(\mathbf{x}_{0}) = \lbrace \mathbf{x'} + \varphi_{\mathbf{x}_{0}}(\mathbf{x'}) \mathbf{d}_{\mathbf{x}_{0}}, \quad \mathbf{x'} \in D_{r_{\mathbf{x}_{0}}}(\mathbf{x}_{0}) \rbrace. \]
Hence, we can introduce the local parametrization:
\[ \begin{array}{rrcl}
X_{\mathbf{x}_{0}} : &  D_{r_{\mathbf{x}_{0}}}(\mathbf{x}_{0}) & \longrightarrow & \mathcal{S} \cap C_{r_{\mathbf{x}_{0}},a_{\mathbf{x}_{0}}}(\mathbf{x}_{0}) \\
& \mathbf{x'} & \longmapsto & \mathbf{x'} + \varphi_{\mathbf{x}_{0}}(\mathbf{x'}) \mathbf{d}_{\mathbf{x}_{0}}  \\
\end{array} \]
and $ \mathcal{S}$ is a $C^{1,1}$-hypersurface in the sense of \cite[Definition 2.2]{MontielRos}. Indeed, $X_{\mathbf{x}_{0}}$ is an homeomorphism, its inverse map is the restriction of $\Pi_{\mathbf{x}_{0}}$ to $ C_{r_{\mathbf{x}_{0}},a_{\mathbf{x}_{0}}}(\mathbf{x}_{0}) $, and $X_{\mathbf{x}_{0}}$ is an immersion of class $C^{1,1}$.
\bigskip

We usually drop the dependence in $\mathbf{x}_{0}$ to lighten the notation, and consider a direct orthonormal frame $(\mathbf{x}_{0},\mathcal{B}_{\mathbf{x}_{0}},\mathbf{d}_{\mathbf{x}_{0}})$ where $\mathcal{B}_{\mathbf{x}_{0}}$ is a basis of $\mathbf{d}_{\mathbf{x}_{0}}^{\perp}$. In this local frame, the point $\mathbf{x}_{0}$ is identified with the zero vector $\mathbf{0} \in \mathbb{R}^{n} $, the affine hyperplane $\mathbf{x}_{0} + \mathbf{d}_{\mathbf{x}_{0}}^{\perp}$ with $\mathbb{R}^{n-1}$ and $\mathbf{x}_{0} + \mathbb{R} \mathbf{d}_{\mathbf{x}_{0}}$ with $\mathbb{R}$. Hence, the cylinder $\mathcal{C}_{r_{\mathbf{x}_{0}},a_{\mathbf{x}_{0}}}(\mathbf{x}_{0}) $ becomes $D_{r}(\mathbf{0'}) \times ]-a,a[$, $\varphi_{\mathbf{x}_{0}}$ is the $C^{1,1}$-map $\varphi : D_{r}(\mathbf{0'}) \rightarrow ]-a,a[$, the projection $\Pi_{\mathbf{x}_{0}}$ is $ X^{-1}: (\mathbf{x'},x_{n}) \mapsto \mathbf{x'}$, and the parametrization $X_{\mathbf{x}_{0}}$ becomes the $C^{1,1}$-map $X : \mathbf{x'} \in D_{r}(\mathbf{0'}) \mapsto (\mathbf{x'},\varphi(\mathbf{x'})) \in \mathcal{S} \cap (D_{r}(\mathbf{0'}) \times ]-a,a[) $. In this setting, $\mathcal{S}$ is a $C^{1,1}$-hypersurface in the sense of Definition \ref{definition_regularite_cunun}.
\bigskip

Since $\mathbf{x'} \in D_{r}(\mathbf{0'}) \mapsto D_{\mathbf{x'}} X$ is injective, the vectors $\partial_{1} X$, $\ldots$, $\partial_{n-1} X$ are linearly independent. For any point $\mathbf{x} \in \mathcal{S} \cap (D_{r}(\mathbf{0'}) \times ]-a,a[)$, we define the tangent hyperplane $T_{\mathbf{x}} \mathcal{S}$ by $ D_{X^{-1}(\mathbf{x})} X (\mathbb{R}^{n-1})$. It is an $(n-1)$-dimensional vector space so $(\partial_{1} X$, $\ldots$, $\partial_{n-1} X)$ forms a basis of $T_{\mathbf{x}} \mathcal{S}$. However, this basis is not necessarily orthonormal. Consequently, the first fundamental form of $\mathcal{S}$ at $\mathbf{x}$ is defined as the restriction of the usual scalar product in $\mathbb{R}^{n}$ to the tangent hyperplane $T_{\mathbf{x} } \mathcal{S}$, i.e. as $\mathrm{\mathbf{I}}(\mathbf{x}): (\mathbf{v},\mathbf{w}) \in T_{\mathbf{x}} \mathcal{S} \times T_{\mathbf{x}} \mathcal{S}  \mapsto \langle \mathbf{v} ~\vert~ \mathbf{w} \rangle$. In the basis $(\partial_{1}X, \ldots , \partial_{n-1}X)$, it is represented by a positive-definite symmetric matrix usually referred to as $(g_{ij})_{1 \leqslant i,j\leqslant n-1}$ and its inverse denoted by $(g^{ij})_{1 \leqslant i,j \leqslant n-1}$ is also explicitly given in this case:
\begin{gather}
g_{ij}  = \left\langle \partial_{i} X ~\vert~ \partial_{j} X \right\rangle  =  \delta_{ij} + \partial_{i} \varphi \partial_{j} \varphi,   \label{expression_premiere_forme_fondamentale} \\
g^{ij} = \delta_{ij} - \dfrac{\partial_{i} \varphi \partial_{j} \varphi}{1+ \Vert \nabla \varphi \Vert^{2}}. \label{expression_inverse_premiere_forme_fondamentale} 
\end{gather}
As a function of $\mathbf{x'}$, note that each coefficient of these two matrices is Lipschitz continuous so it is a $W^{1,\infty}$-map \cite[Section 4.2.3]{EvansGariepy}, and from Rademacher's Theorem \cite[Section 3.1.2]{EvansGariepy}, its differential exists almost everywhere. Moreover, any $\mathbf{v} \in T_{\mathbf{x}}\mathcal{S}$ can be decomposed in the basis $(\partial_{1} X, \ldots, \partial_{n-1}X)$. Denoting by $V_{i}$ the component of $\partial_{i}X$ and $v_{i} = \langle \mathbf{v} ~\vert~ \partial_{i} X \rangle$, we have:
\begin{equation}
\label{eqn_composantes}
\mathbf{v} = \sum_{i = 1}^{n-1} V_{i} \partial_{i} X ~ \Longrightarrow ~v_{j} = \sum_{i = 1}^{n} V_{i} g_{ij} ~\Longrightarrow~ V_{i} = \sum_{j = 1}^{n-1} g^{ij}v_{j} ~\Longrightarrow~ \mathbf{v} =  \sum_{i = 1}^{n-1} \left( \sum_{j =1}^{n-1} g^{ij} v_{j} \right) \partial_{i} X.
\end{equation}    
In particular, we deduce $\mathrm{\mathbf{I}}(\mathbf{v},\mathbf{w}) = \sum_{i,j = 1}^{n-1} g^{ij} v_{i} w_{j}$. Then, the orthogonal of the tangent hyperplane is one dimensional. Hence, there exists a unique unit vector $\mathbf{n}$ orthogonal to the $(n-1)$ vectors $\partial_{1} X$, $\ldots$, $\partial_{n-1} X$ and pointing outwards the inner domain of $\mathcal{S}$ i.e. $\mathrm{det}(\partial_{1} X$, $\ldots$, $\partial_{n-1} X,\mathbf{n} ) > 0$. It is called the unit outer normal to the hypersurface and we have its explicit expression:
\begin{equation}
\label{expression_normale} 
\forall \mathbf{x'} \in D_{r}(\mathbf{0'}), \quad \mathbf{n}\circ X(\mathbf{x'} )  = \dfrac{1}{\sqrt{1 + \Vert \nabla \varphi (\mathbf{x'}) \Vert^{2}}} \left( \begin{array}{c} - \nabla \varphi(\mathbf{x'}) \\ 1 \\ \end{array} \right). 
\end{equation}
It is a Lipschitz continuous map, like the coefficients of the first fundamental form. In particular, it is differentiable almost everywhere and introducing the Gauss map $\mathbf{n}: \mathbf{x} \in \mathcal{S} \mapsto \mathbf{n}(\mathbf{x}) \in \mathbb{S}^{n-1}$, we can compute its differential almost everywhere called the Weingarten map:
\begin{equation}
\label{expression_weingarten_map}
\begin{array}{rrcl} 
D_{\mathbf{x}} \mathbf{n} : &  T_{\mathbf{x}} \mathcal{S} = D_{X^{-1}(\mathbf{x})}X(\mathbb{R}^{2}) & \longrightarrow & T_{\mathbf{n}(\mathbf{x})} \mathbb{S}^{n-1} = D_{X^{-1}(\mathbf{x})}(\mathbf{n} \circ X)(\mathbb{R}^{2})  \\
&  \mathbf{v} =  D_{X^{-1}(\mathbf{x})}X(\mathbf{w}) & \longmapsto & D_{\mathbf{x}} \mathbf{n}(\mathbf{v}) = D_{X^{-1}(\mathbf{x})}(\mathbf{n} \circ X)(\mathbf{w}). \\
 \end{array}  
\end{equation}
Note that $T_{\mathbf{n}(\mathbf{x})} \mathbb{S}^{n-1} = D_{X^{-1}(\mathbf{x})}(\mathbf{n} \circ X)(\mathbb{R}^{2})$ because $\mathbf{n} \circ X$ is a Lipschitz parametrization of $\mathbb{S}^{n-1}$. Since $T_{\mathbf{n}(\mathbf{x})} \mathbb{S}^{n-1} \sim \mathbf{n}(\mathbf{x})^{\perp} $ can be identified with $ T_{\mathbf{x}} \mathcal{S} $, the map $D_{\mathbf{x}} \mathbf{n}$ is an endomorphism of $T_{\mathbf{x}} \mathcal{S}$. Moreover, one can prove it is self-adjoint so it can be diagonalized to obtain $n-1$ eigenvalues denoted by $\kappa_{1}(\mathbf{x})$, $\ldots$, $\kappa_{n-1}(\mathbf{x})$ and called the principal curvatures. Recall that the eigenvalues of an endomorphism do not depend on the chosen basis and thus are really properties of the operator. This assertion also holds for the coefficients of the characteristic polynomial associated with $D_{\mathbf{x}} \mathbf{n}$ so we can introduce them:
\begin{equation} 
\label{expression_symmetric_polynomial_curvature}
\forall l \in \lbrace 0, \ldots, n-1 \rbrace, \quad H^{(l)}(\mathbf{x}) = \sum_{1 \leqslant n_{1} < \ldots < n_{l} \leqslant n-1} \kappa_{n_{1}} \left( \mathbf{x} \right) \ldots \kappa_{n_{l}} \left( \mathbf{x} \right).  
\end{equation}
In particular, $H^{(0)} = 1$, $H^{(1)} = H $ is called the scalar mean curvature, and $H^{(n-1)} = K$ refers to the Gaussian curvature:
\begin{equation}
\label{expression_scalar_mean_curvature} 
H(\mathbf{x}) = \kappa_{1}(\mathbf{x}) + \ldots + \kappa_{n-1}(\mathbf{x}) \qquad \mathrm{and} \qquad K(\mathbf{x}) = \kappa_{1}(\mathbf{x}) \kappa _{2}(\mathbf{x})\ldots \kappa_{n-1}(\mathbf{x}).  
\end{equation}
Moreover, introducing the symmetric matrix $(b_{ij})_{1 \leqslant i,j \leqslant n-1}$ defined by:
\begin{equation}
\label{expression_seconde_forme_fondamentale} 
b_{ij} =  - \langle D \mathbf{n} (\partial_{i} X) ~\vert~ \partial_{j} X \rangle= - \left\langle \partial_{i} (\mathbf{n} \circ X) ~\vert~ \partial_{j} X \right\rangle =  \dfrac{ \mathrm{Hess} ~ \varphi}{ \sqrt{1  + \Vert \nabla \varphi \Vert^{2}}} = \left\langle  \mathbf{n} \circ X  ~\vert~  \partial_{ij} X  \right\rangle ,
\end{equation}
we get from \eqref{eqn_composantes} that the Weingarten map $D \mathbf{n}$ is represented in the local basis $(\partial_{1}X, \ldots, \partial_{n-1}X)$ by the following symmetric matrix:
\begin{equation} 
\label{expression_weingarten_map_local_def}
(h_{ij})_{1 \leqslant i,j \leqslant n-1} = \left( - \sum_{k=1}^{n-1} g^{ik}b_{kj} \right) =  \left( - \sum_{k = 1}^{n-1} \left( \delta_{ik} - \dfrac{\partial_{i} \varphi \partial_{j} \varphi}{1 + \Vert \nabla \varphi \Vert^{2}} \right) \dfrac{\partial_{kj} \varphi}{\sqrt{1 + \Vert \nabla \varphi\Vert^{2}}} \right) . 
\end{equation}
Finally, we introduce the symmetric bilinear form whose representation in the local basis is $(b_{ij})$. It is called the second fundamental form of the hypersurface and it is defined by:
\begin{equation}
\label{eqn_seconde_forme_fondamentale_def}
\begin{array}{rcl} 
\mathrm{\mathbf{II}}(\mathbf{x}) :  ~  T_{\mathbf{x}}(\mathcal{S}) \times T_{\mathbf{x}}  (\mathcal{S}) & \longrightarrow & \mathbb{R} \\
 (\mathbf{v},\mathbf{w}) & \longmapsto & \displaystyle{ \langle -D_{\mathbf{x}}\mathbf{n}(\mathbf{v}) ~\vert~ \mathbf{w} \rangle = \sum_{i,j,k,l = 1}^{n-1} g^{ij} v_{j} g^{kl} w_{l} b_{il} = \sum_{i,j,k = 1}^{n-1} g^{ij} v_{j} v_{k} h_{ki} } . \\
\end{array}
\end{equation} 
We can also decompose $\partial_{ij} X$ in the basis $(\partial_{1}X, \ldots, \partial_{n-1}X,\mathbf{n})$ and its coefficients in the tangent space are the Christoffel symbols:
\[ \partial_{ij} X =  \sum_{k=1}^{n-1} \Gamma^{k}_{ij} \partial_{k} X + b_{ij} \mathbf{n}   \]   
Note that the Christoffel symbols are symmetric with respect to the lower indices: $\Gamma_{ij}^{k} = \Gamma^{k}_{ji}$. They can be expressed only in terms of coefficients of the first fundamental form:
\begin{equation}
\label{expression_christoffel_symbols} 
\Gamma^{k}_{ij} = \dfrac{1}{2} \sum_{l = 1}^{n-1} g^{kl} \left( \partial_{j} g_{li} +  \partial_{i} g_{lj} - \partial_{l} g_{ij} \right). 
\end{equation}
Like the first fundamental form, it is an intrinsic notion, which in particular do not depend on the orientation chosen for the hypersurface, while the Gauss map, the Weingarten map, and the second fundamental form does. Note that in local coordinates, the coefficients of the first fundamental form and the Gauss map are Lipschitz continuous functions i.e. $ \mathbf
n \circ X, g_{ij}, g^{ij} \in W^{1,\infty}(D_{r}(\mathbf{0'}))$. Hence, the Christoffel symbols, the Weingarten map and the coefficients of the second fundamental form exist almost everywhere and $\Gamma^{k}_{ij}, b_{ij}, h_{ij} \in L^{\infty}(D_{r}(\mathbf{0'}))$. Furthermore, one can prove that a $C^{1,1}$-hypersurface satisfies the following relations in the sense of distributions, respectively called the Gauss and Codazzi-Mainardi equations:
\begin{gather}
 \partial_{l} \Gamma_{ij}^{k} -  \partial_{j} \Gamma_{il}^{k}  + \sum_{m=1}^{n-1} \left( \Gamma^{m}_{ij} \Gamma_{ml}^{k} - \Gamma_{il}^{m} \Gamma_{mj}^{k} \right) = \sum_{m=1}^{n-1} g^{km} \left( b_{ij} b_{ml}  - b_{il} b_{mj}  \right)  \label{expression_gauss_equation} \\
 \partial_{k} b_{ij} - \partial_{j} b_{ik} = \sum_{l = 1}^{n-1} \left( \Gamma_{ik}^{l} b_{lj} - \Gamma_{ij}^{l} b_{lk} \right) . \label{expression_codazzi_mainardi_equation}
\end{gather}
In fact, the converse statement is also true in $\mathbb{R}^{3}$: these equations characterize uniquely a surface and it is referred as the Fundamental Theorem of Surface Theory, valid with $C^{1,1}$-regularity \cite{Mardare}. Given a simply-connected open subset $\omega \subseteq \mathbb{R}^{2}$, a symmetric positive-definite matrix $(g_{ij})_{1 \leqslant i, j \leqslant 2} \in W^{1, \infty}(\omega)$ and a symmetric matrix $(b_{ij})_{1 \leqslant i,j \leqslant 2} \in L^{\infty}(\omega)$ satisfying \eqref{expression_gauss_equation} and \eqref{expression_codazzi_mainardi_equation} in the sense of distributions, then there exists an injective $C^{1,1}$-immersion $X : \omega \rightarrow \mathbb{R}^{3}$, unique up to proper isometries of $\mathbb{R}^{3}$, such that the surface $\mathcal{S}: = X(\omega)$ has $(g_{ij})$ and $(b_{ij})$ as coefficients of the first and second fundamental forms. To conclude, we recall that $A(.)$ (respectively $V(.)$) refers to the $n-1$(resp. $n$)-dimensional Hausdorff measure. The integration is always be done with respect to $A$ and the infinitesimal area element is given by $(dA \circ X) (\mathbf{x'}) = \sqrt{\mathrm{\mathrm{det}}(g_{ij})} d\mathbf{x'} = \sqrt{1 + \Vert  \nabla \varphi(\mathbf{x'}) \Vert^{2}} d\mathbf{x'}$. We refer to \cite{DoCarmo,MontielRos} for a more detailed exposition on all the notions quickly introduced here.

\subsection{Geometric functionals involving the position and the normal}  
\begin{proposition}
\label{prop_continuite_ordre_zero}
Under assumption \ref{hypothese_continuite}, for any continuous map $j: \mathbb{R}^{n} \times \mathbb{S}^{n-1} \rightarrow \mathbb{R}$, we have: 
\[ \lim_{i \rightarrow + \infty} \int_{\partial \Omega_{i}} j \left[ \mathbf{x},\mathbf{n} \left(\mathbf{x}\right) \right] dA \left( \mathbf{x} \right) = \int_{\partial \Omega} j \left[ \mathbf{x},\mathbf{n} \left(\mathbf{x}\right) \right] dA \left( \mathbf{x} \right). \]
In particular, the area and the volume are continuous: $A( \partial \Omega_{i} ) \longrightarrow A( \partial \Omega )$ and $V( \Omega_{i} ) \longrightarrow V( \Omega)$.
\end{proposition}

\begin{remark}
Note that the above result states the convergence of $(\partial \Omega_{i})_{i  \in \mathbb{N}}$ to $\partial \Omega$ in the sense of oriented varifolds \cite[Appendix B]{BellettiniMugnai} \cite{SimonBook}. Similar results were obtained in \cite{GuoYang2013}. Moreover, the continuity of volume and the lower semi-continuity of area were already implied by the convergence in the sense of characteristic functions (Point (vi) in Proposition \ref{prop_compacite_boule}) \cite[Proposition 2.3.6]{HenrotPierre}.
\end{remark}

\begin{proof}
Consider Assumption \ref{hypothese_continuite}. Hence, from Theorem \ref{thm_parametrisation_locale_i}, the boundaries $(\partial \Omega_{i})_{i \in \mathbb{N}}$ are locally parametrized by graphs of $C^{1,1}$-maps $\varphi_{i}$ that converge strongly in $C^{1}$ and weakly-star in $W^{2,\infty}$ to the map $\varphi$ associated with $\partial \Omega$. We now detail the procedure which allows to pass from this local result to the global one thanks to a suitable partition of unity. For any $\mathbf{x} \in \partial \Omega$, we introduce the cylinder $\mathcal{C}_{\tilde{r},\varepsilon}(\mathbf{x})$ defined by \eqref{eqn_cylinder} and we assume that $\tilde{r} > 0$ is the one given in Theorem \ref{thm_parametrisation_locale_i}. In particular, it only depends on $\varepsilon$. Since $\partial \Omega$ is compact, there exists a finite number $K \geqslant 1$ of points written $\mathbf{x}_{1}, \ldots, \mathbf{x}_{K}$, such that $ \partial \Omega \subseteq \bigcup_{k=1}^{K} \mathcal{C}_{\frac{\tilde{r}}{2},\frac{\varepsilon}{2}}(\mathbf{x}_{k}) $. We set $\delta = \min(\frac{\tilde{r}}{2},\frac{\varepsilon}{2}) > 0$. From the triangle inequality, the tubular neighbourhood $\mathcal{V}_{\delta}(\partial \Omega) = \lbrace \mathbf{y} \in \mathbb{R}^{n}, \quad d(\mathbf{y},\partial \Omega) < \delta \rbrace $ has its closure embedded in $ \bigcup_{k=1}^{K} \mathcal{C}_{\tilde{r},\varepsilon}(\mathbf{x}_{k}) $. Then, we can introduce a partition of unity on this set. There exists $K$ non-negative $C^{\infty}$-maps $\xi^{k} $ with compact support in $ \mathcal{C}_{\tilde{r},\varepsilon}(\mathbf{x}_{k}) $ and such that $\sum_{k=1}^{K} \xi^{k}(\mathbf{x}) = 1$ for any point $\mathbf{x} \in \mathcal{V}_{\delta}(\partial \Omega)$. Now, we can apply Theorem \ref{thm_parametrisation_locale_i} to the $K$ points $\mathbf{x}_{k}$. There exists $K$ integers $I_{k} \in \mathbb{N}$ and some maps $\varphi_{i}^{k} : \overline{D_{\tilde{r}}}(\mathbf{x}_{k}) \mapsto ]-\varepsilon, \varepsilon[$, with $i \geqslant I_{k}$ and $K \geqslant k \geqslant 1$, such that:
\[ \left\lbrace \begin{array}{rcl}
\displaystyle { \partial \Omega_{i} \cap \overline{\mathcal{C}_{\tilde{r},\varepsilon}(\mathbf{x}_{k})} } &=& \displaystyle{ \left\lbrace (\mathbf{x'},\varphi_{i}^{k}(\mathbf{x'})), \quad \mathbf{x'} \in \overline{D_{\tilde{r}}}(\mathbf{x}_{k}) \right\rbrace } \\
& & \\
\displaystyle { \Omega_{i} \cap \overline{\mathcal{C}_{\tilde{r},\varepsilon}(\mathbf{x}_{k})} } &=& \displaystyle{ \left\lbrace (\mathbf{x'},x_{n}), \quad \mathbf{x'} \in \overline{D_{\tilde{r}}}(\mathbf{x}_{k}) \quad \mathrm{and} \quad -\varepsilon \leqslant x_{n} < \varphi_{i}^{k}(\mathbf{x'}) \right\rbrace }.  \\
\end{array} \right. \] 
Moreover, the $K$ sequences of functions $(\varphi_{i}^{k})_{i \geqslant I_{k}}$ and $(\nabla \varphi_{i}^{k})_{i \geqslant I_{k}}$ converge uniformly on $\overline{D_{\tilde{r}}}(\mathbf{x}_{k})$ respectively to the maps $\varphi^{k}$ and $\nabla \varphi^{k}$ associated with $\partial \Omega$ at each point $\mathbf{x}_{k}$. From the Hausdorff convergence of the boundaries (Point (ii) in Proposition \ref{prop_compacite_boule}), there also exists $I_{0} \in \mathbb{N}$ such that for any integer $i \geqslant I_{0} $, we have $\partial \Omega_{i} \in \mathcal{V}_{\delta}(\partial \Omega)$. Hence, we set $I = \max_{0 \leqslant k \leqslant K } I_{k}$, which thus only depends on $(\Omega_{i})_{i \in \mathbb{N}}$, $\Omega$ and $\varepsilon$. Then, we deduce that for any integer $i \geqslant I$, we have:
\[ \begin{array}{l} 
\displaystyle{ J(\partial \Omega_{i} ) \quad : = \quad \int_{\partial \Omega_{i}} j \left[ \mathbf{x}, \mathbf{n} \left(\mathbf{x} \right) \right] dA(\mathbf{x}) \quad = \quad  \int_{\partial \Omega_{i} \cap \mathcal{V}_{\delta}(\partial \Omega)} j \left[ \mathbf{x},\mathbf{n}\left(\mathbf{x} \right) \right] dA(\mathbf{x}) } \\
 \\
\displaystyle{ \qquad \quad = \quad \int_{\partial \Omega_{i}} \left( \sum_{k=1}^{K} \xi^{k}\left(\mathbf{x}\right) \right) j \left[ \mathbf{x}, \mathbf{n}\left(\mathbf{x}\right) \right] dA(\mathbf{x}) \quad  = \quad   \sum_{k=1}^{K}  \int_{\partial \Omega_{i} \cap \mathcal{C}_{r,\varepsilon}(\mathbf{x}_{k})} \xi^{k} \left(\mathbf{x} \right) j \left[ \mathbf{x},\mathbf{n}\left(\mathbf{x}\right) \right] dA(\mathbf{x}) } \\
 \\
 \displaystyle{ \quad = \quad \sum_{k=1}^{K}  \int_{D_{\tilde{r}}(\mathbf{x}_{k})} \xi^{k} \left( \begin{array}{c} \mathbf{x'} \\  \varphi_{i}^{k}\left( \mathbf{x'} \right) \\ \end{array} \right)   j \left[ \left(  \begin{array}{c} \mathbf{x'} \\ \varphi_{i}^{k} \left(\mathbf{x'} \right) \\ \end{array} \right), \left( \begin{array}{c} \frac{ -\nabla \varphi^{k}_{i} \left( \mathbf{x'} \right)}{\sqrt{1 + \Vert \nabla \varphi^{k}_{i} \left( \mathbf{x'} \right) \Vert^{2}}} \\ \frac{1}{\sqrt{1 + \Vert \nabla \varphi^{k}_{i} \left( \mathbf{x'} \right) \Vert^{2}}} \\ \end{array} \right)  \right] \sqrt{1 + \Vert \nabla \varphi^{k}_{i} \left( \mathbf{x'} \right) \Vert^{2}} d\mathbf{x'} } \\
\end{array}    \]
The last equality comes from \cite[Proposition 5.13]{MontielRos} and Relation \eqref{expression_normale}. The uniform convergence of the $K$ sequences $(\varphi^{k}_{i})_{i \geqslant I}$ and $( \nabla \varphi^{k}_{i})_{i \geqslant I}$ on the compact set $\overline{D_{\tilde{r}}}(\mathbf{x}_{k})$ combined with the continuity of $j$ and $(\xi^{k})_{1 \leqslant k \leqslant K}$ allows one to let $i \rightarrow \infty$ in the above expression. Observing that the limit expression obtained is equal to $J(\partial \Omega)$, we proved that the functional $J$ is continuous. Finally, for the area, take $j \equiv 1$ and for the volume, applying the Divergence Theorem, take $j[\mathbf{x},\mathbf{n}(\mathbf{x})] = \frac{1}{n} \langle \mathbf{x} ~\vert~ \mathbf{n}(\mathbf{x}) \rangle$.
\end{proof}

\begin{proposition}
\label{prop_continuite_ordre_zero_i}
Consider Assumption \ref{hypothese_continuite} and some continuous maps $j, j_{i}: \mathbb{R}^{n} \times \mathbb{S}^{n-1} \rightarrow \mathbb{R}$ such that $(j_{i})_{i \in \mathbb{N}}$ is uniformly bounded on $\overline{B} \times \mathbb{S}^{n-1}$ and diagonally converges to $j$ in the sense of Definition \ref{definition_convergence_diagonale}. Then, we have: 
\[ \lim_{i \rightarrow + \infty} \int_{\partial \Omega_{i}} j_{i} \left[ \mathbf{x},\mathbf{n} \left(\mathbf{x}\right) \right] dA \left( \mathbf{x} \right) = \int_{\partial \Omega} j \left[ \mathbf{x},\mathbf{n} \left(\mathbf{x}\right) \right] dA \left( \mathbf{x} \right). \]
\end{proposition}

\begin{proof}
The proof is identical to the one of Proposition \ref{prop_continuite_ordre_zero}. Using the same partition of unity and the same notation, we get that $\int_{\partial \Omega_{i}}j_{i}[\mathbf{x},\mathbf{n}(\mathbf{x})]dA(\mathbf{x})$ is equal to:
 \[  \sum_{k=1}^{K}  \int_{D_{\tilde{r}}(\mathbf{x}_{k})} \xi^{k} \left( \begin{array}{c} \mathbf{x'} \\  \varphi_{i}^{k}\left( \mathbf{x'} \right) \\ \end{array} \right)   j_{i} \left[ \left(  \begin{array}{c} \mathbf{x'} \\ \varphi_{i}^{k} \left(\mathbf{x'} \right) \\ \end{array} \right), \left( \begin{array}{c} \frac{ -\nabla \varphi^{k}_{i} \left( \mathbf{x'} \right)}{\sqrt{1 + \Vert \nabla \varphi^{k}_{i} \left( \mathbf{x'} \right) \Vert^{2}}} \\ \frac{1}{\sqrt{1 + \Vert \nabla \varphi^{k}_{i} \left( \mathbf{x'} \right) \Vert^{2}}} \\ \end{array} \right)  \right] \sqrt{1 + \Vert \nabla \varphi^{k}_{i} \left( \mathbf{x'} \right) \Vert^{2}} d\mathbf{x'}. \]
Then, instead of using the uniform convergence of each integrand on a compact set as it is the case in Proposition \ref{prop_continuite_ordre_zero}, we apply instead Lebesgue's Dominated Convergence Theorem. Indeed, the diagonal convergence ensures the pointwise convergence of each integrand, which are also, using the other hypothesis, uniformly bounded. Hence, we can let $i \rightarrow + \infty$ in the above expression.  
\end{proof}

\begin{definition}
\label{definition_convergence_diagonale_vector_fields}
Let $ \mathcal{S}, \mathcal{S}_{i}$ be some non-empty compact $C^{1}$-hypersurfaces of $\mathbb{R}^{n}$ such that $(\mathcal{S}_{i})_{i \in \mathbb{N}}$ converges to $\mathcal{S}$ for the Hausdorff distance: $d_{H}(\mathcal{S}_{i},\mathcal{S})\longrightarrow_{i \rightarrow + \infty}  0$. On each hypersurface $\mathcal{S}_{i}$, we also consider a continuous vector field $\mathbf{V}_{i} : \mathbf{x} \in \mathcal{S}_{i} \mapsto \mathbf{V}_{i}(\mathbf{x}) \in T_{\mathbf{x}} \mathcal{S}_{i}$. We say that $(\mathbf{V}_{i})_{i \in \mathbb{N}}$ is diagonally converging to a vector field on $\mathcal{S}$ denoted by $\mathbf{V} : \mathbf{x} \in \mathcal{S} \mapsto \mathbf{V}(\mathbf{x}) \in T_{\mathbf{x}} \mathcal{S}$ if for for any point $\mathbf{x} \in \mathcal{S}$ and for any sequence of points $\mathbf{x}_{i} \in \mathcal{S}_{i}$ that converges to $\mathbf{x}$, we have $ \Vert \mathbf{V}_{i}(\mathbf{x}_{i}) - \mathbf{V}(\mathbf{x}) \Vert \longrightarrow_{i \rightarrow + \infty} 0$. 
\end{definition}

\begin{corollary}
\label{coro_continuite_ordre_zero_vector_fields}
Let $ \varepsilon, B, \Omega, (\Omega_{i})_{i \in \mathbb{N}} $ be as in Assumption \ref{hypothese_continuite}, and consider some continuous vector fields $\mathbf{V}_{i}$ on $\partial \Omega_{i}$ converging to a continuous vector field $\mathbf{V}$ on $\partial \Omega$ as in Definition \ref{definition_convergence_diagonale_vector_fields}. We also assume that $ (\mathbf{V}_{i})_{i \in \mathbb{N}} $ is uniformly bounded. Considering a continuous map $j : \mathbb{R}^{n} \times \mathbb{S}^{n-1} \times \mathbb{R}^{n} \rightarrow \mathbb{R}$, then we have:
\[ \lim_{i \rightarrow + \infty} \int_{\partial \Omega_{i}} j\left[ \mathbf{x}, \mathbf{n}\left(\mathbf{x}\right), \mathbf{V}_{i}\left(\mathbf{x} \right) \right] dA (\mathbf{x} ) = \int_{\partial \Omega} j\left[ \mathbf{x}, \mathbf{n}\left(\mathbf{x}\right), \mathbf{V}\left(\mathbf{x} \right) \right] dA (\mathbf{x} ).    \]
Of course, this continuity result can be extended to a finite number of vector fields.
\end{corollary}

\begin{proof}
We only have to check that the maps $j_{i} : (\mathbf{x},\mathbf{u}) \in \partial \Omega_{i} \times \mathbb{S}^{n-1} \rightarrow j[\mathbf{x},\mathbf{u}, \mathbf{V}_{i}(\mathbf{x})]$ can be extended to $\mathbb{R}^{n} \times \mathbb{S}^{n-1}$ such that their extension satisfy the hypothesis of Proposition \ref{prop_continuite_ordre_zero_i}. This is a standard procedure \cite[Section 5.4.1]{HenrotPierre}. Using the partition of unity given in Proposition \ref{prop_continuite_ordre_zero} and introducing the $C^{1,1}$-diffeomorphisms $\Psi^{k}_{i} : (\mathbf{x'},x_{n}) \in \overline{\mathcal{C}_{r,\varepsilon}}(\mathbf{x_{k}}) \mapsto (\mathbf{x'},\varphi^{k}_{i}(\mathbf{x'}) - x_{n})$, we can set:
\[ \forall  (\mathbf{x},\mathbf{u}) \in  \mathbb{R}^{n} \times \mathbb{S}^{n-1}, \quad j_{i}(\mathbf{x},\mathbf{u}) = \sum_{k = 1}^{K} \xi^{k}(\mathbf{x}) j \left[ (\Psi^{k}_{i})^{-1} \circ \Pi_{\mathbf{x}_{k}} \circ \Psi^{k}_{i}(\mathbf{x}), \mathbf{u}, \mathbf{V}_{i}\circ (\Psi^{k}_{i})^{-1} \circ \Pi_{\mathbf{x}_{k}} \circ \Psi^{k}_{i}(\mathbf{x}) \right]. \]
We recall that $\Pi_{\mathbf{x}_{k}} $ is defined by \eqref{eqn_projection}. Finally, $(j_{i})_{i \in \mathbb{N}}$ diagonally converges to the extension of $(\mathbf{x},\mathbf{u}) \mapsto j[\mathbf{x},\mathbf{u}, \mathbf{V}(\mathbf{x})]$, since $(V_{i})_{i \in \mathbb{N}}$ is diagonally converging to $V$. Moreover, $(\Omega_{i})_{i \in \mathbb{N}} \subset B$, the Gauss map is always valued in $\mathbb{S}^{n-1}$, and $(V_{i})_{i \in \mathbb{N}}$ is uniformly bounded. Hence, $(\mathbf{x},\mathbf{n}_{\partial \Omega_{i}}(\mathbf{x}), \mathbf{V}_{i}(\mathbf{x}) )$ is valued in a compact set. Since $j$ is continuous on this compact set, it is bounded and $(j_{i})_{i \in \mathbb{N}}$ is thus uniformly bounded on $\overline{B} \times \mathbb{S}^{n-1}$. Finally, we can apply Proposition \ref{prop_continuite_ordre_zero_i} to let $i \rightarrow + \infty$.
\end{proof}

\subsection{Some linear functionals involving the second fundamental form} 
From Theorem \ref{thm_parametrisation_locale_i}, we only have the $L^{\infty}$-weak-star convergence of the coefficients associated with the Hessian of the local maps $\varphi_{i}^{k}$ so we consider here the case of functionals whose expressions in the parametrization are linear in $\partial_{pq} \varphi_{i}^{k}$. This is the case for the scalar mean curvature and the second fundamental form of two vector fields.

\begin{proposition}
\label{prop_continuite_ordre_un}
Consider Assumption \ref{hypothese_continuite} and a continuous map $j: \mathbb{R}^{n} \times \mathbb{S}^{n-1} \rightarrow \mathbb{R}$. Then, the following functional is continuous:
\[ \lim_{i \rightarrow + \infty} \int_{\partial \Omega_{i}} H \left( \mathbf{x} \right) j \left[ \mathbf{x},\mathbf{n}\left(\mathbf{x} \right) \right]dA\left(\mathbf{x} \right) = \int_{\partial \Omega} H \left( \mathbf{x} \right) j \left[ \mathbf{x},\mathbf{n}\left(\mathbf{x} \right) \right]dA\left(\mathbf{x} \right). \]
\end{proposition}

\begin{proof}
The proof is identical to the one of Proposition \ref{prop_continuite_ordre_zero}. Using the same notation and the same partition of unity, we have to check that in the parametrization $X_{i}^{k}: \mathbf{x'} \in \overline{D_{\tilde{r}}}(\mathbf{x}_{k}) \mapsto (\mathbf{x'},\varphi_{i}^{k}(\mathbf{x'}))$, the scalar mean curvature $L^{\infty}$-weakly-star converges. It is the trace \eqref{expression_scalar_mean_curvature} of the Weingarten map defined by \eqref{expression_weingarten_map} so Relation \eqref{expression_weingarten_map_local_def} gives:
\begin{equation}
\label{expression_scalar_mean_curvature_local} 
(H\circ X_{i}^{k}) = - \sum_{p,q= 1}^{n-1} g^{pq}b_{qp} =  - \sum_{p,q= 1}^{n-1} \left( \delta_{pq} - \dfrac{\partial_{p} \varphi_{i}^{k} \partial_{q} \varphi_{i}^{k}}{1 + \Vert \nabla \varphi_{i}^{k} \Vert^{2}} \right) \left( \dfrac{\partial_{pq} \varphi_{i}^{k}}{\sqrt{1 + \Vert \nabla \varphi_{i}^{k} \Vert^{2}}} \right) . 
\end{equation}
Using Theorem \ref{thm_parametrisation_locale_i}, the $K$ sequences $(H\circ X_{i}^{k})_{i \in \mathbb{N}}$ weakly-star converge in $L^{\infty}(D_{\tilde{r}}(\mathbf{x}_{k}))$ respectively to $H \circ X^{k}$. The remaining part of each integrand below uniformly converges to the one of $\partial \Omega$ so we can let $i \rightarrow + \infty$ inside:
\[ \sum_{k = 1}^{K} \int_{D_{\tilde{r}}(\mathbf{x}_{k})} (H \circ X^{k}_{i})(\mathbf{x'}) (\xi^{k} \circ X^{k}_{i})(\mathbf{x'})j[X^{k}_{i}(\mathbf{x'}),(\mathbf{n}\circ X^{k}_{i})(\mathbf{x'})] (dA\circ X^{k}_{i})(\mathbf{x'}), \] 
to get the limit asserted in Proposition \ref{prop_continuite_ordre_un}.
\end{proof} 

\begin{corollary}
\label{coro_continuite_ordre_un_convex}
Consider Assumption \ref{hypothese_continuite} and a continuous map $j: \mathbb{R}^{n} \times \mathbb{S}^{n-1} \times \mathbb{R} \rightarrow \mathbb{R}$ which is convex in its last variable. Then, we have:
\[ \int_{\partial \Omega}  j \left[ \mathbf{x}, \mathbf{n} \left( \mathbf{x} \right), H \left(\mathbf{x} \right) \right] dA \left(\mathbf{x} \right)  \leqslant  \liminf_{i \rightarrow + \infty}  \int_{\partial \Omega_{i}}  j \left[ \mathbf{x}, \mathbf{n} \left( \mathbf{x} \right),H \left(\mathbf{x} \right) \right] dA \left(\mathbf{x} \right) . \]
\end{corollary}

\begin{remark}
In particular, this result implies that the Helfrich \eqref{expression_helfrich_energy} and Willmore functionals \eqref{expression_willmore_energy} are lower semi-continuous, and so does the $p$-th power norm of mean curvature $\int \vert H \vert^{p} dA$, $p \geqslant 1$. Note that we are able to treat the critical case $p = 1$, while it is often excluded from many statements of geometric measure theory \cite[Example 4.1]{Delladio} \cite[Definition 2.2]{Mondino} \cite[Definition 4.1.2]{Hutchinson}. 
\end{remark}

\begin{proof}
The arguments are standard \cite[Theorem 2.2.1]{Evans}. We only sketch the proof. First, assume that $j$ is the maximum of finitely many affine functions according to its last variable:
\begin{equation}
\label{expression_continuite_ordre_un_convex} 
\forall t \in \mathbb{R}, \quad j(\mathbf{x}, \mathbf{n}(\mathbf{x}), t ) = \max_{0 \leqslant l \leqslant L} j_{l} \left[ \mathbf{x}, \mathbf{n}(\mathbf{x}) \right] t + \tilde{j}_{l} \left[ \mathbf{x}, \mathbf{n}(\mathbf{x}) \right].  
\end{equation}
For simplicity, let us assume that $j$ only depends on the position. Using a partition of unity as in Proposition \ref{prop_continuite_ordre_zero}, we introduce the local parametrizations $X^{k} : \mathbf{x'} \in D_{\tilde{r}}(\mathbf{x}_{k}) \mapsto (\mathbf{x'},\varphi^{k}(\mathbf{x'})) $ and we make a partition of the set $D_{\tilde{r}}(\mathbf{x}_{k})$ into $L$ disjoints sets. We define for any $ l \in \lbrace 1, \ldots L \rbrace$:
\[   D_{l}^{k} = \left\lbrace \mathbf{x'} \in D_{\tilde{r}}(\mathbf{x}_{k}),~  j \left[ X^{k} \left( \mathbf{x'} \right),  \left(H \circ X^{k} \right) \left(\mathbf{x'} \right) \right]  = j_{l} \left[ X^{k} \left(\mathbf{x'}\right) \right] H \left[ X^{k} \left(\mathbf{x'} \right) \right]  +  \tilde{j}_{l} \left[ X^{k} \left(\mathbf{x'} \right) \right]   \right\rbrace.  \]
Then, applying Proposition \ref{prop_continuite_ordre_un} and following \cite[below (2.9)]{Evans}, we have successively:
\[  \begin{array}{rcl}
\displaystyle{ \int_{\partial \Omega} j[\mathbf{x},H(\mathbf{x})] dA (\mathbf{x}) } & = & \displaystyle{ \sum_{k = 1}^{K} \int_{D_{\tilde{r}}(\mathbf{x}_{k})} j[X^{k},(H \circ X^{k})](dA \circ X^{k}) } \\
& & \\
& = & \displaystyle{ \sum_{k = 1}^{K} \sum_{l = 1}^{L} \int_{D^{k}_{l}} \left( j_{l}[X^{k}] H[X^{k}] + \tilde{j}_{l}[X^{k}] \right) (dA \circ X^{k}) } \\
& & \\
& = & \displaystyle{ \sum_{k = 1}^{K} \sum_{l = 1}^{L} \lim_{i \rightarrow + \infty} \int_{D_{l}^{k}}  \left( j_{l}[X_{i}^{k}] H[X_{i}^{k}] + \tilde{j}_{l}[X_{i}^{k}] \right) (dA \circ X_{i}^{k})   } \\
& & \\
& \leqslant & \displaystyle{   \sum_{k = 1}^{K} \sum_{l = 1}^{L} \liminf_{i \rightarrow + \infty} \int_{D^{k}_{l}}  j[X_{i}^{k},(H \circ X_{i}^{k})](dA \circ X_{i}^{k}) } \\
& & \\
& \leqslant &  \displaystyle{ \liminf_{i \rightarrow + \infty} \int_{\partial \Omega_{i}} j[\mathbf{x},H(\mathbf{x})] dA (\mathbf{x}).  } \\
\end{array} \]
The result holds for maps $j$ that are maximum of finitely many planes. In the general case, we write $j = \lim_{L \rightarrow + \infty} j_{L} $ where $j_{L}$ is defined by \eqref{expression_continuite_ordre_un_convex} and apply the Monotone Convergence Theorem.
\end{proof}

\begin{proposition}
Consider Assumption \ref{hypothese_continuite} and some continuous maps $j, j_{i}: \mathbb{R}^{n} \times \mathbb{S}^{n-1} \rightarrow \mathbb{R}$ such that $(j_{i})_{i \in \mathbb{N}}$ is uniformly bounded on $\overline{B} \times \mathbb{S}^{n-1}$ and diagonally converges to $j$ in the sense of Definition \ref{definition_convergence_diagonale}. Then, the following functional is continuous: 
\[ \lim_{i \rightarrow + \infty} \int_{\partial \Omega_{i}} H \left( \mathbf{x} \right) j_{i} \left[ \mathbf{x},\mathbf{n}\left(\mathbf{x} \right) \right]dA\left(\mathbf{x} \right) = \int_{\partial \Omega} H \left( \mathbf{x} \right) j \left[ \mathbf{x},\mathbf{n}\left(\mathbf{x} \right) \right]dA\left(\mathbf{x} \right). \]
\end{proposition}

\begin{remark}
\label{remarque_vector_fields}
As in Corollary \ref{coro_continuite_ordre_zero_vector_fields}, we can consider here that $j_{i} $ is a continuous map of the position, the normal, and a finite number of uniformly bounded vector fields diagonally converging in the sense of Definition \ref{definition_convergence_diagonale_vector_fields}.
\end{remark}

\begin{proof}
The proof is identical to the one of Proposition \ref{prop_continuite_ordre_un}. Writing the functional in terms of local parametrizations, it remains to check that we can let $i \rightarrow + \infty$ in each integral. From \eqref{expression_scalar_mean_curvature_local}, $(H\circ X_{i}^{k})_{i \in \mathbb{N}}$ weakly-star converges in $L^{\infty}(\overline{D_{\tilde{r}}}(\mathbf{0'}))$ to $H \circ X^{k}$, while the remaining part of the integrand is strongly converging in $L^{1}(\overline{D_{\tilde{r}}}(\mathbf{0'}))$, since the hypothesis allows one to apply Lebesgue's Dominated Convergence Theorem. Hence, the functional is continuous. 
\end{proof} 

\begin{proposition}
\label{prop_continuite_seconde_forme_fondamentale}
Consider Assumption \ref{hypothese_continuite} and some uniformly bounded continuous vector fields $\mathbf{V}_{i}$ and $ \mathbf{W}_{i}$ on $\partial \Omega_{i}$ that are diagonally converging to continuous vector fields $\mathbf{V}$ and $\mathbf{W}$ on $\partial \Omega$ in the sense of Definition \ref{definition_convergence_diagonale_vector_fields}. Let $j, j_{i} : \mathbb{R}^{n} \times \mathbb{S}^{n-1} \rightarrow \mathbb{R}$ be continuous maps such that $(j_{i})_{i \in \mathbb{N}}$ is uniformly bounded on $\overline{B} \times \mathbb{S}^{n-1}$ and diagonally converges to $j$ as in Definition \ref{definition_convergence_diagonale}. Then, we have:
\[ \lim_{i \rightarrow + \infty} \int_{\partial \Omega_{i}} \mathrm{\mathbf{II}} \left(\mathbf{x}\right) \left[ \mathbf{V}_{i} \left( \mathbf{x} \right), \mathbf{W}_{i}\left(\mathbf{x} \right) \right] j_{i} \left[ \mathbf{x}, \mathbf{n} \left( \mathbf{x} \right) \right] dA \left( \mathbf{x} \right) = \int_{\partial \Omega} \mathrm{\mathbf{II}} \left(\mathbf{x}\right) \left[ \mathbf{V} \left( \mathbf{x} \right), \mathbf{W} \left(\mathbf{x} \right) \right] j \left[ \mathbf{x}, \mathbf{n} \left( \mathbf{x} \right) \right] dA \left( \mathbf{x} \right) .  \]
\end{proposition}

\begin{remark}
\label{remarque_lower_semicontinuity}
Note that if $j_{i} = j$ for any $i \in \mathbb{N}$, then the above assertion states that a functional which is linear in the second fundamental form is continuous. Hence, using the same arguments than in Corollary \ref{coro_continuite_ordre_un_convex}, any functional whose integrand is a continuous map of the position, the normal, and the second fundamental form, convex in its last variable, is lower semi-continuous. 
\end{remark}

\begin{proof}
We write the integral in terms of local parametrizations and check that we can let $i \rightarrow + \infty$. In the local basis $(\partial_{1}X_{i}^{k}, \ldots, \partial_{n-1}X_{i}^{k})$, using \eqref{eqn_seconde_forme_fondamentale_def}, the second fundamental form takes the form:  
\[ \left( \mathrm{\mathbf{II}}\circ X_{i}^{k} \right) \left( \mathbf{V}_{i} \circ X_{i}^{k}, \mathbf{W}_{i} \circ X_{i}^{k} \right) = \sum_{p,q,r,s = 1}^{n-1} \left\langle \mathbf{V}_{i} \circ X_{i}^{k} ~\vert~ \partial_{p} X_{i}^{k} \right\rangle g^{pq}b_{qr}g^{rs} \left\langle \mathbf{W}_{i} \circ X_{i}^{k} ~\vert~ \partial_{s} X_{i}^{k} \right\rangle.     \]
Hence, each integrand is the product of $g^{pq}b_{qr}g^{rs}$ with a remaining term. Using the assumptions, the convergence results of Theorem \ref{thm_parametrisation_locale_i}, and Lebesgue's Dominated Convergence Theorem, we get that $g^{pq}b_{qr}g^{rs}$ weakly-star converges in $L^{\infty}$, while the remaining term $L^{1}$-strongly converges.
\end{proof}

\subsection{Some non-linear functionals involving the second fundamental form}
\label{section_continuite_non_lineaire}
All the previous continuity results were obtained by expressing the integrals in the parametrizations associated with a suitable partition of unity, and by observing that each integrand is the product of $b_{pq}$ converging $L^{\infty}$-weakly-star with a remaining term converging $L^{1}$-strongly. We are wondering here if a non-linear function such as the determinant of the $(b_{pq})$ can also $L^{\infty}$-weakly-star converge. Note that the convergence is in $L^{\infty}$ and \textit{not} in $W^{1,p}$ so we cannot use e.g. \cite[Section 8.2.4.b]{Evansbook}.
\bigskip

However, the coefficients of the first and second fundamental forms are not random coefficients. They characterize the hypersurfaces through the Gauss-Codazzi-Mainardi equations \eqref{expression_gauss_equation} and \eqref{expression_codazzi_mainardi_equation}. Hence, using the differential structure of these equations, we want to obtain the $L^{\infty}$-weak-star convergence of non-linear functions of the $b_{pq}$. This is done by considering a generalization of the Div-Curl Lemma due to Tartar. We refer to \cite[Section 5.5]{Evans} for references and it states as follows.

\begin{proposition}[\textbf{Tartar 1979}]
\label{prop_tartar}
Let $n \geqslant 3$ and $U \subset \mathbb{R}^{n-1}$ be open and bounded with smooth boundary. Let us consider a sequence of maps $(u_{i})_{i \in \mathbb{N}}$ weakly-star converging to $u$ in $L^{\infty}(U,\mathbb{R}^{M})$, $M \geqslant 1$, and a continuous functional $F : \mathbb{R}^{M} \rightarrow \mathbb{R}$ such that $(F(u_{i}))_{i \in \mathbb{N}}$ is weakly-star converging in $L^{\infty}(U, \mathbb{R})$. Let us suppose we are given $P$ first-order constant coefficient differential operators $A^{p} v : = \sum_{q = 1}^{n-1} \sum_{m=1}^{M} a_{mq}^{p} \partial_{q} v_{m}$ so that the sequences $(A^{p}u_{i})_{i \in \mathbb{N}}$ lies in a compact subset of $H^{-1}(U)$. We also assume that $(u_{i})_{i \in \mathbb{N}} $ is almost everywhere valued in $K$ for some given compact set $K\subset \mathbb{R}^{M}$. We introduce the following wave cone:
\[ \Lambda = \left\lbrace \lambda \in \mathbb{R}^{M} ~\vert~ \exists \mu \in \mathbb{R}^{n-1} \backslash \lbrace \mathbf{0'} \rbrace, \forall p \in \lbrace 1, \ldots P \rbrace, \quad  \sum_{q = 1}^{n-1} \sum_{m=1}^{M} a_{mq}^{p} \lambda_{m} \mu_{q} = 0 \right\rbrace.   \]
If $F$ is a quadratic form and $F = 0$ on $ \Lambda$, then the weak-star limit of $(F(u_{i}))_{i \in \mathbb{N}}$ is $F(u)$.
\end{proposition}

We now treat the case of $\mathbb{R}^{3}$ to get familiar with the notation and observe how Proposition \ref{prop_tartar} can be used here to obtain the $L^{\infty}$-weak-star convergence of the Gaussian curvature $K = \kappa_{1} \kappa_{2}$. Let $n = 3$, $U = D_{\tilde{r}}(\mathbf{x}_{k})$, and $u_{i} : \mathbf{x'} \mapsto (b_{pq}) \in \mathbb{R}^{2^{2}} $ defined by \eqref{expression_seconde_forme_fondamentale} with $X_{i}^{k} : \mathbf{x'}  \mapsto (\mathbf{x'},\varphi_{i}^{k}(\mathbf{x'})) \in \partial \Omega_{i} $. First, we show that the assumptions of Proposition \ref{prop_tartar} are satisfied. From Theorem \ref{thm_parametrisation_locale_i}, $(u_{i})_{i \in \mathbb{N}}$ $L^{\infty}(U)$-weakly-star converges to $u$ and it is uniformly bounded so it is valued in a compact set. Moreover, in the case $n = 3$, there are only two Codazzi-Mainardi equations \eqref{expression_codazzi_mainardi_equation}:
\[ \left\lbrace \begin{array}{l}
  \partial_{1} b_{12} - \partial_{2} b_{11} = \left( \Gamma_{11}^{1} b_{12} - \Gamma_{12}^{1} b_{11} \right) +  \left( \Gamma_{11}^{2} b_{22} - \Gamma_{12}^{2} b_{21} \right) \\
\\
 \partial_{1} b_{22} - \partial_{2} b_{21} = \left( \Gamma_{21}^{1} b_{12} - \Gamma_{22}^{1} b_{11} \right) + \left( \Gamma_{21}^{2} b_{22} - \Gamma_{22}^{2} b_{21} \right) . \\
\end{array} \right. \]
Hence, the two differential operators $A^{1}u_{i} := \partial_{1} b_{12} - \partial_{2} b_{11} $ and $A^{2} u_{i} := \partial_{1} b_{22} - \partial_{2} b_{21}$ are valued and uniformly bounded in $L^{\infty}(U)$, which is compactly embedded in $H^{-1}(U)$ (Rellich-Kondrachov Embedding Theorem), so we deduce that up to a subsequence, $(A^{1}u_{i})_{i \in \mathbb{N}} $ and $ (A^{2}u_{i})_{i \in \mathbb{N}}$ lies in a compact subset of $H^{-1}(U)$. Let us now have a look at the wave cone:
\[ \Lambda = \left\lbrace \left( \begin{array}{cc} \lambda_{11} & \lambda_{12} \\
\lambda_{21} & \lambda_{22} \\ \end{array} \right) \in \mathbb{R}^{2^{2}} ~\vert~ \exists \left( \begin{array}{c} \mu_{1} \\ \mu_{2} \\ \end{array} \right) \neq \left( \begin{array}{c} 0 \\ 0 \\ \end{array} \right), ~ \mu_{1} \lambda_{12} - \mu_{2} \lambda_{11} = 0 ~ \mathrm{and} ~ \mu_{1} \lambda_{22} - \mu_{2} \lambda_{21} = 0 \right\rbrace.  \] 

\begin{remark}
\label{remark_cone_rtrois}
The wave cone $\Lambda$ is the set of $(2\times 2)$-matrices with zero determinant.
\end{remark}

Consequently, if we want to apply Proposition \ref{prop_tartar} on a quadratic form in the $b_{pq}$, we get from Remark \ref{remark_cone_rtrois} that the determinant is one possibility. Indeed, if we set $F(u_{i}) = \mathrm{det}(u_{i})$, then $F$ is quadratic and $F(\lambda) = 0$ for any $\lambda \in \Lambda$. Since $(F(u_{i}))_{i \in \mathbb{N}}$ is uniformly bounded in $L^{\infty}(U)$, up to a subsequence, it is converging and applying Proposition \ref{prop_tartar}, the limit is $F(u)$. This also proves that $F(u)$ is the unique limit of any converging subsequence. Hence, the whole sequence is converging to $F(u)$ and we are now in position to prove the following result.

\begin{proposition}
\label{prop_continuite_gaussian_curvature}
Consider Assumption \ref{hypothese_continuite} and some continuous maps $j, j_{i}: \mathbb{R}^{3} \times \mathbb{S}^{2} \rightarrow \mathbb{R}$ such that $(j_{i})_{i \in \mathbb{N}}$ is uniformly bounded on $\overline{B} \times \mathbb{S}^{2}$ and diagonally converges to $j$ as in Definition \ref{definition_convergence_diagonale}. Then, we have (note that Remarks \ref{remarque_vector_fields} and \ref{remarque_lower_semicontinuity} also hold here): 
\[ \lim_{i \rightarrow + \infty} \int_{\partial \Omega_{i}} K \left( \mathbf{x} \right) j_{i} \left[ \mathbf{x},\mathbf{n}\left(\mathbf{x} \right) \right]dA\left(\mathbf{x} \right) = \int_{\partial \Omega} K \left( \mathbf{x} \right) j \left[ \mathbf{x},\mathbf{n}\left(\mathbf{x} \right) \right]dA\left(\mathbf{x} \right). \]
In particular, the genus is continuous: $\mathrm{genus}(\partial \Omega_{i}) \longrightarrow_{i \rightarrow + \infty} \mathrm{genus}(\partial \Omega)$.
\end{proposition}

\begin{proof}
As in the proof of Proposition \ref{prop_continuite_ordre_zero}, we can express the functional in the parametrizations associated with the partition of unity. Then, we have to check we can let $i \rightarrow + \infty$ in each integral. Note that $K$ is the determinant \eqref{expression_scalar_mean_curvature} of the Weingarten map \eqref{expression_weingarten_map} so we get from \eqref{expression_weingarten_map_local_def}:
\[ K \circ X^{k}_{i} = \mathrm{det}(h) =  \mathrm{det}(-g^{-1}b) = - \dfrac{\mathrm{det}(b_{pq})}{\mathrm{det}(g_{rs})}. \]
From the foregoing and the uniform convergence of $(g_{rs})$, we get that the sequences $(K \circ X_{i}^{k})_{i \in \mathbb{N}}$ converge $L^{\infty}$-weakly-star respectively to $ K \circ X^{k}$, whereas the remaining term in the integrand is $L^{1}$-strongly converging using the hypothesis and Lebesgue's Dominated Convergence Theorem. Hence, we can let $i \rightarrow + \infty$ and Proposition \ref{prop_continuite_gaussian_curvature} holds. Finally, concerning the genus, we apply the Gauss-Bonnet Theorem $\int_{\partial \Omega_{i}} K dA = 4 \pi (1- g_{i}) \longrightarrow_{i \rightarrow + \infty} \int_{\partial \Omega} K dA = 4 \pi ( 1 - g)$.
\end{proof}

We now establish the equivalent of Proposition \ref{prop_continuite_gaussian_curvature} in $\mathbb{R}^{n}$. First, instead of working with the coefficients $(b_{pq})$ of the second fundamental form \eqref{expression_seconde_forme_fondamentale}, we prefer to work with the ones $(h_{pq})$ representing the Weingarten map. We set $n > 3$, $U = D_{\tilde{r}}(\mathbf{x}_{k})$, and $u_{i} : \mathbf{x'} \in U  \mapsto (h_{pq}) \in \mathbb{R}^{(n-1)^{2}} $ defined by \eqref{expression_weingarten_map_local_def} in the local parametrizations $X_{i}^{k} : \mathbf{x'} \in U \mapsto (\mathbf{x'},\varphi_{i}^{k}(\mathbf{x'})) \in \partial \Omega_{i} $ introduced in the proof of Proposition \ref{prop_continuite_ordre_zero}. Then, we check that the hypothesis of Proposition \ref{prop_tartar} are satisfied. From Theorem \ref{thm_parametrisation_locale_i}, $(u_{i})_{i \in \mathbb{N}}$ weakly-star converges to $u$ in $L^{\infty}(U)$ and it is uniformly bounded so it is valued in a compact set. Using the Codazzi-Mainardi equations \eqref{expression_codazzi_mainardi_equation}, the differential operators:
\[ \partial_{q'} h_{pq} - \partial_{q} h_{pq'} = \sum_{m=1}^{n-1} \left( (\partial_{q'}g^{pm})b_{mq} - (\partial_{q} g^{pm}) b_{mq'} \right) + \sum_{m=1}^{n-1} g^{pm} \left( \partial_{q'} b_{mq} - \partial_{q}b_{mq'} \right), \]
are valued and uniformly bounded in $L^{\infty}(U)$, which is compactly embedded in $H^{-1}(U)$ (Rellich-Kondrachov Embedding Theorem), so up to a subsequence, they lies in a compact set of $H^{-1}(U)$. Finally, we introduce the wave cone of Proposition \ref{prop_tartar}:
\[ \Lambda = \left\lbrace \lambda \in \mathbb{R}^{(n-1)^{2}} ~\vert~ \exists \mu \neq 0_{(n-1)\times 1}, \forall (p,q,m) \in \lbrace 1,\ldots,n-1 \rbrace^{3}, \quad \mu_{m} \lambda_{pq} - \mu_{q} \lambda_{pm} = 0  \right\rbrace.  \] 

\begin{definition}
\label{definition_mineur}
A $p$th-order minor of a square $(n-1)^{2}$-matrix $M$ is the determinant of any $(p \times p)$-matrix $M[I,J]$ formed by the coefficients of $M$ corresponding to rows with index in $I$ and columns with index in $J$, where $I,J \subset \lbrace 1, \ldots, n-1 \rbrace$ have $p$ elements i.e. $\sharp I = \sharp J = p$.
\end{definition}

\begin{remark}
\label{remark_cone_rn}
The wave cone $\Lambda$ is the set of square $(n-1)^{2}$-matrices of rank zero or one. In particular, any minor of order two is zero for such matrices. 
\end{remark}

Consequently, Remark \ref{remark_cone_rn} combined with Proposition \ref{prop_tartar} tells us that continuous functionals are given by the ones whose expressions in the local parametrizations (cf. proof of Proposition \ref{prop_continuite_ordre_zero}) are linear in terms of the form $h_{pq} h_{p'q'} - h_{pq'} h_{p'q}$. However, such terms depend on the partition of unity and on the parametrizations i.e. on the chosen basis $(\partial_{1}X_{i}^{k}, \ldots, \partial_{n-1}X_{i}^{k})$ whereas the integrand of the functional cannot. We now give three applications for which it is the case.

\begin{proposition}
\label{prop_continuite_ordre_deux}
Consider Assumption \ref{hypothese_continuite} and some continuous maps $j, j_{i}: \mathbb{R}^{n} \times \mathbb{S}^{n-1} \rightarrow \mathbb{R}$ so that $(j_{i})_{i \in \mathbb{N}}$ is uniformly bounded on $\overline{B} \times \mathbb{S}^{n-1}$ and diagonally converges to $j$ in the sense of Definition \ref{definition_convergence_diagonale}. Then, introducing $H^{(2)} = \sum_{1 \leqslant p < q \leqslant n-1} \kappa_{p} \kappa_{q}$ defined in \eqref{expression_symmetric_polynomial_curvature}, we have: 
\[ \lim_{i \rightarrow + \infty} \int_{\partial \Omega_{i}} H^{(2)} \left( \mathbf{x} \right) j_{i} \left[ \mathbf{x},\mathbf{n}\left(\mathbf{x} \right) \right]dA\left(\mathbf{x} \right) = \int_{\partial \Omega} H^{(2)} \left( \mathbf{x} \right) j \left[ \mathbf{x},\mathbf{n}\left(\mathbf{x} \right) \right]dA\left(\mathbf{x} \right). \]
Note that Remarks \ref{remarque_vector_fields} and \ref{remarque_lower_semicontinuity} also hold for this functional.
\end{proposition}

\begin{proof}
First, using the notation of Definition \ref{definition_mineur}, note that the characteristic polynomial of $(h_{pq})$, which is the matrix \eqref{expression_weingarten_map_local_def} representing the Weingarten map \eqref{expression_weingarten_map} in the basis $(\partial_{1}X_{i}^{k}, \ldots \partial_{n-1} X_{i}^{k})$, can be expressed as:
\[ P(t) = \mathrm{det} \left( h - tI_{n-1} \right) = (-1)^{n} t^{n} + \sum_{m = 1}^{n-1} (-1)^{n-m} \left(  \sum_{\sharp I = m} \mathrm{det}(h[I,I]) \right) t^{n-m} , \]
but we can also represent the Weingarten map in the basis associated with the principal curvatures:
\[ P(t) = \prod_{m=1}^{n-1} \left( \left(\kappa_{m}\circ X_{i}^{k} \right) - t \right) = \sum_{m = 0}^{n-1} (-1)^{n-m}   \left( H^{(m)}\circ X_{i}^{k}  \right) t^{n-m}. \] 
Since each coefficients of the characteristic polynomial do not depend on the chosen basis, we get:
\begin{equation}
\label{expression_coeff_polynome_characteristique}
\forall m \in \lbrace 0, \ldots, n-1 \rbrace, \quad H^{(m)} \circ X_{i}^{k} = \sum_{\sharp I = m} \mathrm{det}(h[I,I])  .
\end{equation}
If we set $F(\lambda) = \sum_{\sharp I = 2} \mathrm{det}(\lambda[I,I]) $, then $F$ is quadratic and from Remark \ref{remark_cone_rn} we get $F(\lambda) = 0$ for any $\lambda \in \Lambda$.  Since $(F(u_{i}))_{i \in \mathbb{N}}$ is uniformly bounded in $L^{\infty}(U)$, up to a subsequence, it is converging and applying Proposition \ref{prop_tartar}, the limit is $F(u)$, unique limit of any converging subsequence so the whole sequence is converging to $F(u)$. Using \eqref{expression_coeff_polynome_characteristique}, we get that the sequences $(H^{(2)} \circ X_{i}^{k})_{i \in \mathbb{N}}$ converge $L^{\infty}$-weakly-star respectively to $ H^{(2)} \circ X^{k}$, whereas the remaining term in the integrand is $L^{1}$-strongly converging using the hypothesis and Lebesgue's Dominated Convergence Theorem. Hence, we can let $i \rightarrow + \infty$ and the functional is continuous. 
\end{proof}

\begin{corollary}
Considering Assumption \ref{hypothese_continuite}, a continuous map $j: \mathbb{R}^{n} \times \mathbb{S}^{n-1} \times \mathbb{R} \rightarrow \mathbb{R}$ convex in its last variable, and the (Frobenius) $L^{2}$-norm $\Vert D_{\mathbf{x}} \mathbf{n} \Vert_{2} = \sqrt{ \mathrm{trace}(D_{\mathbf{x}} \mathbf{n} \circ D_{\mathbf{x}} \mathbf{n}^{T})} =  (\sum_{m=1}^{n-1} \kappa_{m}^{2})^{\frac{1}{2}}$ of the Weingarten map \eqref{expression_weingarten_map}, we have: 
\[  \int_{\partial \Omega} j \left[ \mathbf{x},\mathbf{n}\left(\mathbf{x} \right), \Vert D_{\mathbf{x}} \mathbf{n} \Vert_{2}^{2} \right] dA\left(\mathbf{x} \right) \leqslant \liminf_{i \rightarrow + \infty} \int_{\partial \Omega_{i}} j \left[ \mathbf{x},\mathbf{n}\left(\mathbf{x} \right), \Vert D_{\mathbf{x}} \mathbf{n} \Vert_{2}^{2} \right] dA\left(\mathbf{x} \right). \]
In particular, the $p$th-power of the $L^{2}$-norm of the second fundamental form $\int \Vert \mathrm{\mathbf{II}} \Vert_{2}^{p} dA$, $p \geqslant 2$ is lower semi-continuous.
\end{corollary}

\begin{proof}
First, assume that $j$ is linear in its last argument. Note that the Frobenius norm $\Vert . \Vert_{2}$ does not depend on the chosen basis so we can consider the one associated with the principal curvatures, and we get $\Vert D \mathbf{n} \Vert_{2}^{2} = \sum_{m=1}^{n-1} \kappa_{m}^{2} = (\sum_{m=1}^{n-1} \kappa_{m})^{2} - \sum_{p \neq q} \kappa_{p} \kappa_{q} = H^{2} - 2 H^{(2)}$. Hence, there exists a continuous map $\tilde{j} : \mathbb{R}^{n} \times \mathbb{S}^{n-1} \rightarrow \mathbb{R}$ such that $ \int_{\partial \Omega_{i}} j [ \mathbf{x},\mathbf{n}(\mathbf{x} ), \Vert D_{\mathbf{x}} \mathbf{n} \Vert_{2}^{2} ] dA(\mathbf{x} ) $ is equal to:
\[   \int_{\partial \Omega_{i}} H^{2} \left( \mathbf{x} \right) \tilde{j} \left[ \mathbf{x}, \mathbf{n} \left( \mathbf{x} \right) \right] dA \left( \mathbf{x} \right)  - 2 \int_{\partial \Omega_{i}} H^{(2)} \left( \mathbf{x} \right) \tilde{j} \left[ \mathbf{x}, \mathbf{n} \left( \mathbf{x} \right) \right] dA \left( \mathbf{x} \right) .   \]
In the left term, the integrand is convex in $H$ so Corollary \ref{coro_continuite_ordre_un_convex} furnishes its lower semi-continuity. Concerning the right one, apply Proposition \ref{prop_continuite_ordre_deux} to get its continuity. Therefore, the functional is lower semi-continuous if $j$ is linear in its last variable. Then, we can apply the standard procedure \cite[Theorem 2.2.1]{Evans} described in Corollary \ref{coro_continuite_ordre_un_convex} to get the same result in the general case. Finally, $\Vert \mathrm{\mathbf{II}}(\mathbf{x}) \Vert_{2}^{2} = \Vert D_{\mathbf{x}} \mathbf{n} \Vert_{2}^{2}$ and if $p \geqslant 2$, $t \mapsto t^{\frac{p}{2}}$ is convex thus $\int \Vert \mathrm{\mathbf{II}} \Vert_{2}^{p}dA$ is lower semi-continuous.
\end{proof}

\begin{proposition}
Consider Assumption \ref{hypothese_continuite}, some continuous maps $j, j_{i} : \mathbb{R}^{n} \times \mathbb{S}^{n-1} \rightarrow \mathbb{R}$ such that $(j_{i})_{i \in \mathbb{N}}$ is uniformly bounded on $\overline{B} \times \mathbb{S}^{n-1}$ and diagonally converges to $j$ as in Definition \ref{definition_convergence_diagonale}, and some vector fields $\mathbf{V}_{i}$ and $ \mathbf{W}_{i}$ on $\partial \Omega_{i}$ uniformly bounded and diagonally converging to vector fields $\mathbf{V}$ and $\mathbf{W}$ on $\partial \Omega$ in the sense of Definition \ref{definition_convergence_diagonale_vector_fields}. Then, the following functional is continuous (note that Remarks \ref{remarque_vector_fields} and \ref{remarque_lower_semicontinuity} also hold here):
\[ J \left( \partial \Omega_{i} \right) = \int_{\partial \Omega_{i}} \left\langle D_{\mathbf{x}}\mathbf{n} \left[ \mathbf{V}_{i} \left( \mathbf{x} \right) \right] ~\vert~ D_{\mathbf{x}} \mathbf{n} \left[ \mathbf{W}_{i}  \left( \mathbf{x} \right) \right] - H \left( \mathbf{x} \right) \mathbf{W}_{i} \left( \mathbf{x} \right) \right\rangle j_{i} \left[ \mathbf{x} , \mathbf{n} \left( \mathbf{x} \right) \right] dA \left( \mathbf{x} \right) \underset{i \rightarrow + \infty}{\longrightarrow} J \left( \partial \Omega \right) . \]
\end{proposition}

\begin{proof}
Again, the idea is to check that the expression of the functional in the parametrization is linear in a term of the form $b_{pq}b_{p'q'}-b_{pq'}b_{p'q'}$. First, the linear term can be expressed as:
\[ \sum_{p,p',p" = 1}^{n-1} \sum_{q,q',q" = 1}^{n-1}  \left\langle \mathbf{V}_{i} \circ X_{i}^{k} ~\vert~ \partial_{q} X_{i}^{k} \right\rangle g^{pq}g^{p'q'} \left( b_{q'p}b_{p"p'} - b_{q'p'}b_{pp"} \right) g^{p"q"} \left\langle \mathbf{W}_{i} \circ X_{i}^{k} ~\vert~ \partial_{q"} X_{i}^{k} \right\rangle   \]
Note that until now, in Section \ref{section_continuite}, we never used the fact that $(g_{pq})$, $(g^{pq})$, $(b_{pq})$ or $(h_{pq})$ are symmetric matrices. Here, let us invert the two indices $b_{pp"} = b_{p"p}$ in the above expression. Then, $ b_{q'p}b_{p"p'} - b_{q'p'}b_{p"p}$ is $L^{\infty}$-weakly-star converging. Indeed, as we did for $(h_{pq})$, we can use the Codazzi-Mainardi equations \eqref{expression_codazzi_mainardi_equation} and Remark \ref{remark_cone_rn} to apply Proposition \ref{prop_tartar} on $(b_{pq})$. Finally, the hypothesis and the convergence results of Theorem \ref{thm_parametrisation_locale_i} gives the $L^{1}$-strong convergence of the remaining term so we can let $i \rightarrow + \infty$ in each integral and the functional is continuous. 
\end{proof}

Note that until now, in Section \ref{section_continuite_non_lineaire}, we only used the Codazzi-Mainardi equations \eqref{expression_codazzi_mainardi_equation}. We want here to use the Gauss equations \eqref{expression_gauss_equation} because from the foregoing, its right member is $L^{\infty}$-weakly-star converging. For this purpose, we need to introduce some concepts of Riemannian geometry which are beyond the scope of the article. Hence, we refer to \cite{Willmore} for precise definitions. Merely speaking, the Riemann curvature tensor $R$ of a Riemannian manifold measures the extend to which the first fundamental form is not locally isometric to a Euclidean space, i.e. the noncommutativity of the covariant derivative. In the basis $(\partial_{1}X, \ldots, \partial_{n-1} X)$, it has the following representation \cite[Section 2.6]{Willmore}:
\[ R^{k}_{jli}  = \sum_{m=1}^{n-1} g^{km}R_{mjli} =  \partial_{l} \Gamma_{ij}^{k} -  \partial_{j} \Gamma_{il}^{k}  + \sum_{m=1}^{n-1} \left( \Gamma^{m}_{ij} \Gamma_{ml}^{k} - \Gamma_{il}^{m} \Gamma_{mj}^{k} \right), \]
where the Christoffels symbols $\Gamma^{k}_{ij}$ were defined in \eqref{expression_christoffel_symbols}. Hence, the Gauss equations \eqref{expression_gauss_equation} state that in the local parametrization, the Riemann curvature tensor is given by:
\[ R^{k}_{jli} = \sum_{m=1}^{n-1} g^{km} \left( b_{ij} b_{ml}  - b_{il} b_{mj}  \right), \]
which is thus $L^{\infty}$-weakly-star converging, and so does the Ricci curvature tensor \cite[Section 3.3]{Willmore}  $Ric_{ij} = \sum_{k=1}^{n-1} R^{k}_{ikj}$ and the scalar curvature $ \mathfrak{R} = \sum_{i,j = 1}^{n-1} g^{ij} R_{ij} $. Hence, the following result holds.

\begin{proposition}
\label{prop_continuite_riemann_tensor}
Consider Assumption \ref{hypothese_continuite}, some continuous maps $j, j_{i} : \mathbb{R}^{n} \times \mathbb{S}^{n-1} \rightarrow \mathbb{R}$ such that $(j_{i})_{i \in \mathbb{N}}$ is uniformly bounded on $\overline{B} \times \mathbb{S}^{n-1}$ and diagonally converges to $j$ as in Definition \ref{definition_convergence_diagonale}, and some vector fields $\mathbf{T}_{i},\mathbf{U}_{i},\mathbf{V}_{i}, \mathbf{W}_{i}$ on $\partial \Omega_{i}$ uniformly bounded and diagonally converging to vector fields $\mathbf{T}, \mathbf{U}, \mathbf{V}, \mathbf{W}$ on $\partial \Omega$ in the sense of Definition \ref{definition_convergence_diagonale_vector_fields}. Then, the three following functionals are continuous (note that Remarks \ref{remarque_vector_fields} and \ref{remarque_lower_semicontinuity} also hold here):
\[ \left\lbrace \begin{array}{l}
\displaystyle{ J \left( \partial \Omega_{i} \right) = \int_{\partial \Omega_{i}} \left\langle R_{\mathbf{x}} \left[ \mathbf{T}_{i} \left( \mathbf{x} \right), \mathbf{U}_{i} \left( \mathbf{x} \right) \right] \mathbf{V}_{i} \left( \mathbf{x} \right) ~\vert~  \mathbf{W}_{i}  \left( \mathbf{x} \right) \right\rangle j_{i} \left[ \mathbf{x} , \mathbf{n} \left( \mathbf{x} \right) \right] dA \left( \mathbf{x} \right) \underset{i \rightarrow + \infty}{\longrightarrow} J \left( \partial \Omega \right) } \\
\\
\displaystyle{ J' \left( \partial \Omega_{i} \right) = \int_{\partial \Omega_{i}} Ric_{\mathbf{x}} \left[ \mathbf{V}_{i} \left( \mathbf{x} \right), \mathbf{W}_{i} \left( \mathbf{x} \right) \right] j_{i} \left[ \mathbf{x} , \mathbf{n} \left( \mathbf{x} \right) \right] dA \left( \mathbf{x} \right) \underset{i \rightarrow + \infty}{\longrightarrow} J' \left( \partial \Omega \right) } \\
\\
\displaystyle{ J "\left( \partial \Omega_{i} \right) = \int_{\partial \Omega_{i}} \mathfrak{R} \left( \mathbf{x} \right) ~ j_{i} \left[ \mathbf{x} , \mathbf{n} \left( \mathbf{x} \right) \right] dA \left( \mathbf{x} \right) \underset{i \rightarrow + \infty}{\longrightarrow} J" \left( \partial \Omega \right) } \\
\end{array} \right. \]
\end{proposition}

\begin{proof}
The proof is same than previous ones. Write the functional in the local parametrizations, and observe that it is a finite sum of integrals whose integrand is the product of a $L^{\infty}$-weakly-star converging term, while the other one is converging $L^{1}$-strongly so we can let $i \rightarrow + \infty$.
\end{proof}

Note that in the case $n = 3$, the scalar curvature $\mathfrak{R}$ is twice the Gaussian curvature $K = \kappa_{1} \kappa_{2}$. Hence, the continuity of the last functional above is the generalization of Proposition \ref{prop_continuite_gaussian_curvature} to $\mathbb{R}^{n}$, $n > 3$, which was the task of the subsection. We conclude by proving Theorem \ref{thm_continuite_rn}.

\begin{proof}[\textbf{Proof of Theorem \ref{thm_continuite_rn}}]
Using Proposition \ref{prop_tartar} and \eqref{expression_codazzi_mainardi_equation}, we showed how to get the $L^{\infty}$-weakly-star convergence of any $h[pp',qq']: = h_{pq}h_{p'q'} - h_{pq'}h_{p'q}$ from the one of $(h_{pq})$ defined in \eqref{expression_weingarten_map_local_def}. Now, we want to apply Proposition \ref{prop_tartar} to $(h[pp',qq'])$. For this purpose, we need to find differential operators which are valued and uniformly bounded in $L^{\infty}$. Using \eqref{expression_codazzi_mainardi_equation}, this is the case for:
\[ \begin{array}{rcl} \begin{array}{|ccc|}
\partial_{q} & h_{pq} & h_{p'q} \\
\partial_{q'} & h_{pq'} & h_{p'q'} \\
\partial_{q"} & h_{pq"} & h_{p'q"} \\
\end{array} &= & \partial_{q}h[pp',q'q"] - \partial_{q'} h[pp',qq"] + \partial_{q"} h[pp',qq'] \\ 
& = & \left( \partial_{q} h_{pq'} - \partial_{q'}h_{pq} \right) h_{p'q"} + \left( \partial_{q'} h_{p'q} - \partial_{q}h_{p'q'} \right) h_{pq"}   \\
& & \quad + \left( \partial_{q} h_{p'q"} - \partial_{q"}h_{p'q} \right) h_{pq'} + \left( \partial_{q"} h_{p'q'} - \partial_{q'}h_{p'q"} \right) h_{pq}  \\
& & \qquad + \left( \partial_{q"} h_{pq} - \partial_{q}h_{pq"} \right) h_{p'q'} + \left( \partial_{q'} h_{p'q"} - \partial_{q"}h_{pq'} \right) h_{p'q}. \\
\end{array} \]
Then, the wave cone associated with these differential operators is thus given by:
\[ \Lambda = \left\lbrace \lambda \in \mathbb{R}^{(n-1)^{4}} ~\vert~ \exists \mu \neq 0_{(n-1) \times 1}, \forall (p,p',q,q',q") \in \lbrace 0, \ldots, n-1 \rbrace, \quad \begin{array}{|ccc|}
\mu_{q} & \lambda_{pq} & \lambda_{p'q} \\
\mu_{q'} & \lambda_{pq'} & \lambda_{p'q'} \\
\mu_{q"} & \lambda_{pq"} & \lambda_{p'q"} \\
\end{array} = 0 \right\rbrace.  \]
As in Remark \ref{remark_cone_rtrois}, one can check that the wave cone is given by all $(n-1)^{2}$-matrices for which any minor of order three are zero in the sense of Definition \ref{definition_mineur}. Finally, combining \eqref{expression_coeff_polynome_characteristique} and Proposition \ref{prop_tartar}, we get that functionals linear in $H^{(3)}$ are continuous. This procedure can be done recursively similarly to $H^{(l)}$ for any $  l \geqslant 3 $ so Theorem \ref{thm_continuite_rn} holds.
\end{proof}

\subsection{Existence of a minimizer for various geometric functionals}
\label{section_existence_functional}
We are now in position to establish general existence results in the class $\mathcal{O}_{\varepsilon}(B)$. More precisely, we can minimize any functional (and constraints) constructed  from those given before in Section \ref{section_continuite}. Indeed, considering a minimizing sequence in $\mathcal{O}_{\varepsilon}(B)$, there exists a converging subsequence in the sense of Proposition \ref{prop_compacite_boule} (i)-(vi). Then, applying the appropriate continuity results, we can pass to the limit in the functional and the constraints to get the existence of a minimizer. 
\bigskip

In this section, we first give a proof of Theorem \ref{thm_existence_rtrois} and state/prove its generalization to $\mathbb{R}^{n}$. Then, we establish the existence for a very general model of vesicles. In particular, we prove that hold Theorems \ref{thm_existence_canham_helfrich}, \ref{thm_existence_helfrich}, and \ref{thm_existence_willmore}. We refer to Sections \ref{section_canham_helfrich}, \ref{section_helfrich}, and \ref{section_willmore} of the introduction for a detailed exposition on these three models. Finally, we present two more applications that show how to use other continuity results to get the existence of a minimizer in the class $\mathcal{O}_{\varepsilon}(B)$.

\begin{proof}[\textbf{Proof of Theorem \ref{thm_existence_rtrois}}]
Consider a minimizing sequence $(\Omega_{i})_{i \in \mathbb{N}} \subset \mathcal{O}_{\varepsilon}(B)$. From Proposition \ref{prop_compacite_boule}, up to a subsequence, it is converging to an open set $\Omega \in \mathcal{O}_{\varepsilon}(B)$. Since Assumption \ref{hypothese_continuite} holds, we can combine Propositions \ref{prop_continuite_ordre_zero}, \ref{prop_continuite_ordre_un}, and \ref{prop_continuite_gaussian_curvature} to let $i \rightarrow + \infty$ in the equalities of the form:
\[ \int_{\partial \Omega_{i}} g_{0}\left[\mathbf{x},\mathbf{n}\left( \mathbf{x} \right) \right] dA \left( \mathbf{x} \right) + \int_{\partial \Omega_{i}} H \left(\mathbf{x} \right) g_{1} \left[\mathbf{x},\mathbf{n}\left(\mathbf{x}\right) \right] dA \left( \mathbf{x} \right) + \int_{\partial \Omega_{i}}  K \left( \mathbf{x} \right) g_{2}\left[ \mathbf{x},\mathbf{n}\left(\mathbf{x}\right) \right] dA \left( \mathbf{x} \right) = \widetilde{C} . \]
Then, apply Proposition \ref{prop_continuite_ordre_zero}, Corollary \ref{coro_continuite_ordre_un_convex} and Remark \ref{remarque_lower_semicontinuity} on Proposition \ref{prop_continuite_gaussian_curvature}, to obtain the lower semi-continuity of the functional and that inequality contraints remain true as $i \rightarrow + \infty$. Therefore, $\Omega$ is a minimizer of the functional satisfying the constraints in the class $\mathcal{O}_{\varepsilon}(B)$. 
\end{proof}

\begin{theorem}
Let $\varepsilon > 0$ and $B \subset \mathbb{R}^{n}$ be a bounded open set containing a ball of radius $3 \varepsilon$ such that $\partial B$ has zero $n$-dimensional Lebesgue measure. Consider $(C, \widetilde{C}) \in \mathbb{R} \times \mathbb{R}$, some continuous maps $j_{0}, f_{0} , g_{0}, g_{l} : \mathbb{R}^{n} \times \mathbb{S}^{n-1} \rightarrow \mathbb{R}$, and some maps $j_{l}, f_{l}: \mathbb{R}^{n} \times \mathbb{S}^{n-1} \times \mathbb{R} \rightarrow \mathbb{R}$ which are continuous and convex in their last variable for any $l \in \lbrace 1, \ldots, n-1 \rbrace$. Then, the following problem has at least one solution (for the notation, we refer to Section \ref{section_geometrie}):
\[ \inf \int_{\partial \Omega} j_{0}\left[\mathbf{x},\mathbf{n}\left( \mathbf{x} \right) \right] dA \left( \mathbf{x} \right) +  \sum_{l = 1}^{n-1} \int_{\partial \Omega} j_{l} \left[ \mathbf{x},\mathbf{n}\left(\mathbf{x}\right),H^{(l)}\left( \mathbf{x} \right)\right] dA \left( \mathbf{x} \right),  \] 
where the infimum is taken among any $\Omega \in \mathcal{O}_{\varepsilon}(B)$ satisfying a finite number of constraints of the following form:
\[ \left\lbrace \begin{array}{l} 
 \displaystyle{ \int_{\partial \Omega} f_{0}\left[\mathbf{x},\mathbf{n}\left( \mathbf{x} \right) \right] dA \left( \mathbf{x} \right) + \sum_{l = 1}^{n-1} \int_{\partial \Omega} f_{l} \left[ \mathbf{x},\mathbf{n}\left(\mathbf{x}\right),H^{(l)}\left(\mathbf{x}\right)\right] dA \left( \mathbf{x}\right) \leqslant C } \\
\\
\displaystyle{\int_{\partial \Omega} g_{0}\left[\mathbf{x},\mathbf{n}\left( \mathbf{x} \right) \right] dA \left( \mathbf{x} \right) + \sum_{l = 1}^{n-1} \int_{\partial \Omega} H^{(l)} \left(\mathbf{x} \right) g_{l} \left[\mathbf{x},\mathbf{n}\left(\mathbf{x}\right) \right] dA \left( \mathbf{x} \right)  = \widetilde{C}. } \\
\end{array} \right. \]
\end{theorem}

\begin{proof}
Consider a minimizing sequence $(\Omega_{i})_{i \in \mathbb{N}} \subset \mathcal{O}_{\varepsilon}(B)$. From Proposition \ref{prop_compacite_boule}, up to a subsequence, it is converging to an open set $\Omega \in \mathcal{O}_{\varepsilon}(B)$. Since Assumption \ref{hypothese_continuite} holds, we can apply Theorem \ref{thm_continuite_rn} to let $i \rightarrow + \infty$ in the following equality:
\[ \int_{\partial \Omega_{i}} g_{0}\left[\mathbf{x},\mathbf{n}\left( \mathbf{x} \right) \right] dA \left( \mathbf{x} \right) + \sum_{l = 1}^{n-1} \int_{\partial \Omega_{i}} H^{(l)} \left(\mathbf{x} \right) g_{l} \left[\mathbf{x},\mathbf{n}\left(\mathbf{x}\right) \right] dA \left( \mathbf{x} \right) = \widetilde{C} . \]
Then, we can use again Theorem \ref{thm_continuite_rn} for any $l_{0} \in \lbrace 1, \ldots ,n-1 \rbrace$ by setting $j_{l_{0}} = g_{l_{0}}$ and $j_{l} = 0$ for any $l \neq l_{0}$ to obtain the continuity of any $\int H^{(l_{0})}(.) g_{l_{0}}[.,\mathbf{n}(.)]$ and Remark \ref{remarque_lower_semicontinuity} gives the lower semi-continuity of any $\int f_{l_{0}}[.,\mathbf{n}(.),H^{(l_{0})}(.)]$ and $\int j_{l_{0}}[.,\mathbf{n}(.),H^{(l_{0})}(.)]$. Hence, the functional is lower-semi-continuous and the inequality constraint remains true as $i \rightarrow + \infty$. Therefore, $\Omega$ is a minimizer of the functional satisfying the constraints. 
\end{proof}

\begin{proposition}
\label{prop_existence_vesicles}
Let $H_{0}, M_{0}, k_{G}, k_{m} \in \mathbb{R}$ and $\varepsilon, k_{b}, A_{0}, V_{0} > 0$ such that $A_{0}^{3} > 36 \pi V_{0}^{2}$. Then, the following problem modelling the equilibrium shapes of vesicles \cite[Section 2.5]{Seifert} has at least one solution (see Notation \ref{notation_geometrie_rtrois}):
\[ \inf_{\substack{\Omega \in \mathcal{O}_{\varepsilon}(\mathbb{R}^{n}) \\ A(\partial \Omega) = A_{0} \\ V(\Omega) = V_{0} }} \dfrac{k_{b}}{2} \int_{\partial \Omega} (H-H_{0})^{2}dA + k_{G} \int_{\partial \Omega} K dA + k_{m} \left( \int_{\partial \Omega} H dA - M_{0} \right)^{2} .  \]
\end{proposition}

\begin{proof}
Let us consider a minimizing sequence $(\Omega_{i})_{i \in \mathbb{N}} \subset \mathcal{O}_{\varepsilon}(\mathbb{R}^{n})$ of the functional satisfying the area and volume constraints. First, we need to find an open ball $B$ such that $(\Omega_{i})_{i \in \mathbb{N}} \subset \mathcal{O}_{\varepsilon}(B)$. This can be done if we can bound the diameter thanks to the functional and the area constraint. The first step is to control the Willmore energy \eqref{expression_willmore_energy}. Denoting by $J$ the functional, we have:
\[  \begin{array}{rcl}
\displaystyle{ \dfrac{k_{b}}{4}  \int_{\partial \Omega} H^{2} dA } & = & \displaystyle{  \dfrac{k_{b}}{4}  \int_{\partial \Omega} (H-H_{0}+H_{0})^{2} dA \quad \leqslant \quad  \dfrac{k_{b}}{2} \int_{\partial \Omega} (H-H_{0})^{2}dA  + \dfrac{k_{b}H_{0}^{2}}{2} A(\partial \Omega) }  \\
 & & \\
 & \leqslant & \displaystyle{  J(\partial \Omega) + \dfrac{k_{b}H_{0}^{2}}{2}  A(\partial \Omega) +  \vert k_{G} \vert ~ \begin{array}{|c|}
\displaystyle{ \int_{\partial \Omega} KdA} \\ 
 \end{array} + \vert k_{m} \vert  \left( \int_{\partial \Omega} HdA - M_{0} \right)^{2} } \\
 & & \\
 & \leqslant & \displaystyle{  J(\partial \Omega) +  \dfrac{k_{b}H_{0}^{2}}{2} A(\partial \Omega) +  \vert k_{G} \vert  \int_{\partial \Omega} \vert K \vert dA + 2 \vert k_{m} \vert  \left( \int_{\partial \Omega} HdA\right)^{2} + 2 \vert k_{m} \vert M_{0}^{2}  . } \\
\end{array}  \]
The second step is to use Point (iii) in Theorem \ref{thm_boule_equiv_cunun} and Remark \ref{remarque_borne_linfty}. Considering a point $\mathbf{x} \in \partial \Omega$ in which the Gauss map $\mathbf{n}$ is differentiable, and a unit eigenvector $\mathbf{e}_{l}$ associated with the eigenvalue $\kappa_{l}$ of the Weingarten map $D_{\mathbf{x}}\mathbf{n}$, we have:
\begin{equation}
\label{eqn_kappa}
\vert \kappa_{l} (\mathbf{x} ) \vert = \Vert \kappa_{l} (\mathbf{x} ) \mathbf{e}_{l} \Vert = \Vert D_{\mathbf{x}} \mathbf{n} (\mathbf{e}_{l}) \Vert \leqslant \Vert D_{\mathbf{x}} \mathbf{n} \Vert_{\mathcal{L}(T_{\mathbf{x}}\partial \Omega)} \Vert \mathbf{e}_{l} \Vert \leqslant \dfrac{1}{\varepsilon},  
\end{equation}
from which we deduce that $\max_{1 \leqslant l \leqslant n-1} \Vert \kappa_{l} \Vert_{L^{\infty}(\partial \Omega)} \leqslant \frac{1}{\varepsilon}$. Hence, we obtain:
\[ \dfrac{k_{b}}{4}  \int_{\partial \Omega} H^{2} dA \leqslant   J(\partial \Omega) +  \dfrac{k_{b}H_{0}^{2}}{2} A(\partial \Omega) +  \dfrac{\vert k_{G} \vert}{\varepsilon^{2}}  A(\partial \Omega) + \dfrac{8 \vert k_{m} \vert}{\varepsilon^{2}} A  \left(\partial \Omega\right)^{2}  + 2 \vert k_{m} \vert M_{0}^{2} .    \] 
The final step is to apply \cite[Lemma 1.1]{Simon} to get four positive constants $C_{0}, C_{1}, C_{2},C_{3}$ such that:
\[ \mathrm{diam}( \Omega ) \leqslant C_{0} J(\partial \Omega) A (\partial \Omega) + C_{1}  A (\partial \Omega)  + C_{2} A (\partial \Omega)^{2} + C_{3}  A (\partial \Omega)^{3}.   \]
Hence, we can bound uniformly the diameter of the $\Omega_{i}$ and there exists a ball $B \subset \mathbb{R}^{n}$ sufficiently large such that $(\Omega_{i})_{i \in \mathbb{N}} \subset \mathcal{O}_{\varepsilon}(B)$. From Proposition \ref{prop_compacite_boule}, up to a subsequence, it is converging to an $\Omega \in \mathcal{O}_{\varepsilon}(B)$. Then, we can apply:
\begin{itemize}
\item Corollary \ref{coro_continuite_ordre_un_convex} with $j(x,y,z)=\frac{k_{b}}{2}(z-H_{0})^{2}$ to get the lower semi-continuity of $\frac{k_{b}}{2}\int(H-H_{0})^{2}$;
\item Proposition \ref{prop_continuite_gaussian_curvature} with $j_{i} \equiv 1$ to obtain the continuity of $\kappa_{G} \int K$;
\item Proposition \ref{prop_continuite_ordre_un} with $j \equiv 1$ to have the continuity of $\int H dA$ thus the one of $k_{m}(\int H dA - M_{0} ) ^{2}$. 
\end{itemize}
The functional is lower semi-continuous and from Proposition \ref{prop_continuite_ordre_zero} with $j\equiv 1$ and $j(x,y)=\langle x ~\vert~ y \rangle$, the area and volume constraints are also continuous so let $i \rightarrow + \infty$ and $\Omega$ is a minimizer.
\end{proof}

\begin{proof}[\textbf{Proof of Theorem \ref{thm_existence_canham_helfrich}}]
It is the particular case $k_{m} = 0$ in Proposition \ref{prop_existence_vesicles}. This can be also deduced from Theorem \ref{thm_existence_rtrois}, it suffices to follow the method described in the next proof.
\end{proof}

\begin{proof}[\textbf{Proof of Theorem \ref{thm_existence_helfrich}}]
First, as in the proof of Proposition \ref{prop_existence_vesicles} , one can show that minimizing in $\mathcal{O}_{\varepsilon}(\mathbb{R}^{n})$ or in $\mathcal{O}_{\varepsilon}(B)$ is equivalent here. Then, apply Theorem \ref{thm_existence_rtrois} by setting $j_{0} = j_{2} \equiv 0$ and $j_{1}(x,y,z) = (z-H_{0})^{2}$ which is continuous and convex in $z$. The area and volume constraints can be expressed as in Proposition \ref{prop_continuite_ordre_zero} by setting $g_{1} = g_{2} \equiv 0$ and successively $g_{0}\equiv 1$, $g_{0}(x,y)=\langle x ~\vert~ y \rangle$. Using the Gauss-Bonnet Theorem, the genus constraint is included in $ \int K dA = 4 \pi(1-g) : = K_{0} $. Hence, Theorem \ref{thm_existence_rtrois} gives the existence of a minimizer satisfying the three constraints. Finally, we can apply \cite[Proposition 2.2.17]{HenrotPierre} to ensure that the compact minimizer is connected since it is the case for any minimizing sequence of compact sets. Hence, using again the Gauss-Bonnet Theorem, the minimizer has the right genus so Theorem \ref{thm_existence_helfrich} holds.  
\end{proof}

\begin{proof}[\textbf{Proof of Theorem \ref{thm_existence_willmore}}]
The proof is identical to the previous one. We just need to set $H_{0} = 0$ and add a fourth equality constraint of the form $g_{0} = g_{2} \equiv 0$, $g_{1} \equiv 1$.
\end{proof}

\begin{proposition}
Let $\varepsilon > 0$ and $B \subset \mathbb{R}^{4}$ be a bounded open set containing a ball of radius $3 \varepsilon$, and such that $\partial B$ has zero $4$-dimensional Lebesgue measure. We consider two bounded continuous vector fields of $\mathbb{R}^{4}$ denoted by $\mathbf{V}, \mathbf{W} : \mathbb{R}^{4} \rightarrow \mathbb{R}^{4}$ and a continuous map $j: \mathbb{R}^{4} \times \mathbb{S}^{3} \times \mathbb{R}$, which is convex in its last variable. Then, the following problem has at least one solution (for the notation we refer to Section \ref{section_geometrie} and Proposition \ref{prop_continuite_riemann_tensor} above):
\[ \inf \int_{\partial \Omega} j \left[ \mathbf{x}, \mathbf{n} \left( \mathbf{x} \right) , Ric_{\mathbf{x}} \left( \mathbf{V} \left( \mathbf{x} \right) \wedge \mathbf{n} \left(\mathbf{x}\right) ,  \mathbf{W} \left( \mathbf{x} \right) - \left\langle \mathbf{W} \left( \mathbf{x} \right) ~\vert~ \mathbf{n} \left( \mathbf{x} \right) \right\rangle \mathbf{n} \left( \mathbf{x} \right) \right) \right] dA \left( \mathbf{x} \right), \]
where the infimum is taken among all $\Omega \in \mathcal{O}_{\varepsilon}(B)$ satisfying the following constraint:
\[ \int_{\partial \Omega} \mathfrak{R} \left( \mathbf{x} \right) \left\langle \mathbf{V} \left( \mathbf{x} \right) ~\vert~ \mathbf{n} \left( \mathbf{x} \right)  \right\rangle dA \left( \mathbf{x} \right) = \int_{\partial \Omega}  H^{(2)} \left( \mathbf{x} \right) \left\langle \mathbf{W} \left( \mathbf{x} \right) ~\vert~ \mathbf{n} \left( \mathbf{x} \right) \right\rangle dA \left( \mathbf{x} \right) .  \]
\end{proposition}

\begin{proof}
Consider a minimizing sequence $(\Omega_{i})_{i \in \mathbb{N}} \subset \mathcal{O}_{\varepsilon}(B)$ of the functional satisfying the constraint. From Proposition \ref{prop_compacite_boule}, up to a subsequence, it is converging to a set $\Omega \in \mathcal{O}_{\varepsilon}(B)$. We define $V_{i} := \mathbf{V} \wedge \mathbf{n}_{\partial \Omega_{i}}$ and $W_{i} := \mathbf{W} - \langle \mathbf{W} ~\vert~ \mathbf{n}_{\partial \Omega_{i}} \rangle \mathbf{n}_{\partial \Omega_{i}}$ which are two continuous vector fields on $\partial \Omega_{i}$, uniformly bounded since $\mathbf{V}$ and $\mathbf{W}$ are. We now check the diagonal convergence. Choose any sequence of points $\mathbf{x}_{i} \in \partial \Omega_{i}$ converging to $\mathbf{x} \in \partial \Omega$. Using the partition of unity introduced in Proposition \ref{prop_continuite_ordre_zero}, we get that $\mathbf{x} \in \partial \Omega \cap \mathcal{C}_{\tilde{r},\varepsilon}(\mathbf{x}_{k})$ for some $k \in \lbrace 1, \ldots K \rbrace$. Hence, there exists $\mathbf{x'} \in D_{\tilde{r}}(\mathbf{x}_{k})$ such that $\mathbf{x} = (\mathbf{x'},\varphi^{k}(\mathbf{x'}))$. Since $(\mathbf{x}_{i})_{i \in \mathbb{N}}$ is converging to $\mathbf{x}$, for $i$ sufficiently large, we can write $\mathbf{x}_{i} = (\mathbf{x}_{i}', \varphi_{i}^{k}(\mathbf{x}_{i}'))$ with $\mathbf{x}_{i}' \in D_{\tilde{r}}(\mathbf{x}_{k})$. Hence, $\mathbf{x}_{i}' \rightarrow \mathbf{x'}$ and $\varphi_{i}^{k}(\mathbf{x}_{i}') \rightarrow \varphi^{k}(\mathbf{x'}) $, but we also have from the triangle inequality:
\[ \Vert \nabla \varphi_{i}^{k} (\mathbf{x}_{i}') - \nabla \varphi^{k}(\mathbf{x'}) \Vert \leqslant  \Vert \nabla \varphi_{i}^{k} - \nabla \varphi^{k} \Vert_{C^{0}(\overline{D_{\tilde{r}}}(\mathbf{x}_{k}))} + \Vert \nabla \varphi^{k}(\mathbf{x}_{i}') -  \nabla \varphi^{k}(\mathbf{x'}) \Vert . \] 
From \eqref{eqn_convergence} and the continuity of $\nabla \varphi^{k}$, we can let $i \rightarrow + \infty$ and the diagonal convergence of $(\nabla \varphi_{i}^{k})_{i \in \mathbb{N}}$ to $\nabla \varphi^{k}$ holds. Then, using \eqref{expression_normale}, $n_{\partial \Omega_{i}}$ is also diagonally converging to $n_{\partial \Omega}$, and so does $V_{i}$ and $W_{i}$. If $j$ is linear in its last variable, we can apply Proposition \ref{prop_continuite_riemann_tensor} to obtain the continuity of the functional, otherwise we can use Remark \ref{remarque_lower_semicontinuity} on the previous case to get the lower semi-continuity of the functional. Finally, apply Theorem \ref{thm_continuite_rn} with $j_{i}^{l} \equiv 0 $ if $l \neq 2$ and $j_{i}^{2} = \langle \mathbf{V} ~\vert~ \mathbf{n} \rangle$ to have the continuity of the left member of the constraint. The continuity of the right one comes from Proposition \ref{prop_continuite_riemann_tensor} on $J^{"}$ with $j_{i} = \langle \mathbf{W} ~\vert~ \mathbf{n} \rangle$. Hence, we can let $i \rightarrow + \infty$ in the constraint..
\end{proof}

\begin{proposition}
Let $\varepsilon, A_{0}, V_{0} > 0$ be such that $A_{0}^{3} > 36 \pi V_{0}^{2}$, and let $B \subset \mathbb{R}^{3}$ be a ball of radius at least $3 \varepsilon $. We consider a bounded vector field in $\mathbb{R}^{3}$ denoted by $\mathbf{V}: \mathbb{R}^{3} \rightarrow \mathbb{R}^{3} $ and a continuous map $j : \mathbb{R}^{3} \times \mathbb{R}^{2} \times \mathbb{R} \rightarrow \mathbb{R}$ which is convex in its last variable. Then, the following problem has at least one solution:
\[ \inf_{\substack{\Omega \in \mathcal{O}_{\varepsilon}(B) \\ A(\partial \Omega) = A_{0} \\ V(\Omega) = V_{0} }} \int_{\partial \Omega} j \left[ \mathbf{x}, \mathbf{n} \left(\mathbf{x}\right), \kappa_{\mathbf{v}} \left( \mathbf{x} \right) \right] dA \left( \mathbf{x} \right), \]
where $\kappa_{\mathbf{v}}$ is the normal curvature at $\mathbf{x}$ i.e. the curvature at $\mathbf{x}$ of the curve formed by the intersection of the surface $\partial \Omega $ with the plane spanned by $\mathbf{n}(\mathbf{x})$ and the vectore $ \mathbf{v} : = \mathbf{V}(\mathbf{x}) - \langle \mathbf{V}(\mathbf{x}) ~\vert~ \mathbf{n}(\mathbf{x}) \rangle \mathbf{n}(\mathbf{x})$.
\end{proposition}

\begin{proof}
First, \cite[Proposition 3.26, Remark 3.27]{MontielRos} gives $ \kappa_{\mathbf{v}} =  \kappa_{1} \vert \langle \mathbf{v} \vert \mathbf{e}_{1} \rangle \vert^{2} + \kappa_{2} \vert \langle \mathbf{v} \vert \mathbf{e}_{2} \rangle \vert^{2} = \mathrm{\mathbf{II}}(\mathbf{v},\mathbf{v}) $.  Then, as in the previous proof, we can show that $\mathbf{v}_{\partial \Omega_{i}}$ is diagonally converging to $\mathbf{v}_{\partial \Omega}$. Finally, if $j$ is linear in its last variable, we can apply Proposition \ref{prop_continuite_seconde_forme_fondamentale} to get the continuity, otherwise use Remark \ref{remarque_lower_semicontinuity} to get its lower semi-continuity. The area and volume constraints are continuous from Proposition \ref{prop_continuite_ordre_zero}. Hence, from Proposition \ref{prop_compacite_boule}, a minimizing sequence has a converging subsequence to an $\Omega$ and from the foregoing we can let $i \rightarrow + \infty$ in the functional and constraints so $\Omega$ is a minimizer.  
\end{proof}

\small
\bibliographystyle{plain}
\bibliography{biblio}

\end{document}